\documentclass[11pt]{amsart}
\usepackage{amscd,amsmath,amssymb,amsfonts,verbatim}
\usepackage[cmtip, all]{xy}
\usepackage{MnSymbol}
\usepackage{comment}

\newtheorem{thm}{Theorem}[section]
\newtheorem{prop}[thm]{Proposition}
\newtheorem{lem}[thm]{Lemma}
\newtheorem{cor}[thm]{Corollary}

\theoremstyle{definition}

\newtheorem{defn}[thm]{Definition}
\theoremstyle{remark}
\newtheorem{remk}[thm]{Remark}
\newtheorem{remks}[thm]{Remarks}

\newtheorem{exm}[thm]{Example}
\newtheorem{exms}[thm]{Examples}
\newtheorem{notat}[thm]{Notation}
\numberwithin{equation}{section}

{\hfill$\square$\end{defn}}
{\hfill$\square$\end{remk}}
{\hfill$\square$\end{remks}}
{\hfill$\square$\end{exm}}
{\hfill$\square$\end{exms}}
{\hfill$\square$\end{notat}}

\newcommand{\thmref}{Theorem~\ref}
\newcommand{\propref}{Proposition~\ref}
\newcommand{\corref}{Corollary~\ref}

\newcommand{\lemref}{Lemma~\ref}

\newcommand{\sA}{{\mathcal A}}
\newcommand{\sB}{{\mathcal B}}

\newcommand{\sE}{{\mathcal E}}
\newcommand{\sF}{{\mathcal F}}
\newcommand{\sG}{{\mathcal G}}
\newcommand{\sH}{{\mathcal H}}
\newcommand{\sI}{{\mathcal I}}

\newcommand{\sK}{{\mathcal K}}

\newcommand{\sM}{{\mathcal M}}

\newcommand{\sO}{{\mathcal O}}
\newcommand{\sP}{{\mathcal P}}
\newcommand{\sQ}{{\mathcal Q}}

\newcommand{\sV}{{\mathcal V}}

\newcommand{\sZ}{{\mathcal Z}}

\newcommand{\F}{{\mathbb F}}

\renewcommand{\H}{{\mathbb H}}

\renewcommand{\P}{{\mathbb P}}
\newcommand{\Q}{{\mathbb Q}}

\newcommand{\Z}{{\mathbb Z}}

\newcommand{\fm}{{\mathfrak m}}

\newcommand{\fp}{{\mathfrak p}}

\newcommand{\ff}{{\mathfrak f}}

\newcommand{\Irr}{{\rm Irr}}
\newcommand{\Ker}{{\rm Ker}}
\newcommand{\gr}{{\rm gr}}

\newcommand{\CH}{{\rm CH}}

\newcommand{\surj}{\twoheadrightarrow}
\newcommand{\inj}{\hookrightarrow}
\newcommand{\red}{{\rm red}}

\newcommand{\Hom}{{\rm Hom}}

\newcommand{\Spec}{{\rm Spec \,}}

\newcommand{\ms}{{\rm ms}}
\newcommand{\bk}{{\rm bk}}
\newcommand{\Tr}{{\rm Tr}}

\newcommand{\ab}{\rm ab}

\newcommand{\Gal}{{\rm Gal}}

\newcommand{\sHom}{{\mathcal{H}{om}}}

\newcommand{\Sch}{{\operatorname{\mathbf{Sch}}}}

\newcommand{\<}{\langle}
\renewcommand{\>}{\rangle}

\renewcommand{\max}{{\operatorname{\rm max}}}

\newcommand{\cyc}{{\operatorname{\rm cyc}}}

\newcommand{\et}{{\text{\'et}}}

\newcommand{\ds}{{/\kern-3pt/}}

\renewcommand{\log}{{\operatorname{log}}}

\newcommand{\Tor}{{\operatorname{Tor}}}

\newcommand{\un}{\underline}
\newcommand{\ov}{\overline}

\renewcommand{\dim}{\text{\rm dim}}

\newcommand{\tuborg}{\left\{\begin{array}{ll}}
\newcommand{\sluttuborg}{\end{array}\right.}

\newcommand{\zar}{{\rm zar}}
\newcommand{\nis}{{\rm nis}}

\newcommand{\dlog}{{\rm dlog}}

\newcommand{\abk}{{\rm abk}}

\newcommand{\adiv}{{\rm adiv}}

\newcommand{\wt}{\widetilde}
\newcommand{\wh}{\widehat}
\newcommand{\cont}{{\rm cont}}

\newcommand{\Fil}{{\rm fil}}

\newcommand{\etl}{{\acute{e}t}}

\newcounter{elno}

\newcounter{elno-abc}   

\newcounter{elno-abc-prime}

\begin{document}
\title{Ramified class field theory and duality over finite fields}
\author{Rahul Gupta, Amalendu Krishna}
\address{Fakult\"at f\"ur Mathematik, Universit\"at Regensburg, 
93040, Regensburg, Germany.}
\email{Rahul.Gupta@mathematik.uni-regensburg.de}
\address{Department of Mathematics, Indian Institute of Science,  
Bangalore, 560012, India.}
\email{amalenduk@iisc.ac.in}




\maketitle

\begin{quote}\emph{Abstract.} 
We prove a duality theorem for the $p$-adic {\'e}tale motivic 
cohomology of a variety $U$ which is the
complement of a divisor on a smooth projective variety over $\F_p$. 
This extends the duality theorems of Milne and Jannsen-Saito-Zhao.
The duality introduces a filtration on $H^1_{\etl}(U, {\Q}/{\Z})$.
We identify this filtration to the classically known 
Matsuda filtration when the
reduced part of the divisor is smooth. 
We prove a reciprocity theorem
for the idele class groups with modulus introduced by Kerz-Zhao and R{\"u}lling-Saito.
As an application, we derive the failure of Nisnevich descent for 
Chow groups with modulus. 
\end{quote}
\setcounter{tocdepth}{1}
\tableofcontents

\section{Introduction}\label{sec:Intro}
The objective of this paper is to study the duality and reciprocity theorems
for non-complete smooth varieties over finite fields and draw consequences.
Below we describe the contexts and statements of our main results.
 
\subsection{The duality theorem}\label{sec:D*}
Let $k$ be a finite field of characteristic $p$ 
and $X$ a smooth projective variety of dimension
$d$ over $k$. Let $W_m\Omega^r_{X,\log}$ be the logarithmic Hodge-Witt
sheaf on $X$, defined as the image of the dlog map 
from the Milnor $K$-theory sheaf $\sK^M_{r,X}$ to the $p$-typical
de Rham-Witt sheaf $W_m\Omega^r_X$ in the {\'e}tale topology.
Milne \cite{Milne-Zeta} proved that there is a perfect pairing of
cohomology groups 
\begin{equation}\label{eqn:Milne-Z}
H^i_\etl(X, W_m\Omega^r_{X,\log}) \times  
H^{d+1-i}_\etl(X, W_m\Omega^{d-r}_{X,\log})
\to {\Z}/{p^m}.
\end{equation}
By \cite[Theorem~8.4]{Geisser-Levine}, there is an isomorphism
$H^i_\etl(X, W_m\Omega^r_{X,\log}) \cong H^{i+r}_{\etl}(X, {\Z}/{p^m}(r))$,
where the latter is the $p$-adic {\'e}tale motivic cohomology 
due to Suslin-Voevodsky. Milne's theorem can thus be considered as
a perfect duality for the $p$-adic {\'e}tale motivic cohomology of
smooth projective varieties over $k$. The corresponding 
duality for the $\ell$-adic {\'e}tale  cohomology is the classical
Poincar{\'e} duality for {\'e}tale cohomology for a prime $\ell \neq p$.

If $U$ is a smooth quasi-projective variety over $k$ which is not 
complete, then there is a perfect duality (see \cite{Saito-89}) for the 
$\ell$-adic {\'e}tale cohomology in the form
\begin{equation}\label{eqn:Saito}
H^{i}_{\etl}(U, {\Z}/{\ell^m}(r)) \times H^{2d+1-i}_{c}(U, {\Z}/{\ell^m}(d-r))
\to {\Z}/{\ell^m},
\end{equation}
where $H^{i}_{c}(U,  {\Z}/{\ell^m}(j))$ is the $\ell$-adic {\'e}tale cohomology
of $U$ with compact support.

However, there is no known $p$-adic analogue of the {\'e}tale cohomology of $U$ 
with compact support such that a $p$-adic analogue of ~\eqref{eqn:Saito} could
hold. Construction of this duality is yet an open problem.
Recall that one of the applications (which is the interest of this paper)
of such a duality theorem is to the
study of the mod-$p$ {\'e}tale fundamental group of $U$, which is in general
a very complicated object. A duality theorem such as above would allow one to study this
group in terms of a more tractable {\'e}tale cohomology
of $U$ with compact support.

In \cite{JSZ}, Jannsen-Saito-Zhao proposed an approach in a special case
when $U$ is the complement of a divisor $D$ on a
smooth projective variety $X$ over $k$ such that $D_\red$ is a 
simple normal crossing divisor.
They constructed a relative version of the logarithmic Hodge-Witt sheaves
on $X$, denoted by $W_m\Omega^r_{X|D,\log}$. Using these sheaves, they showed
that there is a semi-perfect pairing (see Definition~\ref{defn:PF-lim-colim})
\begin{equation}\label{eqn:JSZ-main}
H^i_\etl(U, W_m\Omega^r_{U,\log}) \times  {\varprojlim}_n
H^{d+1-i}_\etl(X, W_m\Omega^{d-r}_{X|nD,\log}) \to {\Z}/{p^m}
\end{equation}
of topological abelian groups,
where the first group has discrete topology and the second has 
profinite topology. This pairing is perfect when $i = 1$ and $r = 0$.

In this paper, we propose a different approach to the $p$-adic
duality for $U$. This new approach has the advantage that it allows
$D$ to be an arbitrary divisor on $X$. This is possible by virtue of the
choice of the relative version of the logarithmic 
Hodge-Witt sheaves. Instead of using the Hyodo-Kato de Rham-Witt sheaves
with respect to a suitable log structure considered in \cite{JSZ},
we use a more ingenuous version of the relative logarithmic 
Hodge-Witt sheaf, which we denote by $W_m\Omega^r_{(X,D), \log}$. The latter is defined to be the 
kernel of the 
canonical surjection $W_m\Omega^r_{X,\log} \surj W_m\Omega^r_{D,\log}$.
Our main result on the $p$-adic duality for $U$ is roughly the
following. We refer to \thmref{thm:Duality-main} for the precise
statement.

\begin{thm}\label{thm:Main-1}
Let $X$ be a smooth projective variety of dimension $d$
over a finite field $k$ of characteristic $p$. 
Let $D \subset X$ be an effective Cartier divisor 
with complement $U$. Let $r, i \ge 0$ and $m \ge 1$ be integers.
Then there is a semi-perfect pairing of topological abelian groups
\[
H^i_{\etl}(U, W_m\Omega^r_{U, \log}) \times
{\varprojlim}_n
H^{d+1-i}_\etl(X, W_m\Omega^{d-r}_{(X,nD),\log}) \to {\Z}/{p^m}.
\]
This pairing is perfect if $D_\red$ is a simple normal crossing divisor,
$i = 1, r = 0$ and one of the conditions $\{d \neq 2, k \neq \F_2\}$
holds.
\end{thm}

We show that \thmref{thm:Main-1} recovers the duality theorem of
Jannsen-Saito-Zhao if  $D_\red$ is a simple normal crossing divisor.

\vskip .3cm

As a consequence of \thmref{thm:Main-1}, we obtain a filtration
$\{\Fil^{\etl}_{nD} H^1(K)\}_n$, where $K$ is the function field of $X$ and
$H^1(K)$ is a shorthand for $H^1_\etl(K, {\Q}/{\Z})$ (see \S~\ref{sec:Filt-et}).
We let $\Fil_{D} H^1(K)$ be the subgroup of $H^1(K)$ introduced in
\cite[Definition~7.12]{Gupta-Krishna-REC}. This coincides with
the filtration defined in \cite[Definition~2.9]{Kerz-Saito} by
\cite[Theorem~1.2]{Gupta-Krishna-BF}. $\Fil_{D} H^1(K)$ can be described as 
the subgroup
of continuous characters of the absolute Galois group of $K$ whose Artin 
conductors (see \cite[Definition~3.2.5]{Matsuda}) 
at the generic points of $D$ are bounded by the multiplicities of $D$ in
$D_\red$. This is a more intricate subgroup of $H^1(K)$ than
$\Fil^{\etl}_{D} H^1(K)$ because the latter can be described in terms of 
the simpler objects such as the cohomology of the relative logarithmic Hodge-Witt sheaves. 
On the other hand, $\Fil_{D} H^1(K)$ determines the ramification theory of
finite {\'e}tale coverings of $U$.
It is therefore desirable to know if and
when these two filtrations agree. Our next result is the following.

\begin{thm}\label{thm:Main-2}
Let $X$ be a smooth projective variety of dimension $d$
over a finite field $k$ of characteristic $p$. 
Let $D \subset X$ be an effective Cartier divisor 
such that $D_\red$ is regular. Assume that either $d \ne 2$ or $k \neq \F_2$.
Then 
\[
\Fil_D H^1(K) = \Fil^{\etl}_D H^1(K).
\]
\end{thm}

\vskip .3cm

\subsection{The reciprocity theorem}\label{sec:R**}
The purpose of reciprocity theorems in class field theory over a perfect 
field is to be able to represent the abelianized 
{\'e}tale fundamental groups of varieties 
over the field in terms of idele class groups which are often described
in terms of explicit sets of generators and relations.
Let us assume that the base field $k$ is finite of characteristic $p$.
In this case, such a reciprocity theorem for smooth projective varieties
over $k$ is due to Kato-Saito \cite{Kato-Saito-83} which describes the abelianized
{\'e}tale fundamental group in terms of the Chow group of 0-cycles.

For the more intricate case of a
smooth quasi-projective variety $U$ which is not complete, an approach
was introduced by Kato-Saito \cite{Kato-Saito} whose underlying idea
is to study the so called `{\'e}tale fundamental group of $X$ with modulus $D$',
where $D \subset X$ is a fixed divisor which is supported on the 
complement of $U$ in a normal compactification $U \subset X$.
This group characterizes the finite {\'e}tale coverings of $U$ whose 
ramification along $X\setminus U$ is bounded by $D$ in a certain sense.
There are various ways to make sense of this bound on the ramification,
and they give rise to several definitions of the 
{\'e}tale fundamental group with modulus.
It turns out that depending on what one wants to do, each of these has 
certain advantage over the others.

Kato and Saito were able to describe $\pi^{\ab}_1(U)$ in terms of the limit 
(over $D$) of the idele class groups with modulus 
$H^d_{\nis}(X, \sK^M_{d,(X,D)})$, where $d = \dim(X)$. 
In \cite[Theorems~1.1]{Gupta-Krishna-BF}, it was shown that
$H^d_{\nis}(X, \sK^M_{d,(X,D)})$ describes $\pi^{\adiv}_1(X,D)$ for
every $D$ if we let $\pi^{\adiv}_1(X,D)$ be the Pontryagin dual to
the Matsuda filtration $\Fil_D  H^1(K)$.
It was also shown in loc. cit. that $\pi^{\adiv}_1(X,D)$
coincides with the fundamental group
with modulus $\pi^{\ab}_1(X,D)$, introduced earlier by Deligne and Laumon
\cite{Laumon} if $X$ is smooth. The latter was shown to coincide (along the degree zero parts)
with the Chow group of 0-cycles with modulus $\CH_0(X|D)$ by Kerz-Saito 
\cite{Kerz-Saito} (when $p \neq 2$) and Binda-Krishna-Saito \cite{BKS}.

If we use Kato's Swan conductor instead of Matsuda's Artin conductor
to bound the ramification in terms of a divisor supported away from $U$, we are led to a
different notion of the abelianized {\'e}tale fundamental group
with modulus which we denote by $\pi^{\abk}_1(X,D)$.
This is defined as the Pontryagin dual to the subgroup
$\Fil^{\bk}_D H^1(K)$. The latter is 
the subgroup
of continuous characters of the absolute Galois group of $K$ whose Swan 
conductors (defined by Kato \cite{Kato-89}) 
at the generic points of $D$ are bounded by the multiplicities of $D$ in
$D_\red$. 

One can now ask if $\pi^{\abk}_1(X,D)$ could be described by
an idele class group with modulus, similar to the $K$-theoretic 
idele class group of Kato-Saito and the cycle-theoretic idele class group of
Kerz-Saito. Our next result solves this problem.

Let $\wh{\sK}^M_{r,X}$ be the improved Milnor $K$-theory sheaf of Gabber and
Kerz \cite{Kerz-10}.
Let $\wh{\sK}^M_{r,X|D}$ be the relative Milnor $K$-theory sheaf, defined 
locally as the image of the map
$\sK^M_{1, (X,D)} \otimes_{\Z} j_* \sO^{\times}_U \otimes_{\Z} \cdots  
\otimes_{\Z} j_* \sO^{\times}_U \to \wh{\sK}^M_{r,X}$.
We refer to \lemref{lem:RS-GK} for the proof that this map is defined.
This sheaf was considered by R{\"u}lling-Saito \cite{Rulling-Saito} when
$X$ is smooth and $D_\red$ is a simple normal crossing divisor.
There are
degree maps $\deg \colon H^d_\nis(X, \wh{\sK}^M_{d,X|D}) \to \Z$ 
(see \S~\ref{sec:Degree}) and 
$\deg' \colon \pi^{\abk}_1(X,D) \to \wh{\Z}$. We let
$H^d_\nis(X, \wh{\sK}^M_{d,X|D})_0 = \Ker(\deg)$ and 
$\pi^{\abk}_1(X,D)_0 = \Ker(\deg')$.

\begin{thm}\label{thm:Main-3}
Let $X$ be a normal projective variety of dimension $d$
over a finite field. 
Let $D \subset X$ be an effective Cartier divisor whose complement is regular.
Then there is a continuous reciprocity homomorphism
\[
\rho'_{X|D} \colon H^d_\nis(X, \wh{\sK}^M_{d,X|D}) \to \pi_1^{\abk}(X, D)
\]
with dense image such that the induced map
\[
\rho'_{X|D} \colon H^d_\nis(X, \wh{\sK}^M_{d,X|D})_0 \xrightarrow{\cong} 
\pi_1^{\abk}(X, D)_0
\]
is an isomorphism of finite groups.
\end{thm}

\vskip .3cm

Let $H^i_\sM(X|D, \Z(j))$ be the motivic cohomology with modulus. 
This is defined as the Nisnevich hypercohomology 
of the sheafified cycle complex with modulus
$z^j(X|D, 2j- \bullet)$, introduced in \cite{Binda-Saito}. 
Using \thmref{thm:Main-3} and 
\cite[Theorem~1]{Rulling-Saito}, we obtain the following.

\begin{cor}\label{cor:Main-4}
Assume in \thmref{thm:Main-3} that $X$ is regular and
$D_\red$ is a simple normal crossing 
divisor. Then there is an isomorphism of finite groups
\[
cyc'_{X|D} \colon H^{2d}_\sM(X|D, \Z(d))_0 \xrightarrow{\cong} 
\pi_1^{\abk}(X, D)_0.
\]
\end{cor}

This result can be viewed as the cycle-theoretic description of
the $\pi_1^{\abk}(X, D)$, analogous a similar result for $\pi_1^{\adiv}(X,D)$, proven in
\cite{Kerz-Saito}, \cite{BKS} and \cite{Gupta-Krishna-BF}.

\vskip .3cm

\subsection{Failure of Nisnevich descent for Chow group with modulus}
\label{sec:Chow-nis}
Recall that the classical higher Chow groups of smooth varieties
satisfy the Nisnevich descent in the sense that the canonical map
$\CH^i(X,j) \to H^{2i-j}_\sM(X, \Z(i))$ is an isomorphism for $i, j \ge 0$,
where the latter is the Nisnevich hypercohomology of the sheafified
(in Nisnevich topology) Bloch's cycle complex. 
However, this question for the higher Chow groups with modulus is 
yet an open problem. Some cases of this were verified by
R{\"u}lling-Saito \cite[Theorem~3]{Rulling-Saito}.
As an application of \corref{cor:Main-4}, we prove the following 
result which provides a counterexample to the Nisnevich descent for the
Chow groups with modulus. This was
one of our motivations for studying the reciprocity for
$\pi_1^{\abk}(X, D)$.

\begin{thm}\label{thm:Main-5}
Let $X$ be a smooth projective surface
over a finite field. 
Let $D \subset X$ be an effective Cartier divisor such that
$D_\red$ is a simple normal crossing divisor.
Then the canonical map
\[
\CH_0(X|D) \to H^{4}_\sM(X, \Z(2))
\]
is not always an isomorphism.
\end{thm}

Using \thmref{thm:Main-5}, one can show that the Nisnevich descent for the Chow groups with
modulus fails over infinite fields too.

Recall that the map in \thmref{thm:Main-5} is known to be an isomorphism if $X$ is a curve.
This is related to the fact that for a Henselian discrete valuation field
$K$ with perfect residue field, the Swan conductor and the Artin conductor
agree for the characters of the absolute Galois group of $K$.
In this context, we also remark that one could attempt to define an analogue
of the Deligne-Laumon fundamental group with modulus $\pi^{\ab}_1(X,D)$
by using the Brylinski-Kato filtration along all integral curves in $X$ not
contained in $D$. However, the resulting fundamental group will coincide with
$\pi^{\adiv}_1(X,D)$ and not yield anything new for the same reason as above.

\vskip .3cm

\subsection{Overview of proofs}\label{sec:Overview}
We give a brief overview of our proofs.
The main idea behind the proof of our duality theorem is the observation that
the naive relative logarithmic Hodge-Witt sheaves are isomorphic to the
naive relative mod-$p$ Milnor $K$-theory sheaves (in the sense of
Kato-Saito) in the pro-setting. The hope that the  
{\'e}tale cohomology of relative mod-$p$ Milnor $K$-theory sheaves are the 
correct objects to use for duality originated from our results in previous 
papers which showed that the Nisnevich cohomology of the
Kato-Saito relative Milnor $K$-theory yields a reciprocity isomorphism
without any condition on the divisor.

To implement  the above ideas, we need to prove many results about 
the relative logarithmic Hodge-Witt sheaves and their relation with
various versions of relative Milnor $K$-theory. One of these is a 
pro-isomorphism between the relative logarithmic Hodge-Witt sheaves
and the twisted de Rham-Witt sheaves. The latter are easier objects
to work with because duality for them follows from the 
Grothendieck-Serre coherent duality.
The next step is to construct higher Cartier operators on the
twisted de Rham-Witt sheaves. This allows us to construct the relative
version of the two-term complexes considered by Milne
\cite{Milne-Zeta}.
The proof of the duality theorem is then reduced to coherent duality
using an induction procedure.  

To prove \thmref{thm:Main-2}, we use our duality theorem and the reciprocity
theorem of \cite{Gupta-Krishna-BF} to reduce it to proving an
independent statement that the top Nisnevich cohomology of the
relative Milnor $K$-theory sheaf coincides with the corresponding
{\'e}tale cohomology if the reduced divisor is regular
(see \thmref{thm:Comparison-3}).

To prove \thmref{thm:Main-3}, we follow the strategy we used 
in the proof of a similar result for $\pi^{\adiv}_1(X,D)$ in
\cite{Gupta-Krishna-BF}. But there are some new ingredients to be used.
The first key ingredient is a result of Kerz-Zhao
\cite{Kerz-Zhau} which gives an idelic presentation of
$H^d_\nis(X, \wh{\sK}^M_{d,X|D})$.
The second key result is a theorem of Kato \cite{Kato-89}
which gives a criterion for the
characters of the absolute Galois group of a Henselian discrete valuation
field to annihilate various subgroups of the Milnor $K$-theory of the
field under a suitable pairing (see \thmref{thm:Fil-main}).
What remains after this is to prove \propref{prop:Fil-SH-BK}.

We prove various results about relative Milnor $K$-theory in
sections ~\ref{sec:Milnor-K}, ~\ref{sec:KZ} and ~\ref{sec:Comp**}.
We study the relative logarithmic Hodge-Witt sheaves and prove some
results about their cohomology in \S~\ref{sec:Hodge-Witt}. 
We introduce the Cartier operators on the twisted de Rham-Witt sheaves
in \S~\ref{sec:Cartier}. We construct the pairing for the duality theorem
and prove its perfectness in sections ~\ref{sec:DT} and ~\ref{sec:Perfectness}.
The reciprocity theorem is proven in \S~\ref{sec:REC*} and the
failure of Nisnevich descent for the Chow groups with
modulus is shown in \S~\ref{sec:CE}.

\vskip .3cm

\subsection{Notation}\label{sec:Notn}
We shall work over a field $k$ of characteristic $p > 0$ throughout this
paper. We let $\Sch_k$ denote the category of separated and
essentially of finite type schemes over $k$. The product $X \times_{\Spec(k)} Y$
in $\Sch_k$ will be written as $X \times Y$. We let
$X^{(q)}$ (resp. $X_{(q)}$) denote the set of points on $X$ having codimension
(resp. dimension) $q$.
We let $\Sch_{k/\zar}$ (resp. $\Sch_{k/\nis}$, resp. $\Sch_{k/ \etl})$
denote the Zariski (resp. Nisnevich, resp. {\'e}tale) site of $\Sch_{k}$.
We let $\epsilon \colon
\Sch_{k/ \et} \to \Sch_{k/ \nis}$ denote the canonical morphism of sites.
If $\sF$ is a sheaf on $\Sch_{k/\nis}$, 
we shall denote $\epsilon^* \sF$ also by $\sF$ as long as the usage of
the {\'e}tale topology is clear in a context.
For $X \in \Sch_k$, we shall let $\psi \colon X \to X$ denote the
absolute Frobenius morphism.

For an abelian group $A$, we shall write $\Tor^1_{\Z}(A, {\Z}/n)$ as
$_nA$ and $A/{nA}$ as $A/n$.
The tensor product $A \otimes_{\Z} B$ will be written as $A \otimes B$.
We shall let $A\{p'\}$ denote the
subgroup of elements of $A$ which are torsion of order prime to $p$.
We let $A\{p\}$ denote the subgroup of elements of $A$ which are torsion of 
order some power of $p$.

\section{Relative Milnor $K$-theory}\label{sec:Milnor-K}
In this section, we recall several versions of relative Milnor $K$-theory
sheaves and establish some relations among their cohomology.

For a commutative ring $A$, the Milnor $K$-group $K^M_r(A)$ 
(as defined by Kato \cite{Kato-86}) is
the $r$-th graded piece of the graded ring $K^M_*(A)$.
The latter is the quotient of the tensor algebra $T_*(A^{\times})$ by the 
two-sided graded ideal generated by the homogeneous elements
$a_1 \otimes \cdots \otimes a_r$ such that $r \ge 2$, and $a_i + a_j = 1$ for 
some $1 \le i \neq j \le r$.
The residue class of $a_1 \otimes \cdots \otimes a_r \in T_r(A^{\times})$
in $K^M_r(A)$ is denoted by the Milnor symbol $\un{a} = \{a_1,  \ldots , a_r\}$.
Given an ideal $I \subset A$, the relative Milnor $K$-theory $K^M_*(A, I)$ is
defined as the kernel of the restriction map
$K^M_*(A) \to K^M_*(A/I)$. 

Let $\wt{K}^M_r(A)$ denote the $r$-th graded piece of the graded ring 
$\wt{K}^M_*(A)$, where the latter is the quotient of the tensor algebra 
$T_*(A^{\times})$ by the two-sided graded ideal generated by the homogeneous 
elements $a_1 \otimes a_2$ such that $a_1 + a_2 = 1$.
We let $\wt{K}^M_*(A, I)$ be the kernel of the restriction map
$\wt{K}^M_*(A) \to \wt{K}^M_*(A/I)$.
It is clear that there is a natural surjection $\wt{K}^M_*(A) \surj
{K}^M_*(A)$. This is an isomorphism if $A$ is a local ring with infinite
residue field (see \cite[Proposition~2]{Kerz-10}). The following says that
a similar thing holds also for the relative $K$-theory.

\begin{lem}\label{lem:BT-KS}
Let $A$ be a local ring and let $I \subset A$ be an
ideal. Then the following hold.
\begin{enumerate}
\item
$K^M_r(A,I)$ and $\wt{K}^M_r(A,I)$ are generated by the Milnor symbols 
$\{a_1, \ldots , a_r\}$ 
such that $a_i \in K^M_1(A,I)$ for some $1 \le i \le r$.
\item
The canonical map $\wt{K}^M_r(A,I) \to K^M_r(A,I)$ is surjective.
This is an isomorphism if $A$ has infinite residue field.
\end{enumerate}
\end{lem}
\begin{proof}
The assertion (1) for $K^M_r(A,I)$ is \cite[Lemma~1.3.1]{Kato-Saito} and the 
proof of (1) for $\wt{K}^M_r(A,I)$ is completely identical to that of
$K^M_r(A,I)$.
The second part of (2) follows from the corresponding result in the 
non-relative case mentioned above. To prove the first part of (2), 
we fix an integer $r \ge 1$.
Let $\un{a} = \{a_1, \ldots , a_r\} \in \wt{K}^M_r(A)$ be such that $a_i \in 
\wt{K}^M_1(A,I) = K^M_1(A,I)$ for some $1 \le i \le r$. It is then clear that 
$\un{a} \in \wt{K}^M_r(A,I)$. Using this observation, our assertion follows
directly from item (1) for $K^M_r(A,I)$.
\end{proof}
 
For a ring $A$ as above, we let $\wh{K}^M_*(A)$ denote the improved
Milnor $K$-theory defined independently by Gabber and Kerz \cite{Kerz-10}.
For an ideal $I \subset A$, we define $\wh{K}^M_*(A,I)$ to be the kernel
of the map $\wh{K}^M_*(A) \to \wh{K}^M_*(A/I)$. 
Let $K_*(A)$ denote the Quillen $K$-theory of $A$.
We state the following facts as a lemma and refer to \cite{Kerz-10}
for their source.

\begin{lem}\label{lem:Milnor-Kerz}
There are natural maps
\begin{equation}\label{eqn:Milnor-K-0}
K^M_*(A) \xleftarrow{\alpha_A} \wt{K}^M_*(A) \xrightarrow{\beta_A} 
\wh{K}^M_*(A) \xrightarrow{\gamma_A} K_*(A),
\end{equation}
where $\alpha_A$ is always surjective and $\beta_A$ is surjective
if $A$ is local. These two maps are isomorphisms if $A$ is either a field
or a local ring with infinite residue field.
\end{lem}

Let $X$ be a Noetherian scheme and $\iota \colon D \inj X$ a 
closed immersion.
We shall say that $(X,D)$ is a modulus pair if $D$ is an effective
Cartier divisor on $X$. This Cartier divisor may be empty.
If $\sP$ is a property of schemes,
we shall say that $(X,D)$ satisfies $\sP$ if $X$ does so.
We shall say that $(X,D)$ has dimension $d$ if $X$ has Krull
dimension $d$.

Given a Noetherian scheme $X$ and an integer $r \ge 1$, 
let $\sK^M_{r,X}$ denote the sheaf on $X_\nis$
whose stalk at a point $x \in X$ is $K^M_{r}(\sO^h_{X,x})$. We let
$\sK^M_{r, (X,D)}$ be the kernel of the restriction map
$\sK^M_{r,X} \surj \sK^M_{r,D} := \iota_* \sK^M_{r,D}$. We shall usually refer to
$\sK^M_{r, (X,D)}$ as the Kato-Saito relative Milnor $K$-sheaves.
We define the sheaves $\wt{\sK}^M_{r, (X,D)}$ and $\wh{\sK}^M_{r, (X,D)}$
in an analogous way. We let $\sK_{r,X}$ be the Quillen $K$-theory sheaf on
$X_\nis$.

\begin{lem}\label{lem:Milnor-top-coh}
Let $X$ be a reduced Noetherian
scheme of Krull dimension $d$ and let $D \subset X$
be a nowhere dense closed subscheme. Then the canonical map
\[
H^d_\nis(X,  {\wt{\sK}^M_{r,(X,D)}}/n) \to
H^d_\nis(X,  {\sK^M_{r,(X,D)}}/n) 
\]
is an isomorphism for every integer $n \ge 0$.
\end{lem}
\begin{proof}
We look at the commutative diagram
\begin{equation}\label{eqn:Milnor-top-coh-0}
\xymatrix@C.8pc{
\wt{\sK}^M_{r,(X,D)} \ar@{^{(}->}[r] \ar[d] & \wt{\sK}^M_{r,X} \ar@{->>}[d] \\
\sK^M_{r,(X,D)} \ar@{^{(}->}[r] & \sK^M_{r,X}.}
\end{equation}

Lemma~\ref{lem:BT-KS} says that the left vertical arrow is surjective.
We have seen above that the kernel of the right vertical arrow is
supported on a nowhere dense closed subscheme of $X$ (this uses reducedness of $X$).
Hence, the same holds for the left vertical arrow.
But this easily implies the lemma when $n = 0$.

For $n \ge 1$, we use the commutative diagram
\[
\xymatrix@C.8pc{
H^d_\nis(X,  {\wt{\sK}^M_{r,(X,D)}}) \ar[r]^-{n} \ar[d] & 
H^d_\nis(X,  {\wt{\sK}^M_{r,(X,D)}}) \ar[r] \ar[d] & 
H^d_\nis(X,  {\wt{\sK}^M_{r,(X,D)}}/n) \ar[r] \ar[d] & 0 \\
H^d_\nis(X,  {\sK^M_{r,(X,D)}}) \ar[r]^-{n} &
H^d_\nis(X,  {\sK^M_{r,(X,D)}}) \ar[r] & H^d_\nis(X,  {\sK^M_{r,(X,D)}}/n)  
\ar[r] & 0.}
\]
It is easily seen that the two rows are exact. The $n = 0$ case shows that
the left and the middle vertical arrows are isomorphisms.
It follows that the right vertical arrow is also an isomorphism.
\end{proof}

For $D \subset X$ as above and
$n \ge 1$, let 
\begin{equation}\label{eqn:Milnor-mod-n}
\ov{{\wt{\sK}^M_{r,(X,D)}}/n}
:= {\wt{\sK}^M_{r,(X,D)}}/{(\wt{\sK}^M_{r,(X,D)} \cap n \wt{\sK}^M_{r,X})} =
{\rm Image}(\wt{\sK}^M_{r,(X,D)} \to {\wt{\sK}^M_{r,X}}/n).
\end{equation}
Since the map $n\wt{\sK}^M_{r,X} \to 
n \wt{\sK}^M_{r,D}$ is surjective, it easily follows that there are
exact sequences
\begin{equation}\label{eqn:Rel-Milnor}
0 \to \ov{{\wt{\sK}^M_{r,(X,D)}}/n} \to
{\wt{\sK}^M_{r,X}}/n \to  {\wt{\sK}^M_{r,D}}/n \to 0;
\end{equation}
\[
_{n}{\wt{\sK}^M_{r,D}} \to {\wt{\sK}^M_{r,(X,D)}}/n \to 
\ov{{\wt{\sK}^M_{r,(X,D)}}/n} \to 0.
\]
We let $\ov{{\wh{\sK}^M_{r,(X,D)}}/n}
:= {\wh{\sK}^M_{r,(X,D)}}/{(\wh{\sK}^M_{r,(X,D)} \cap n \wh{\sK}^M_{r,X})}$. 
It is then clear that the two exact sequences of ~\eqref{eqn:Rel-Milnor}
hold for the improved Milnor $K$-sheaves as well.

The following is immediate from ~\eqref{eqn:Rel-Milnor}.

\begin{lem}\label{lem:Milnor-d}
Let $X$ be a Noetherian scheme of Krull dimension $d$ and let $D \subset X$
be a nowhere dense closed subscheme. Then the canonical map
\[
H^d_\nis(X,  {\wt{\sK}^M_{r,(X,D)}}/n) \to 
H^d_\nis(X, \ov{{\wt{\sK}^M_{r,(X,D)}}/n})
\]
is an isomorphism for all $n \ge 1$.
\end{lem}

To proceed further, we need the following result about the
sheaf cohomology. Given a field $k$ and a topology
$\tau$ on $\Spec(k)$, let $cd_\tau(k)$ denote the cohomological
dimension of $k$ for torsion $\tau$-sheaves.

\begin{lem}\label{lem:Sheaf-0}
Let $k$ be a field and $X \in \Sch_k$. Let $\sF$ be a torsion sheaf on $X_\nis$
such that $\sF_x = 0$ for every $x \in X_{(q)}$ with $q > 0$. Then 
$H^q_{\tau}(X, \sF) = 0$ for $q > cd_\tau(k)$ and $\tau \in \{\nis, \etl\}$.
In particular, given an exact sequence of Nisnevich sheaves
\[
0 \to \sF \to \sF' \to \sF'' \to 0,
\]
the induced map $H^q_{\tau}(X, \sF') \to H^q_{\tau}(X, \sF'')$ is an isomorphism
for $q > cd_\tau(k)$ and $\tau \in \{\nis, \etl\}$.
\end{lem}
\begin{proof}
We fix $\tau \in \{\nis, \etl\}$.
Let $J$ be the set of finite subsets of $X_{(0)}$. 
Given any $\alpha \in J$, let $S_\alpha \subset X$ be the
0-dimensional reduced closed subscheme defined by $\alpha$ and let
$U_\alpha = X \setminus S_\alpha$. Let $\iota_\alpha \colon S_\alpha \inj X_\alpha$
and $j_\alpha \colon U_\alpha \inj X$ be the inclusions.    
As $J$ is cofiltered 
with respect to inclusion, there is a cofiltered system of short exact
sequences of $\tau$-sheaves
\begin{equation}\label{eqn:Sheaf-0-0}
0 \to (j_\alpha)_!(\sF|_{U_\alpha}) \to \sF \to (\iota_\alpha)_*(\sF|_{S_\alpha}) \to 0.
\end{equation}
It is easily seen from our assumption that 
${\underset{\alpha \in J}\varinjlim} \
(j_\alpha)_!(\sF|_{U_\alpha}) = 0$ so that the map
$\sF \to {\underset{\alpha \in J}\varinjlim} \ (\iota_\alpha)_*(\sF|_{S_\alpha})$ is an
isomorphism. Since $H^*_\tau(S_\alpha, \sF|_{S_\alpha}) \cong
H^*_\tau(X, (\iota_\alpha)_*(\sF|_{S_\alpha}))$, we are done by 
\cite[Proposition~58.89.6]{SP}.
\end{proof}

\begin{lem}\label{lem:Milnor-d-0}
Let $k$ be a field and $(X, D)$ a reduced modulus pair in $\Sch_k$. 
Assume that $X$ is of pure dimension $d \ge 1$ and
$\tau \in \{\nis, \etl\}$. Let $n \ge 1$ be an integer and
\begin{equation}\label{eqn:Milnor-d-0-0}
H^d_\tau(X, \wt{\sK}^M_{r,(X,D)}) \to 
H^d_\tau(X,  \wh{\sK}^M_{r,(X,D)});
\end{equation}
\begin{equation}\label{eqn:Milnor-d-0-0-*}
H^d_\tau(X,  \ov{{\wt{\sK}^M_{r,(X,D)}}/n}) \to 
H^d_\tau(X,  \ov{{\wh{\sK}^M_{r,(X,D)}}/n})
\end{equation}
the canonical maps. Then the following hold.
\begin{enumerate}
\item
If $k$ is infinite, then both maps are isomorphisms.
\item
If $d = 1$ and $r \le 1$, 
then both maps are isomorphisms.
\item
If $d \ge 2$ and $\tau = \nis$, then  both maps are isomorphisms.
\item
If $k$ is finite, $d \ge 3$ and $\tau = \etl$, then 
~\eqref{eqn:Milnor-d-0-0-*} is an isomorphism. 
\item
If $k \neq \F_2$ is finite, $d = 2,  r \le 2$ and
$\tau = \etl$, then ~\eqref{eqn:Milnor-d-0-0-*} is an isomorphism.
\end{enumerate}
\end{lem}
\begin{proof}
The items (1) and (2) are clear from \lemref{lem:Milnor-Kerz}.
To prove (3) and (4), we let $n \ge 1$ and consider 
the commutative diagram
\[
\xymatrix@C.3pc{
H^{d-1}_\tau(X, {\wt{\sK}^M_{r,X}}/n) \ar[r] \ar[d] &
H^{d-1}_\tau(D, {\wt{\sK}^M_{r,D}}/n) \ar[r] \ar[d] &
H^d_\tau(X,  \ov{{\wt{\sK}^M_{r,(X,D)}}/n}) \ar[r] \ar[d] &
H^{d}_\tau(X, {\wt{\sK}^M_{r,X}}/n) \ar[r] \ar[d] & 
H^{d}_\tau(D, {\wt{\sK}^M_{r,D}}/n) \ar[d] \\
H^{d-1}_\tau(X, {\wh{\sK}^M_{r,X}}/n) \ar[r]  &
H^{d-1}_\tau(D, {\wh{\sK}^M_{r,D}}/n) \ar[r]  &
H^d_\tau(X,  \ov{{\wh{\sK}^M_{r,(X,D)}}/n}) \ar[r]  &
H^{d}_\tau(X, {\wh{\sK}^M_{r,X}}/n) \ar[r]  & 
H^{d}_\tau(D, {\wh{\sK}^M_{r,D}}/n)}
\]
with $d \ge 2$.
The two rows are exact by ~\eqref{eqn:Rel-Milnor}. It suffices to show that all vertical 
arrows except possibly the middle one are isomorphisms.
But this follows easily from Lemmas~\ref{lem:Milnor-Kerz} and
~\ref{lem:Sheaf-0} if we observe that $cd_\nis(k) = 0$ for $k$ arbitrary and
$cd_\etl(k) = 1$ for $k$ finite (see \cite[Expos{\'e}~X]{SGA4}). 
The proof of ~\eqref{eqn:Milnor-d-0-0} for $\tau = \nis$
is identical if we observe that \lemref{lem:Sheaf-0}
holds even if the sheaf $\sF$ is not torsion when $\tau = \nis$.

We now prove (5). We only consider the case $r = 2$ as the other cases
are trivial.
We look at the commutative diagram of Nisnevich (or {\'e}tale) sheaves
\begin{equation}\label{eqn:Milnor-d-0-1}
\xymatrix@C.8pc{
0 \ar[r] & \ov{{\wt{\sK}^M_{r,(X,D)}}/n} \ar[r] \ar[d] &
{\wt{\sK}^M_{r,X}}/n \ar[r] \ar[d]^-{\alpha_X} & 
{\wt{\sK}^M_{r,D}}/n \ar[r] \ar[d]^-{\alpha_D} & 0 \\
0 \ar[r] & \ov{{\wh{\sK}^M_{r,(X,D)}}/n} \ar[r] &
{\wh{\sK}^M_{r,X}}/n \ar[r] & {\wh{\sK}^M_{r,D}}/n \ar[r] & 0.}
\end{equation}

We now recall from \cite[Proposition~10(6)]{Kerz-10} that for a
local ring $A$ containing $k$, the map $\wh{K}^M_2(A) \to K_2(A)$
is an isomorphism. Using this, it follows from \cite[Corollary~3.3]{Kolster}
that $\wh{K}^M_2(A)$ is generated by the Milnor symbols $\{a, b\}$ subject
to bilinearity, Steinberg relation, and the relation $\{a, b\} = -\{b, a\}$.
Since $\Ker(\wt{K}^M_2(A) \surj \wh{K}^M_2(A))$ surjects onto
$\Ker({\wt{K}^M_2(A)}/n \surj {\wh{K}^M_2(A)}/n)$, it follows that
the latter is generated by the images of the symbols of the type
$\{a, b\} + \{b, a\}$. Since $A^{\times} \to (A/I)^{\times}$ is surjective for
any ideal $I \subset A$, we conclude that the map
\begin{equation}\label{eqn:Milnor-d-0-2}
\Ker({\wt{K}^M_2(A)}/n \surj {\wh{K}^M_2(A)}/n) \to
\Ker({\wt{K}^M_2(A/I)}/n \surj {\wh{K}^M_2(A/I)}/n)
\end{equation}
is surjective.
We have therefore shown that $\Ker(\alpha_X) \surj \Ker(\alpha_D)$ in
~\eqref{eqn:Milnor-d-0-1}. Subsequently, a diagram chase shows that
\begin{equation}\label{eqn:Milnor-d-0-3}
\ov{{\wt{\sK}^M_{r,(X,D)}}/n} \to \ov{{\wh{\sK}^M_{r,(X,D)}}/n} 
\end{equation}
is a surjective map of Nisnevich (or {\'e}tale) sheaves. 
We let $\sF$ denote the kernel of this map.
A diagram similar to ~\eqref{eqn:Milnor-top-coh-0}
(with $\sK^M_{r,X}$ replaced by $\wh{\sK}^M_{r,X}$) and
\lemref{lem:Milnor-Kerz} together show that 
$\sF_x = 0$ for every $x \in X_{(q)}$ with $q > 0$.
We now apply \lemref{lem:Sheaf-0} to finish the proof.
\end{proof}

Combining Lemmas~\ref{lem:Milnor-top-coh}, ~\ref{lem:Milnor-d} and
~\ref{lem:Milnor-d-0}, we get the following key result that we shall use.

\begin{prop}\label{prop:Milnor-iso}
Let $k$ be a field and $(X, D)$ a reduced modulus pair in $\Sch_k$. 
Assume that $X$ is of pure dimension $d \ge 2$. Then there are
natural isomorphisms
\[
H^d_\nis(X,  \sK^M_{r,(X,D)}) \xrightarrow{\cong} 
H^d_\nis(X,  \wh{\sK}^M_{r,(X,D)});
\]
\[
H^d_\nis(X,  {\sK^M_{r,(X,D)}}/n) \xrightarrow{\cong} 
H^d_\nis(X,  \ov{{\wh{\sK}^M_{r,(X,D)}}/n})
\]
for every integer $n \ge 1$. If $d = 1$, these isomorphisms hold if
either $k$ is infinite or $r \le 1$.
\end{prop}

\vskip .3cm

\section{The idele class group of Kerz-Zhao}\label{sec:KZ}
Let $A$ be a local domain with fraction field $F$ and let $I = (f)$ be a 
principal ideal, where $f \in A$ is a non-zero element. 
Let $A_f$ denote the localization $A[f^{-1}]$ obtained by
inverting the powers of $f$ so that there are inclusions of rings $A \inj A_f \inj F$.
We let $\wh{K}^M_1(A|I) = K^M_1(A,I)$ and for $r \ge 2$, we let
$\wh{K}^M_r(A|I)$ denote the image of the canonical map
of abelian groups
\begin{equation}\label{eqn:RS-0}
K^M_1(A,I)  \otimes (A_f)^{\times} \otimes \cdots \otimes (A_f)^{\times} \to 
K^M_r(F),
\end{equation}
induced by the product in Milnor $K$-theory, where the tensor product is
taken $r$ times. 
These groups coincide with the relative Milnor $K$-groups of R\"ulling-Saito 
(see \cite[Definition~2.4 and Lemma~2.1]{Rulling-Saito}) when
$A$ is regular and 
$\Spec(A/f)_\red$ is a normal crossing divisor on $\Spec(A)$.
We shall denote the associated sheaf on an integral scheme $X$ with
an effective Cartier divisor $D$ by $\wh{\sK}^M_{r, X|D}$.

We let $E_r(A,I)$ be the image of the canonical map
${K}^M_r(A,I) \to K^M_r(F)$. This is same as the image of the map
$\wt{K}^M_r(A,I) \to K^M_r(F)$ by \lemref{lem:BT-KS}.
The following result will be important for us.
\begin{lem}\label{lem:RS-GK}
We have the following.
\begin{enumerate}
\item
The composite map 
\[
\wh{K}^M_r(A|I) \to K^M_r(F) \xrightarrow{\partial}
{\underset{{\rm ht} (\fp) = 1}\oplus} \ K^M_{r-1}(k(\fp))
\]
is zero. In particular, $\wh{K}^M_r(A|I) \subset \wh{K}^M_r(A)$ if 
$A$ is regular and equicharacteristic or a Henselian discrete valuation ring.
\item
There is an inclusion $E_r(A,I) \subseteq \wh{K}^M_r(A|I)$
which is an equality if $I$ is a prime ideal.
\end{enumerate}
\end{lem}
\begin{proof}
We let $\fp \subset A$ be a prime ideal of height one such that $f \in \fp$.
To prove (1), it suffices to show that the composite
$\wh{K}^M_r({A_\fp}|{I_\fp}) \to K^M_r(F) \xrightarrow{\partial}
K^M_{r-1}(k(\fp))$
is zero, where $\partial$ is the boundary map defined in 
\cite[\S~1]{Kato-86}. We can therefore assume that $\dim(A) = 1$ and
$f \notin A^\times$. By the definition of $\partial$, we can assume that
$A$ is normal. But one easily checks in this case that
$\wh{K}^M_r(A|I) \subset \wh{K}^M_r(A)$ 
(see \cite[Proposition~2.8]{Rulling-Saito})
and $\partial(\wh{K}^M_r(A)) = 0$ (see \cite[Proposition~2.7]{Dahlhausen}).
The second part of (1) follows from \cite[Proposition~10]{Kerz-10} and
\cite[Theorem~5.1]{Luders}.

We now prove (2). One checks using \lemref{lem:BT-KS} that
$E_r(A,I) \subseteq \wh{K}^M_r(A|I)$. Suppose now that $I$ is prime 
and consider an element
$\alpha = \{1 + af, u_2, \ldots , u_{r}\}$, where
$a \in A$ and $u_i \in (A_f)^\times$ for $i \ge 2$. 
Since $I$ is prime, it is easy  to check that $u_i = a_if^{n_i}$ for
$a_i \in A^{\times}$ and $n_i \ge 0$. Using the bilinearity and the Steinberg
relations, it easily follows  from  \cite[Lemma~1]{Kato-86} that
$\alpha \in E_r(A,I)$. 
\end{proof}

Suppose that $k$ is any field and $X \in \Sch_k$ is regular.
Let $D \subset X$ be an effective Cartier divisor.
It follows from \lemref{lem:RS-GK} that there are
inclusions $\wh{\sK}^M_{r, X|nD} \subset \wh{\sK}^M_{r, X} \supset 
\wh{\sK}^M_{r, (X,nD)}$ for $n \ge 1$.
We let $\ov{{\wh{\sK}^M_{r,X|D}}/m} = 
{\wh{\sK}^M_{r, X|D}}/{(\wh{\sK}^M_{r, X|D} \cap m \wh{\sK}^M_{r,X})}$.

\begin{prop}\label{prop:Rel-RS}
Let $k$ be a field and $X \in \Sch_k$ a
regular scheme of dimension $d \ge 1$. 
Let $D \subset X$ be a simple normal crossing divisor.
Let $\tau \in \{\nis, \etl\}$ and $m \ge 1$.
Then the conclusions (1) $\sim$ (5) of \lemref{lem:Milnor-d-0}
hold for the maps 
\[
\{H^d_\tau(X, \wt{\sK}^M_{r,(X,nD)})\}_n \to 
\{H^d_\tau(X,  \wh{\sK}^M_{r,X|nD})\}_n;
\]
\[
\{H^d_\tau(X,  \ov{{\wt{\sK}^M_{r,(X,nD)}}/m})\}_n \to 
\{H^d_\tau(X,  \ov{{\wh{\sK}^M_{r,X|nD}}/m})\}_n
\]
of pro-abelian groups.
\end{prop}
\begin{proof}
We let $\sE_{r, (X,nD)} = {\rm Image}(\wt{\sK}^M_{r,(X,nD)} \to 
\wh{\sK}^M_{r,X})$ and $\ov{{\sE_{r, (X,nD)}}/m} = 
{\rm Image}(\wt{\sK}^M_{r,(X,nD)} \to {\wh{\sK}^M_{r,X}}/m)$.
It is easy to check using \cite[Proposition~2.8]{Rulling-Saito} that
the inclusions $\{\sE_{r, (X,nD)}\}_n \inj \{\wh{\sK}^M_{r,X|nD}\}_n$ and
$\{\ov{{\sE_{r, (X,nD)}}/m}\}_n \inj \{\ov{{\wh{\sK}^M_{r,X|nD}}/m}\}_n$
are pro-isomorphisms. The rest of the proof is now completely identical to that
of \lemref{lem:Milnor-d-0}.
\end{proof}

Let us now assume that $k$ is any field
and $X \in \Sch_k$ is an integral scheme of dimension $d \ge 1$. 
We shall endow $X$ with its canonical dimension function $d_X$ (see \cite{Kerz-11}).
Assume that $D \subset X$ is an effective Cartier divisor. 
Let $F$ denote the function field of $X$ and $U = X \setminus D$.
We refer the reader to \cite[\S~1]{Kato-Saito} or
\cite[\S~2]{Gupta-Krishna-REC} for the definition 
and all properties of Parshin chains and their Milnor $K$-groups that
we shall need. 

We let $\sP_{U/X}$ denote the set of Parshin chains on the pair $(U \subset X)$
and let $\sP^{\max}_{U/X}$ be the subset of $\sP_{U/X}$ consisting of
maximal Parshin chains. 
If $P$ is a Parshin chain on $X$  of dimension $d_X(P)$, then we shall consider 
$K^M_{d_X(P)}(k(P))$ as a topological abelian group with its canonical
topology if $P$ is maximal (see \cite[\S~2,5]{Gupta-Krishna-REC}). 
Otherwise, we shall consider
$K^M_{d_X(P)}(k(P))$ as a topological abelian group with its discrete topology.
If $P = (p_0, \ldots , p_d)$ is a maximal Parshin chain on $X$, we shall let $P' = 
(p_0, \ldots , p_{d-1})$.

Recall from \cite[Definition~3.1]{Gupta-Krishna-REC} that the idele
group of $(U \subset X)$ is defined as 
\begin{equation}\label{eqn:Idele-U-0}
I_{U/X} = {\underset{P \in \sP_{U/X}}\bigoplus} K^M_{d_X(P)}(k(P)).
\end{equation}
We consider $I_{U/X}$ as a topological group with its direct sum topology.
We let 
\begin{equation}\label{eqn:Idele-D}
I(X|D) = {\rm Coker}\left({\underset{P \in \sP^{\max}_{U/X}}\bigoplus}
\wh{K}^M_{d_X(P)}(\sO^h_{X, P'}|I_D) \to I_{U/X}\right),
\end{equation}
where $I_D$ is the extension of the ideal sheaf $\sI_D \subset \sO_X$
to $\sO^h_{X,P'}$ and the map on the right is induced by the composition
$\wh{K}^M_{d_X(P)}(\sO^h_{X, P'}|I_D) \inj K^M_{d_X(P)}(k(P)) \to I_{U/X}$ for
$P \in \sP^{\max}_{U/X}$.
We consider $I(X|D)$ a topological group with its quotient topology.
Recall from \cite[\S~3.1]{Gupta-Krishna-REC} that $I(X,D)$ is
defined analogous to $I(X|D)$, where we replace 
$\wh{K}^M_{d_X(P)}(\sO^h_{X, P'}|I_D)$ by ${K}^M_{d_X(P)}(\sO^h_{X, P'}, I_D)$.

We let $\sQ_{U/X}$ denote the set of all $Q$-chains on $(U \subset X)$.
Recall that the idele class group of the pair $(U \subset X)$
is the topological abelian group
\begin{equation}\label{eqn:IC-0}
C_{U/X} = {\rm Coker}\left({\underset{Q \in \sQ_{U/X}}\bigoplus}
K^M_{d_X(Q)}(k(Q)) \xrightarrow{\partial_{U/X}} I_{U/X}\right)
\end{equation}
with its quotient topology.
We let
\begin{equation}\label{eqn:IC-1}
C(X|D) = {\rm Coker}\left({\underset{Q \in \sQ_{U/X}}\bigoplus}
K^M_{d_X(Q)}(k(Q)) \to I(X|D)\right).
\end{equation}
This is a topological group with its quotient topology.
This idele class group was introduced by Kerz-Zhao in 
\cite[Definition~2.1.4]{Kerz-Zhau}. We let $C(X,D)$ be the cokernel of the
map ${\underset{Q \in \sQ_{U/X}}\oplus}
K^M_{d_X(Q)}(k(Q)) \to I(X,D)$.
One checks using part (2) of \lemref{lem:RS-GK} that there is a commutative diagram
\begin{equation}\label{eqn:IC-2}
\xymatrix@C.8pc{
I_{U/X} \ar[r] \ar[d] & I(X,D) \ar[r] \ar[d] & I(X|D) \ar[d] \\
C_{U/X} \ar[r] & C(X,D) \ar[r] & C(X|D).}
\end{equation}

If $\dim(X) = d$, we let $C_{KS}(X|D) = H^d_\nis(X, \wh{\sK}^M_{d, X|D})$
and $C^{\etl}_{KS}(X|D) = H^d_\etl(X, \wh{\sK}^M_{d, X|D})$.
We let $C_{KS}(X,D) = H^d_\nis(X, {\sK}^M_{d, (X,D)})$
and $C^{\etl}_{KS}(X,D) = H^d_\etl(X, {\sK}^M_{d, (X,D)})$.

\begin{lem}\label{lem:KS-GK-iso}
The canonical map $C_{KS}(X,D) \to C_{KS}(X|D)$ is surjective.
It is an isomorphism if $D \subset X$ is a reduced Cartier divisor. 
\end{lem}
\begin{proof}
It is clear that the stalks of the kernel and
cokernel of the
map $\sK^M_{r,(X,D)} \to \wh{\sK}^M_{r,X|D}$ vanish at the generic point of $X$.
It follows moreover from \lemref{lem:RS-GK} that these stalks vanish at all 
points of $X$ with dimension $d-1$ or more if $D$ is reduced.
A variant of \lemref{lem:Sheaf-0} therefore implies the desired result.
\end{proof}

Let $\sZ_0(U)$ be the free abelian group on $U_{(0)}$.
It is clear from the definitions of $C_{KS}(X|D)$ and $C(X|D)$ that there
is a diagram
\begin{equation}\label{eqn:idele}
\xymatrix@C.8pc{
\sZ_0(U) \ar[dr] \ar[d] & \\
C(X|D) \ar@{.>}[r] & C_{KS}(X|D),}
\end{equation} 
where the left vertical arrow is induced by the inclusion of the Milnor
$K$-groups of the length zero Parshin chains
on $(U \subset X)$ into $I_{U/X}$, and the diagonal arrow is induced by
the coniveau spectral sequence for $C_{KS}(X|D)$.
The following is \cite[Theorem~3.1.1]{Kerz-Zhau} whose proof is
obtained by simply repeating the proof of \cite[Theorem~8.2]{Kerz-11}
mutatis mutandis (see \cite[Theorem~3.8]{Gupta-Krishna-REC}).

\begin{thm}\label{thm:KZ-main} 
Assume that $U$ is regular. Then one has an isomorphism  
\[
\phi_{X|D} \colon C(X|D) \xrightarrow{\cong} C_{KS}(X|D)
\]
such that ~\eqref{eqn:idele} is commutative.
\end{thm}

\subsection{Degree map for $C_{KS}(X|D)$}\label{sec:Degree}
Let $k$ be any field and $X \in \Sch_k$ an integral scheme of dimension $d \ge 1$. 
Assume that $D \subset X$ is an effective Cartier divisor.
Let $\CH^F_r(X)$ denote the Chow group of $r$-dimensional cycles
on $X$ as in \cite[Chapter~1]{Fulton}.
It follows from \lemref{lem:RS-GK} and \cite[Proposition~1]{Kato-86} that
\begin{equation}\label{eqn:KS-Chow-0}
\wh{\sK}^M_{r,X|D} \to {\underset{x \in X^{(0)}}\oplus} (\iota_x)_*(K^M_{r}(k(x)))
\to {\underset{x \in X^{(1)}}\oplus} (\iota_x)_*(K^M_{r-1}(k(x))) \to
\cdots 
\hspace*{2cm}
\end{equation}
\[
\hspace*{3.8cm} \cdots
\to {\underset{x \in X^{(d-1)}}\oplus} (\iota_x)_*(K^M_{r-d+1}(k(x))) 
\xrightarrow{\partial}
{\underset{x \in X^{(d)}}\oplus} (\iota_x)_*(K^M_{r-d}(k(x))) \to 0
\]
is a complex of Nisnevich sheaves on $X$.
An elementary cohomological argument (see \cite[\S~2.1]{Gupta-Krishna-BF})
shows that by taking the cohomology of 
the sheaves in this complex, one gets a canonical  
homomorphism 
\begin{equation} \label{eqn:RS-Chow*-1}
\nu_{X|D} \colon C_{KS}(X|D) \to \CH^{F}_0(X). 
\end{equation}
If $X$ is projective over $k$, there is a degree map 
$\deg \colon \CH^{F}_0(X) \to \Z$. We let $\deg \colon C_{KS}(X|D) \to \Z$
be the composition $C_{KS}(X|D) \xrightarrow{\nu_{X|D}} \CH^{F}_0(X) 
\xrightarrow{\deg} \Z$. We let
$C_{KS}(X|D)_0 = \Ker(\deg)$.
We let $C_{KS}(X,D)_0$ be the kernel of the composite map
$C_{KS}(X,D) \surj C_{KS}(X|D) \xrightarrow{\deg} \Z$.

\section{Comparison theorem for cohomology of 
$K$-sheaves}\label{sec:Comp**}
Let $k$ be a perfect field of characteristic $p > 0$
and $X \in \Sch_k$ an integral regular scheme of dimension $d$.
Let $D \subset X$ be an effective Cartier divisor.
Let $U = X \setminus D$ and $F = k(X)$.
Let $m, r \ge 1$ be integers. 
A combination of  \cite[Theorem~1.1.5]{JSZ} and
\cite[Theorems~2.2.2, 3.3.1]{Kerz-Zhau} implies the following.

\begin{thm}\label{thm:RS-comp}
Assume that $k$ is finite and $D_\red$ is a simple normal crossing 
divisor. Then the canonical map
\[
H^d_\nis(X, \ov{{\wh{\sK}^M_{d, X|D}}/{p^m}}) \to 
H^d_\etl(X, \ov{{\wh{\sK}^M_{d, X|D}}/{p^m}})
\]
is an isomorphism. 
\end{thm}

\begin{cor}\label{cor:RS-K-comp}
Under the assumptions of \thmref{thm:RS-comp}, the canonical map
\[
\{H^d_\nis(X, \ov{{\wh{\sK}^M_{d, (X,nD)}}/{p^m}})\}_n \to
\{H^d_\etl(X, \ov{{\wh{\sK}^M_{d, (X,nD)}}/{p^m}})\}_n 
\]
is an isomorphism if either $k \neq \F_2$ or $d \neq 2$.
\end{cor}
\begin{proof}
Combine Lemma~\ref{lem:Milnor-d-0}, \propref{prop:Rel-RS} and
\thmref{thm:RS-comp}.
\end{proof}

Our goal in this section is to prove an analogue of 
\thmref{thm:RS-comp} for the
cohomology of $\sK^M_{d, (X,D)}$ when $D_\red$ is regular. This will be used in the proof of
\thmref{thm:Main-2}.

\subsection{Filtrations of Milnor $K$-theory}
\label{sec:Filt*}
Let $A$ be a regular local ring essentially of finite type over $k$
and $I = (f)$ a principal ideal, where $f \in A$ is a non-zero 
element such that $R = {A}/{I}$ is regular.
Let $A_f$ denote the localization $A[f^{-1}]$ obtained by
inverting the powers of $f$. Let $F$ be the quotient field of
$A$ so that there are inclusions of rings $A \inj A_f \inj F$.
For an integer $n \ge 1$, let $I^n = (f^n)$. 

We let $E_r(A,I^n)$ be the image of the canonical map
$\wt{K}^M_r(A,I^n) \to \wh{K}^M_r(A) \subset K^M_r(F)$.
One easily checks that 
there is a commutative diagram (see ~\eqref{eqn:Milnor-mod-n})
\begin{equation}\label{eqn:Rel-filtn-0}
\xymatrix@C.8pc{
\wh{K}^M_r(A,I^n) \ar@{->>}[r] 
& \ov{{\wh{K}^M_r(A,I^n)}/{p^m}} \ar[rr]^-{\cong} & &
\frac{\wh{K}^M_r(A,I^n) +  p^m\wh{K}^M_r(A)}{p^m\wh{K}^M_r(A)} 
\ar@{^{(}->}[dd] & \\
\wt{K}^M_r(A,I^n) \ar@{->>}[r] \ar[u] \ar[d] & 
\ov{{\wt{K}^M_r(A,I^n)}/{p^m}} \ar@{->>}[r] \ar[u] \ar[d] & 
\frac{E_r(A,I^n) +  p^m\wh{K}^M_r(A)}{p^m\wh{K}^M_r(A)} \ar@{^{(}->}[ur]
\ar@{^{(}->}[d]  & \\
\wh{K}^M_r(A|I^n) \ar@{->>}[r] & \ov{{\wh{K}^M_r(A|I^n)}/{p^m}} 
\ar[r]^-{\cong} & 
\frac{\wh{K}^M_r(A|I^n) + p^m\wh{K}^M_r(A)}{p^m\wh{K}^M_r(A)} \ar@{^{(}->}[r] & 
\frac{\wh{K}^M_r(A)}{p^m\wh{K}^M_r(A)},}
\end{equation}
where the existence of the left-most bottom vertical arrow and
the right-most bottom horizontal inclusion follows from 
\cite[Proposition~2.8]{Rulling-Saito} (see also 
\cite[Lemma~3.8]{Gupta-Krishna-p}).

We consider the filtrations $F^\bullet_{m,r}(A)$ and
$G^\bullet_{m,r}(A)$ of ${\wh{K}^M_r(A)}/{p^m}$ by letting
$F^n_{m,r}(A) =  \frac{E_r(A,I^n) +  p^m\wh{K}^M_r(A)}{ p^m\wh{K}^M_r(A)}$ and
$G^n_{m,r}(A) = \frac{\wh{K}^M_r(A|I^n) + p^m\wh{K}^M_r(A)}{p^m\wh{K}^M_r(A)}$.
We then have a commutative diagram of filtrations 
\begin{equation}\label{eqn:Rel-filtn-1}
\xymatrix@C.8pc{
E_r(A, I^\bullet) \ar@{^{(}->}[r] \ar@{->>}[d] &
\wh{K}^M_r(A|I^\bullet) \ar@{^{(}->}[r] \ar@{->>}[d] & 
\wh{K}^M_r(A) \ar@{->>}[d] \\
F^\bullet_{m,r}(A) \ar@{^{(}->}[r] & G^\bullet_{m,r}(A) \ar@{^{(}->}[r] & 
{\wh{K}^M_r(A)}/{p^m}.}
\end{equation}

\begin{lem}\label{lem:Rel-filtn-2}
One has $\wh{K}^M_r(A|I^{n+1}) \subseteq E_r(A, I^n)$ and
$G^{n+1}_{m,r}(A) \subseteq F^{n}_{m,r}(A)$. Furthermore,
$E_r(A, I) = \wh{K}^M_r(A|I)$ and $F^{1}_{m,r}(A) = G^{1}_{m,r}(A)$.
\end{lem}
\begin{proof}
This is immediate from \cite[Proposition~2.8]{Rulling-Saito}.
\end{proof}

We let $\ov{\wt{K}^M_r(A_f)}$ be the image of the canonical map
$\wt{K}^M_r(A_f) \to K^M_r(F)$. We let
$H^n_{m,r}(A) = \frac{\wh{K}^M_r(A|I^n) + 
p^m\ov{\wt{K}^M_r(A_f)}}{p^m\ov{\wt{K}^M_r(A_f)}}$.
 Note that 
the canonical map $\wt{K}^M_r(A) \surj \wh{K}^M_r(A)$ is surjective.  
We therefore have a surjective map $G^n_{m,r}(A) \surj H^n_{m,r}(A)$.

\begin{lem}\label{lem:BK-RS}
The canonical map 
\[
\frac{G^n_{m,r}(A)}{G^{n+1}_{m,r}(A)} \to \frac{H^n_{m,r}(A)}{H^{n+1}_{m,r}(A)}
\]
is an isomorphism.
\end{lem}
\begin{proof}
The map in question is clearly surjective since its both sides are quotients
of $\wh{K}^M_r(A|I^n)$. So we need to prove only the injectivity.
Using \lemref{lem:Rel-filtn-2} and an easy reduction, one can see that
it suffices to show that the map
\[
\frac{\ov{\wt{K}^M_r(A_f)}}{\wh{K}^M_r(A)} \xrightarrow{p^m}
\frac{\ov{\wt{K}^M_r(A_f)}}{\wh{K}^M_r(A)}
\]
is injective.

Since $\ov{\wt{K}^M_r(A_f)}$ has no $p$-torsion, a snake lemma argument
shows that the previous injectivity is equivalent to the injectivity of the
map ${\wh{K}^M_r(A)}/{p^m} \to {\ov{\wt{K}^M_r(A_f)}}/{p^m}$.
We now look at the composition 
${\wh{K}^M_r(A)}/{p^m} \to {\ov{\wt{K}^M_r(A_f)}}/{p^m} \to
{K^M_r(F)}/{p^m}$. It suffices to show that this composition is injective.
But this is an easy consequence of \cite[Proposition~10(8)]{Kerz-10}
and \cite[Theorem~8.1]{Geisser-Levine}. 
\end{proof}

\vskip .3cm

We now consider the map
\begin{equation}\label{eqn:Rel-filtn-3}
\rho \colon \Omega^{r-1}_R \oplus \Omega^{r-2}_R \to
{\wh{K}^M_r(A|I^n)}/{\wh{K}^M_r(A|I^{n+1})}
\end{equation}
having the property that 
\[
\rho(a\dlog b_1 \wedge \cdots \wedge \dlog b_{r-1},0)
= \{1 + \wt{a}f^n, \wt{b}_1, \ldots , \wt{b}_{r-1}\};
\]
\[
\rho(0, a\dlog b_1 \wedge \cdots \wedge \dlog b_{r-2}) = 
\{1 + \wt{a}f^n, \wt{b}_1, \ldots , \wt{b}_{r-2}, f\},
\]
where $\wt{a}$ and $\wt{b}_i$'s are arbitrary lifts of $a$ and $b_i$'s,
respectively, under the surjection $A \surj R$.
It follows from \cite[\S4, p.~122]{Bloch-Kato} that this map is well defined.
It is clear from ~\eqref{eqn:Rel-filtn-1} and \lemref{lem:Rel-filtn-2}
that $\rho$ restricts to a map
\begin{equation}\label{eqn:Rel-filtn-4}
\rho \colon \Omega^{r-1}_R \to
{E_r(A, I^n)}/{\wh{K}^M_r(A|I^{n+1})}.
\end{equation}
This map is surjective by \lemref{lem:BT-KS}.

Let $\psi \colon R \to R$ be the absolute Frobenius morphism.
Let $Z_i\Omega^{r-1}_R$ be the unique $R$-submodule of $\psi^i_* \Omega^{r-1}_R$
such that the inverse Cartier operator 
(see \cite[Chap.~0, \S~2]{Illusie})
induces an $R$-linear isomorphism $C^{-1} \colon Z_{i-1}\Omega^{r-1}_R
\xrightarrow{\cong} {Z_{i}\Omega^{r-1}_R}/{d\Omega^{r-2}_R}$, where
we let $Z_1 \Omega^{r-1}_R = \Ker(\Omega^{r-1}_R \to \Omega^{r}_R)$.
We let $B_1\Omega^{r-1}_R = d\Omega^{r-2}_R$ and
let $B_i \Omega^{r-1}_R$ ($i \ge 2$) 
be the unique $R$-submodule of $\psi^i_* \Omega^{r-1}_R$
such that $C^{-1} \colon B_{i-1} \Omega^{r-1}_R \to 
{B_i\Omega^{r-1}_R}/{d\Omega^{r-2}_R}$ is an isomorphism of $R$-modules.

\begin{lem}\label{lem:Coherence}
Let $n = n_1p^s$, where $s \ge 0$ and $p \nmid n_1$.
Then $\rho$ induces isomorphisms
\[
\rho^n_{m,r} \colon \frac{\psi^m_* \Omega^{r-1}_R}{Z_m\Omega^{r-1}_R}
\xrightarrow{\cong} \frac{F^n_{m,r}(A)}{G^{n+1}_{m,r}(A)} \ \ {\rm if} \
m \le s;
\]
\[
\rho^n_{m,r} \colon \frac{\psi^s_* \Omega^{r-1}_R}{B_s\Omega^{r-1}_R}
\xrightarrow{\cong} \frac{F^n_{m,r}(A)}{G^{n+1}_{m,r}(A)} \ \ {\rm if} \
m > s.
\]
\end{lem}
\begin{proof}
Once it exists,  $\rho^n_{m,r}$ is clearly surjective in both cases. We therefore
need to prove its existence and injectivity. These are easy consequences
of various lemmas in \cite[\S~4]{Bloch-Kato} (see also
\cite[Proposition~1.1.9]{JSZ}) once we have \lemref{lem:BK-RS}.
The proof goes as follows.

When $m \le s$, we look at the diagram
\begin{equation}\label{eqn:Coherence-0}
\xymatrix@C.8pc{
\frac{\psi^m_* \Omega^{r-1}_R}{Z_m\Omega^{r-1}_R} 
\ar@{.>}[r]^-{\rho^n_{m,r}} \ar@{^{(}->}[d] &
\frac{F^n_{m,r}(A)}{G^{n+1}_{m,r}(A)} \ar@{^{(}->}[d] & \\
\frac{\psi^m_* \Omega^{r-1}_R}{Z_m\Omega^{r-1}_R} \oplus
\frac{\psi^m_* \Omega^{r-2}_R}{Z_m\Omega^{r-2}_R} \ar[r]^-{\rho} &
\frac{G^n_{m,r}(A)}{G^{n+1}_{m,r}(A)} \ar[r]^-{\cong} &
\frac{H^n_{m,r}(A)}{H^{n+1}_{m,r}(A)},}
\end{equation}
where the bottom right isomorphism is by \lemref{lem:BK-RS}.
The right vertical arrow is clearly injective. 
The composite arrow on the bottom is an isomorphism by
\cite[(4.7), Remark~4.8]{Bloch-Kato}. It follows that the top
horizontal arrow exists and it is injective.

When $m > s$, we look at the commutative diagram
\begin{equation}\label{eqn:Coherence-1}
\xymatrix@C.8pc{
& \frac{\psi^s_* \Omega^{r-1}_R}{B_s\Omega^{r-1}_R}
 \ar@{.>}[r]^-{\rho^n_{m,r}}  \ar[d] \ar@{^{(}->}[dl] &
\frac{F^n_{m,r}(A)}{G^{n+1}_{m,r}(A)} \ar@{^{(}->}[d] & \\
\frac{\psi^s_* \Omega^{r-1}_R}{B_s\Omega^{r-1}_R} \oplus
\frac{\psi^s_* \Omega^{r-2}_R}{B_s\Omega^{r-2}_R} \ar@{->>}[r] &
{\rm Coker}(\theta) \ar[r]^-{\rho} & 
\frac{G^n_{m,r}(A)}{G^{n+1}_{m,r}(A)} \ar[r]^-{\cong} &
\frac{H^n_{m,r}(A)}{H^{n+1}_{m,r}(A)},}
\end{equation}
where $\theta$ is defined in \cite[Lemma~4.5]{Bloch-Kato}.
The map $\rho$ on the bottom is again an isomorphism. 
Hence, it suffices to show that the left vertical arrow is injective.
But this is immediate from the definition of $\theta$ (see loc. cit.)
since $C^{-s}$ is an isomorphism. 
\end{proof}

\subsection{The comparison theorem}\label{sec:NIS-ET}
Let $k$ be a perfect field of characteristic $p > 0$. 
Let $X \in \Sch_k$ be a regular scheme of pure dimension $d \ge 1$ and
let $D \subset X$ be a regular closed subscheme of codimension one.
Let $\sF^n_{m,r, X}$ be the sheaf on $X_{\nis}$ such that for any
point $x \in X$, the stalk $(\sF^n_{m,r, X})_x$ coincides with
$F^n_{m,r}(\sO^{h}_{X,x})$, where $I = \sI_D \sO^{h}_{X,x}$.
The corresponding {\'e}tale sheaf is defined
similarly by replacing the Henselization $\sO^h_{X,x}$ by the strict
Henselization $\sO^{sh}_{X,x}$.
We define the Nisnevich and {\'e}tale sheaves $\sG^n_{m,r, X}$
in an identical way, using the groups $G^n_{m,r}(A)$ for stalks.

\begin{lem}\label{lem:Comparison-0}
The change of topology map
\[
H^i_\nis(X, {\sF^1_{m,r, X}}/{\sF^n_{m,r, X}}) \to
H^i_\etl(X, {\sF^1_{m,r, X}}/{\sF^n_{m,r, X}})
\]
is an isomorphism for every $i \ge 0$.
\end{lem}
\begin{proof}
We consider the exact sequence of Nisnevich (or {\'e}tale) sheaves
\begin{equation}\label{eqn:Comparison-10}
0 \to \frac{\sF^n_{m,r, X}}{\sG^{n+1}_{m,r, X}}
\to \frac{\sG^1_{m,r, X}}{\sG^{n+1}_{m,r, X}}
\to \frac{\sF^1_{m,r, X}}{\sF^n_{m,r, X}} \to 0,
\end{equation}
where we have replaced $\sG^1_{m,r, X}$ by $\sF^1_{m,r, X}$ in the quotient on the
right using \lemref{lem:Rel-filtn-2}.
Working inductively on $n$, it follows from \cite[Proposition~1.1.9]{JSZ}
that the middle term in ~\eqref{eqn:Comparison-10} has identical
Nisnevich and {\'e}tale cohomology. It suffices therefore to show 
that the same holds for the left term in  ~\eqref{eqn:Comparison-10} too.

Since this isomorphism is obvious for $i = 0$, it suffices to show using the 
Leray spectral sequence that
${\bf R}^i\epsilon_*({\sF^n_{m,r, X}}/{\sG^{n+1}_{m,r, X}}) = 0$ for
$i > 0$.
Equivalently, we need to show that if $A = \sO^h_{X,x}$ for some $x \in X$
and $f \in A$ defines $D$ locally at $x$, then
$H^i_\etl(A,  {F^n_{m,r}(A)}/{G^{n+1}_{m,r}(A)}) = 0$ for $i > 0$.
But this is immediate from \lemref{lem:Coherence}. 
\end{proof}

\begin{lem}\label{lem:Comparison-1}
Assume that $k$ is finite. Then the change of topology map
\[
H^d_\nis(X, {\sF^n_{m,d, X}}) \to
H^d_\etl(X, {\sF^n_{m,d, X}})
\]
is an isomorphism.
\end{lem}
\begin{proof}
We first assume $n =1$. Using \lemref{lem:Rel-filtn-2}, we can
replace $\sF^1_{m,d, X}$ by $\sG^1_{m,d, X}$. 
By \cite[Theorem~1.1.5]{JSZ} and 
\cite[Theorems~2.2.2, Proposition~2.2.5]{Kerz-Zhau},
the dlog map $\sG^1_{m,d, X} \to W_m\Omega^d_{X, \log}$ is an isomorphism
in Nisnevich and {\'e}tale topologies. We conclude by
\cite[Theorem~3.3.1]{Kerz-Zhau}.

We now assume $n \ge 2$ and look at the exact sequence
\[
0 \to \sF^n_{m,d, X} \to \sF^1_{m,d, X} \to \frac{\sF^1_{m,d, X}}{\sF^n_{m,d, X}}
\to 0.
\]
We saw above that the middle term can be replaced by $W_m\Omega^d_{X, \log}$.
We can now conclude the proof by Lemmas~\ref{lem:Comparison-0}
and the independent statement \lemref{lem:Kato-complex} via the 5-lemma.
\end{proof}

\begin{lem}\label{lem:Comparison-2}
Assume that $k$ is finite. Then the change of topology map
\[
H^d_\nis(X, \ov{{\wt{\sK}^M_{d, (X, nD)}}/{p^m}}) \to
H^d_\etl(X, \ov{{\wt{\sK}^M_{d, (X, nD)}}/{p^m}})
\]
is an isomorphism if either $d \neq 2$ or $k \neq \F_2$.
\end{lem}
\begin{proof}
In view of \lemref{lem:Comparison-1}, it suffices to show that the canonical 
map (see ~\eqref{eqn:Rel-filtn-0}) 
\begin{equation}\label{eqn:Comparison-2-0}
H^d_\tau(X, \ov{{\wt{\sK}^M_{d, (X, nD)}}/{p^m}}) \to 
H^d_\tau(X, {\sF^n_{m,d, X}}) 
\end{equation}
is an isomorphism for $\tau \in \{\nis, \etl\}$.
Using \lemref{lem:Sheaf-0} and ~\eqref{eqn:Rel-filtn-0}, the proof of this is 
completely identical to that of \lemref{lem:Milnor-d-0}.
\end{proof}

The following is the main result we were after in this section. We restate
all assumptions for convenience.

\begin{thm}\label{thm:Comparison-3}
Let $k$ be a finite field of characteristic $p$. 
Let $X \in \Sch_k$ be a regular scheme of pure dimension $d \ge 1$ and
let $D \subset X$ be a regular closed subscheme of codimension one.
Assume that either $d \neq 2$ or $k \neq \F_2$.
Then there is a natural isomorphism
\[
{H^d_\nis(X, {\sK^M_{d,(X,nD)}})}/{p^m} \xrightarrow{\cong}
H^d_\etl(X, \ov{{\wh{\sK}^M_{d, (X,nD)}}/{p^m}})
\]
for every $m,n \ge 1$.
\end{thm}
\begin{proof}
Combine Lemmas~\ref{lem:Milnor-top-coh}, ~\ref{lem:Milnor-d}, 
~\ref{lem:Milnor-d-0} and ~\ref{lem:Comparison-2}.
\end{proof}

\begin{cor}\label{cor:Comparison-4}
Under the assumptions of \thmref{thm:Comparison-3}, there is a commutative
diagram
\begin{equation}\label{eqn:Comparison-4-0}
\xymatrix@C.8pc{
{H^d_\nis(X, {\sK^M_{d,(X,nD)}})}/{p^m} \ar[r] \ar[d] &
H^d_\etl(X, \ov{{\wh{\sK}^M_{d, (X,nD)}}/{p^m}}) \ar[d] \\
H^d_\nis(X, \ov{{\wh{\sK}^M_{d, X|nD}}/{p^m}}) \ar[r] & 
H^d_\etl(X, \ov{{\wh{\sK}^M_{d, X|nD}}/{p^m}})}
\end{equation}
for $m,n \ge 1$ in which the horizontal arrows are isomorphisms.
\end{cor}
\begin{proof}
In view of Theorems~\ref{thm:RS-comp} and ~\ref{thm:Comparison-3}, we
only have to explain the vertical arrows.
But their existence follows from Lemmas~\ref{lem:Milnor-top-coh}, 
~\ref{lem:Milnor-d} and ~\ref{lem:Milnor-d-0}, and the 
diagram ~\eqref{eqn:Rel-filtn-0}.
\end{proof}

\begin{remk}\label{remk:Comparison-5}
  The reader may note that \thmref{thm:RS-comp} holds whenever $D_\red$ is a simple normal crossing
  divisor, but  this is not the case with \thmref{thm:Comparison-3}. The main reason for this
  is that while $\wh{K}^M_{d,X|D}$ satisfies cdh-descent for closed covers \cite{JSZ}, the same
  is not true for $\wh{K}^M_{d,(X,D)}$. This prevents us from using an induction argument on the number
  of irreducible components of $D$.
\end{remk}

\section{Relative logarithmic Hodge-Witt sheaves}
\label{sec:Hodge-Witt}
In this section, we shall recall the relative logarithmic Hodge-Witt sheaves
and show that they coincide with the Kato-Saito relative Milnor $K$-theory
with ${\Z}/{p^m}$ coefficients in a pro-setting. We fix a field $k$ of 
characteristic $p > 0$ and let $X \in \Sch_k$.

\subsection{The relative Hodge-Witt sheaves}\label{sec:dRW}
Let $\{W_m\Omega^\bullet_X\}_{m \ge 1}$ denote the pro-complex of de 
Rham-Witt (Nisnevich) sheaves on $X$. This is a pro-complex of differential
graded algebras with the structure map $R$ and the
differential $d$. Let $[-]_m \colon \sO_X \to W_m\sO_X$ be the multiplicative
Teichm{\"u}ller homomorphism. Recall that the pro-complex 
$\{W_m\Omega^\bullet_X\}_{m \ge 1}$ is equipped with the
Frobenius homomorphism of graded algebras $F \colon W_m\Omega^r_X \to 
W_{m-1}\Omega^r_X$ and the additive Verschiebung homomorphism 
$V \colon W_m\Omega^r_X \to W_{m+1}\Omega^r_X$ .
These homomorphisms satisfy the following properties which we shall use
frequently.

\begin{enumerate}
\item $FV = p = VF$.
\item $FdV = d$, $dF = p Fd$ and $p dV = Vd$.
\item $F(d[a]) = [a^{p-1}] d[a]$
for all  $a \in \sO_X$. 
\item $V(F(x) y ) = x V(y)$ and $V(x dy)= Vx d Vy$ for all 
$x \in W_m \Omega^i_X$, $y \in W_m \Omega^j_X$. 
\end{enumerate}

Let $\Fil^i_V W_m\Omega^r_X = V^iW_{m-i}\Omega^r_X + dV^iW_{m-i}\Omega^{r-1}_X,
\ \Fil^i_R W_m\Omega^r_X = \Ker(R^{r-i})$ and $\Fil^i_p W_m\Omega^r_X = 
\Ker(p^{r-i})$ for $0 \le i \le r$. By \cite[Lemma~3.2.4]{Hesselholt-Madsen}
and \cite[Proposition~I.3.4]{Illusie}, we know that
$\Fil^i_V W_m\Omega^r_X = \Fil^i_R W_m\Omega^r_X \subseteq
\Fil^i_p W_m\Omega^r_X$ and the last inclusion is an equality if
$X$ is regular. We let $Z W_m\Omega^r_X = \Ker(d \colon W_m\Omega^r_X \to
W_m\Omega^{r+1}_X)$ and $B W_m\Omega^r_X = 
{\rm Image}(d \colon W_m\Omega^{r-1}_X \to
W_m\Omega^{r}_X)$. We set $\sH^r(W_m\Omega^\bullet_X) = 
{Z W_m\Omega^r_X}/{B W_m\Omega^r_X}$.

Let $\psi \colon X \to X$ denote the absolute Frobenius morphism.
One then knows that $d \colon \psi_* W_1 \Omega^r_X \to \psi_* W_1 \Omega^{r+1}_X$
is $\sO_X$-linear. In particular, the sheaves $\psi_* Z W_1\Omega^r_X, \ 
\psi_* B W_1\Omega^r_X$ and
$\psi_* \sH^r(W_1\Omega^\bullet_X)$ are quasi-coherent on $X_\nis$.
If $X$ is $F$-finite, then $\psi$ is a finite morphism, and therefore,
$\psi_* Z W_1\Omega^r_X, \ \psi_* B W_1\Omega^r_X$ and
$\psi_* \sH^r(W_1\Omega^\bullet_X)$ are coherent sheaves on $X_\nis$.

Let $D \subset X$ be a closed subscheme defined by a sheaf of ideals $\sI_D$
and let $nD \subset X$ be the closed subscheme defined by $\sI^n_D$ for 
$n \ge 1$. 
We let $W_m\Omega^\bullet_{(X,D)} = \Ker(W_m\Omega^\bullet_X \surj 
\iota_* W_m\Omega^\bullet_D)$, where $\iota \colon D \inj X$ is the inclusion.
We shall often use the notation $W_m\Omega^\bullet_{(A,I)}$ if $X = \Spec(A)$
and $D = \Spec(A/I)$.
If $D$ is an effective Cartier divisor on $X$ and $n\in \Z$, 
we let $W_m \sO_X(nD)$ be the 
Nisnevich sheaf on $X$ which is locally defined as
$W_m \sO_X(n D)  = [f^{-n}] \cdot W_m \sO_X \subset j_* W_m\sO_U$,
where $f\in \sO_X$ is a local equation of $D$ and $j \colon U \inj X$ is the
inclusion of the complement of $D$ in $X$. One knows that $W_m \sO_X(nD)$
is a sheaf of invertible $W_m\sO_X$-modules. We let $W_m\Omega^\bullet_X(nD) =
W_m\Omega^\bullet_X \otimes_{W_m\sO_X} W_m\sO_X(nD)$. It is clear that there is
a canonical $W_m\sO_X$-linear map $W_m\Omega^\bullet_X(-D) \to  
W_m\Omega^\bullet_{(X,D)}$. The reader is invited to compare the
following with a weaker statement \cite[Proposition~2.5]{Geisser-Hesselholt-top}.

\begin{lem}\label{lem:pro-iso-rel}
If $(X,D)$ is a modulus pair and $m \ge 1$ an
integer, then the canonical map
of pro-sheaves 
\[
\{W_m\Omega^\bullet_X(-nD)\}_n \to \{W_m\Omega^\bullet_{(X,nD)}\}_n
\]
on $X_\nis$ is surjective.
If $X$ is furthermore regular, then the 
map $W_m\Omega^\bullet_X(-D) \to W_m\Omega^\bullet_{(X,D)}$ is injective.
\end{lem}
\begin{proof}
The first part follows directly from \lemref{lem:pro-iso-rel-local} below.
We prove the second part.
Let $U = X \setminus D$ and $j \colon U \inj X$ the inclusion.
Using the N{\'e}ron-Popescu approximation and the Gersten resolution of
$W_m\Omega^\bullet_X$ (when $X$ is smooth), proven in
\cite[Proposition~II.5.1.2]{Gross}, it follows that the canonical map
$W_m\Omega^\bullet_X \to j_* W_m\Omega^\bullet_U$ is injective (see
\cite[Proposition~2.8]{KP-dRW}).
Since $W_m\sO_X(-D)$ is invertible as a $W_m\sO_X$-module, 
there are natural isomorphisms 
$j_* W_m\Omega^\bullet_U \otimes_{W_m\sO_X} W_m\sO_X(-D) 
\to j_*(W_m\Omega^\bullet_U \otimes_{W_m\sO_U} j^* W_m\sO_X(-D)) \cong
j_* W_m\Omega^\bullet_U$ (see \cite[Exercise~II.5.1(d)]{Hartshorne-AG}).
Using the invertibility of $W_m\sO_X(-D)$ again, we conclude that the
canonical map $W_m\Omega^\bullet_X(-D) \to j_* W_m\Omega^\bullet_U$ is injective.
Since this map factors through 
$W_m\Omega^\bullet_X(-D) \to W_m\Omega^\bullet_{(X,D)} \inj W_m\Omega^\bullet_X$,
the lemma follows.
\end{proof}

\begin{lem}\label{lem:pro-iso-rel-local}
Let $A$ be a commutative $\F_p$-algebra and $I = (f) \subseteq A$ a
principal ideal. Let $m \ge 1$ be any integer and $n = p^m$.
Then the inclusion
$W_m\Omega^r_{(A, I^n)} \subseteq W_m\Omega^r_A$ of $W_m(A)$-modules 
factors through $W_m\Omega^r_{(A, I^n)} \subseteq  [f] \cdot W_m\Omega^r_A$
for every $r \ge 0$.
\end{lem}
\begin{proof}
By \cite[Lemma~1.2.2]{Hesselholt-2004}, $W_m\Omega^r_{(A,I)}$ is the
$W_m(A)$-submodule of $W_m\Omega^r_A$ generated by the de Rham-Witt forms
of the type $a_0da_1 \wedge \cdots \wedge da_r$,
where $a_i \in W_m(A)$ for all $i$ and $a_i \in W_m(I)$ for some $i$.
We let $\omega = a_0da_1 \wedge \cdots \wedge da_r \in 
W_m \Omega^r_{(A, I^n)}$ such that $a_i \in W_m(I^{n})$ for some 
$i \neq 0$. 
We can assume after a permutation that $i = 1$. We let
$\omega' = da_2 \wedge \cdots \wedge da_r$.
We can write $a_1 = \stackrel{m-1}{\underset{i = 0}\sum} 
V^i([f^{n}]_{m-i}[a'_i]_{m-i})$ for some $a'_i \in A$.
It follows that
$\omega = \stackrel{m-1}{\underset{i = 0}\sum} 
a_0dV^i([f^{n}]_{m-i}[a'_i]_{m-i}) \wedge \omega'$.

We fix an integer $0 \le i \le m-1$ and let
$\omega_i = a_0dV^i([f^{n}]_{m-i}[a'_i]_{m-i}) \wedge \omega'$.
It suffices to show that each $\omega_i \in W_m\Omega^r_X(-D)$, where
$X = \Spec(A)$ and $D = \Spec(A/I)$.
However, we have
\[
\begin{array}{lll}
\omega_i  & = & a_0dV^i([f^{n}]_{m-i}[a'_i]_{m-i}) \wedge \omega' \\
& = & a_0d([f^{p^{m-i}}]_mV^i[a'_i]_{m-i}) \wedge \omega' \\
& = & [f^{p^{m-i}}]_ma_0dV^i[a'_i]_{m-i} \wedge \omega' +
p^{m-i} [f^{p^{m-i}-1}]_ma_0V^i[a'_i]_{m-i}d[f]_m \wedge \omega'.
\end{array}
\]
It is clear that the last term lies in $W_m\Omega^r_X(-D)$.
This finishes the proof.
\end{proof}

The following result will play a key role
in our exposition.

\begin{cor}\label{cor:pro-iso-rel*}
If $(X,D)$ is a regular modulus pair, then the 
canonical map of pro-sheaves 
\[
\{W_m\Omega^\bullet_X(-nD)\}_n \to \{W_m\Omega^\bullet_{(X,nD)}\}_n
\]
on $X_\nis$ is an isomorphism for every $m \ge 1$.
\end{cor}

\subsection{The relative logarithmic sheaves}
\label{sec:Rel-lo}
In the remaining part of \S~\ref{sec:Hodge-Witt}, our default topology will
be the {\'e}tale topology. All other topologies will be mentioned specifically.
Let $k$ be a field of characteristic $p > 0$ and $D \subseteq X$ a closed
immersion in $\Sch_k$.
Recall from \cite{Illusie} that $W_m\Omega^r_{X, \log}$ is the {\'e}tale
subsheaf of $W_m\Omega^r_X$ which is the image of the
map $\dlog \colon {\wt{\sK}^M_{r,X}}/{p^m} \to W_m\Omega^r_X$, given by
$\dlog(\{a_1, \ldots , a_r\}) = \dlog[a_1]_m \wedge \cdots \wedge \dlog[a_r]_m$.
It is easily seen that this map exists.
One knows (e.g., see \cite[Remark~1.6]{Luders-Morrow}) that this map in fact 
factors through 
$\dlog \colon {\wh{\sK}^M_{r,X}}/{p^m} \surj W_m\Omega^r_{X, \log}$.
Moreover, this map is multiplicative.
The naturality of the dlog map gives rise to its relative version
\begin{equation}\label{eqn:dlog-rel}
\dlog \colon \ov{{\wh{\sK}^M_{r,(X,D)}}/{p^m}} \to W_m\Omega^r_{(X,D), \log}
:= \Ker(W_m\Omega^r_{X, \log} \surj W_m\Omega^r_{D, \log})
\end{equation}
which is a morphism of {\'e}tale sheaves of ${\wh{\sK}^M_{*,X}}/{p^m}$-modules
as $r \geq 0$ varies.

Recall that the Frobenius map
$F \colon W_{m+1}\Omega^r_X \to  W_{m}\Omega^r_X$ sends $\Ker(R)$ into the 
subsheaf $dV^{m-1}\Omega^{r-1}_X$ so that there is an induced map
$\ov{F} \colon W_{m}\Omega^r_X \to {W_{m}\Omega^r_X}/{dV^{m-1}\Omega^{r-1}_X}$.
We let $\pi \colon W_{m}\Omega^r_X \to {W_{m}\Omega^r_X}/{dV^{m-1}\Omega^{r-1}_X}$
denote the projection map. Since $R - F \colon
W_{m+1}\Omega^r_X \to  W_{m}\Omega^r_X$ is surjective by 
\cite[Proposition~I.3.26]{Illusie}, it follows that $\pi -\ov{F}$ is also
surjective. 
Indeed, this surjectivity is proven for smooth schemes over $k$ in loc. cit., but 
then the claim follows because $X$ can be seen locally as a closed subscheme 
of a regular scheme and a regular scheme is a
 filtered inductive limit of smooth schemes 
 over $k$, by N\'eron–Popescu desingularisation.
 We thus get a commutative diagram 
\begin{equation}\label{eqn:dlog-2-*}
\xymatrix@C.8pc{
0 \ar[r] &  W_{m}\Omega^r_{X, \log} \ar[r] \ar@{->>}[d] &
W_{m}\Omega^r_X \ar@{->>}[d] \ar[r]^-{\pi - \ov{F}} & 
\frac{W_{m}\Omega^r_X}{dV^{m-1}\Omega^{r-1}_X} \ar[r] \ar@{->>}[d] & 0 \\
0 \ar[r] &  W_{m}\Omega^r_{D, \log} \ar[r] &
W_{m}\Omega^r_D \ar[r]^-{\pi - \ov{F}}  & 
\frac{W_{m}\Omega^r_D}{dV^{m-1}\Omega^{r-1}_D} \ar[r] & 0.}
\end{equation}
The middle and the right vertical arrows are surjective by definition 
 and the two rows are exact by loc. cit. and \cite[Corollary~4.2]{Morrow-ENS}. 
The
surjectivity of the left vertical arrow follows by applying the dlog map to
the surjection $\wh{\sK}^M_{r,X} \surj \wh{\sK}^M_{r,D}$.
Taking the kernels of the vertical arrows, we get a short exact sequence
\begin{equation}\label{eqn:dlog-2}
0 \to W_m\Omega^r_{(X,D), \log} \to W_m\Omega^r_{(X,D)} \to 
\frac{W_m\Omega^r_{(X,D)}}{W_m\Omega^r_{(X,D)} \cap dV^{m-1} \Omega^{r-1}_{X}} \to 0.
\end{equation}

The following property of the relative dlog map will be important to us.

\begin{lem}\label{lem:Rel-dlog-iso}
Assume that $k$ is perfect and $X$ is regular. Then the 
map of {\'e}tale pro-sheaves
\[
\dlog \colon \{\ov{{\wh{\sK}^M_{r,(X,nD)}}/{p^m}}\}_n \to
\{W_m\Omega^r_{(X,nD), \log}\}_n 
\]
is an isomorphism for every $m, r \ge 1$.
\end{lem}
\begin{proof}
We consider the commutative diagram of pro-sheaves:
\begin{equation}\label{eqn:dlog-0}
\xymatrix@C.8pc{
0 \ar[r] & \{\ov{{\wh{\sK}^M_{r,(X,nD)}}/{p^m}}\}_n \ar[r] \ar[d]_-{\dlog} &
{\wh{\sK}^M_{r,X}}/{p^m} \ar[d]^-{\dlog}_-{\cong} \ar[r] & 
\{{\wh{\sK}^M_{r,nD}}/{p^m}\}_n \ar[d]^-{\dlog} \ar[r] & 0 \\ 
0 \ar[r] & \{W_m\Omega^r_{(X,nD), \log}\}_n \ar[r] & W_m\Omega^r_{X,\log} \ar[r] &
\{W_m\Omega^r_{nD,\log}\}_n \ar[r] & 0.}
\end{equation}

The middle vertical arrow is an isomorphism by the Bloch-Gabber-Kato theorem
for fields \cite[Corollary~2.8]{Bloch-Kato} and the proof of the Gersten 
conjecture for Milnor $K$-theory (\cite[Proposition~10]{Kerz-10}) and
logarithmic Hodge-Witt sheaves (\cite[Th{\`e}or{\'e}me~1.4]{Gross-Suwa-88}). 
The right vertical arrow is an
isomorphism by \cite[Theorem~0.3]{Luders-Morrow}. Since the rows are exact,
it follows that the left vertical arrow is also an isomorphism.
This finishes the proof.
\end{proof}

We now prove a Nisnevich version of \lemref{lem:Rel-dlog-iso}.
For any $k$-scheme $X$, we let $W_m\Omega^r_{X, \log, \nis}$ be the image of the
map of Nisnevich sheaves
$\dlog \colon {\wh{\sK}^M_{r,X}}/{p^m} \to W_m\Omega^r_X$, given by
$\dlog(\{a_1, \ldots , a_r\}) = \dlog[a_1]_m \wedge \cdots \wedge \dlog[a_r]_m$
(see \cite[Remark~1.6]{Luders-Morrow} for the existence of this map).
The naturality of this map gives rise to its relative version
\begin{equation}\label{eqn:dlog-rel-*}
\dlog \colon \ov{{\wh{\sK}^M_{r,(X,D)}}/{p^m}} \to W_m\Omega^r_{(X,D), \log, \nis}
:= \Ker(W_m\Omega^r_{X, \log, \nis} \surj W_m\Omega^r_{D, \log, \nis})
\end{equation}
if $D \subseteq X$ is a closed immersion.

\begin{lem}\label{lem:Rel-dlog-iso-nis}
Assume that $k$ is perfect and $X$ is regular. Then the 
map of Nisnevich pro-sheaves
\[
\dlog \colon \{\ov{{\wh{\sK}^M_{r,(X,nD)}}/{p^m}}\}_n \to
\{W_m\Omega^r_{(X,nD), \log, \nis}\}_n 
\]
is an isomorphism for every $m, r \ge 1$.
\end{lem}
\begin{proof}
The proof is completely identical to that of \lemref{lem:Rel-dlog-iso}, where
we only have to observe that the middle and the right vertical arrows
are isomorphisms even in the Nisnevich topology, by the same references that
we used for the {\'e}tale case.
\end{proof}

\begin{lem}\label{lem:Kato-complex}
Assume that $k$ is a finite field and $X$ is regular of pure dimension 
$d \ge 1$. Then the canonical map
\[
H^i_\nis(X, {\wh{\sK}^M_{d,X}}/{p^m}) \to H^i_\etl(X, {\wh{\sK}^M_{d,X}}/{p^m})
\]
is an isomorphism for $i = d$ and surjective for $i = d-1$.
\end{lem}
\begin{proof}
We have seen in the proofs of Lemmas~\ref{lem:Rel-dlog-iso} and
~\ref{lem:Rel-dlog-iso-nis} that for every $r \ge 0$,
there are natural isomorphisms
\begin{equation}\label{eqn:Kato-complex-0}
{({\wh{\sK}^M_{r,X}}/{p^m})}_{\nis} \cong W_m\Omega^r_{X, \log, \nis} \  \
{\rm and} \ \
{({\wh{\sK}^M_{r,X}}/{p^m})}_{\etl} \cong W_m\Omega^r_{X, \log}.
\end{equation}

Using these isomorphisms, the lemma follows from 
\cite[Propositions~3.3.2, 3.3.3]{Kerz-Zhau}. The
only additional input one has to use is that 
$H^b_x(X_\etl, W_m\Omega^d_{X, \log}) = 0$ if $x \in X^{(a)}$ and $b < a$ by
\cite[Corollary~2.2]{Milne-Zeta} so that the right hand side of
[loc. cit.,  (3.3.3)] is zero. This is needed in the proof of the
$i = d-1$ case of the lemma. 
\end{proof}

Assume now that $k$ is a perfect field, $X \in \Sch_k$ 
is a regular scheme of pure dimension 
$d \ge 1$ and $D \subset X$  is a nowhere dense closed
subscheme. We fix integers $m \geq 1$ and $i, r\geq 0$. 
Under these assumptions, we shall prove the 
following results.

\begin{lem}\label{lem:log-HW}
There is a short exact sequence
\[
0 \to W_{m-i}\Omega^r_{X, \log} \xrightarrow{\un{p}^i}
W_{m}\Omega^r_{X, \log} \xrightarrow{R^{m-i}}  W_{i}\Omega^r_{X, \log}
\to 0
\]
of sheaves on $X$ in Nisnevich and {\'e}tale 
topologies.
\end{lem}
\begin{proof}
The {\'e}tale version of the lemma is already known by
\cite[Lemma~3]{CTSS}. To prove its Nisnevich version, we can use
the first isomorphism of ~\eqref{eqn:Kato-complex-0}.
This reduces the proof to showing that tensoring the exact sequence
\[
0 \to {\Z}/{p^{m-i}} \xrightarrow{p^i} {\Z}/{p^{m}} \xrightarrow{\pi_{m-i}} 
{\Z}/{p^{i}} \to 0
\]
with ${\wh{\sK}^M_{r,X}}$ yields an exact sequence
\begin{equation}\label{eqn:log-HW-0}
0 \to  {\wh{\sK}^M_{r,X}}/{p^{m-i}} \xrightarrow{p^i} 
{\wh{\sK}^M_{r,X}}/{p^{m}} \xrightarrow{\pi_{m-i}} {\wh{\sK}^M_{r,X}}/{p^{i}} \to 0
\end{equation}
of Nisnevich sheaves on $X$.
But this follows directly from the fact that ${\wh{\sK}^M_{r,X}}$ has
no $p^i$-torsion (see \cite[Theorem~8.1]{Geisser-Levine}).
\end{proof}

\begin{lem}\label{lem:log-HW-pro}
 There is a short exact sequence
\[
0 \to \{W_{m-i}\Omega^r_{nD, \log}\}_n \xrightarrow{\un{p}^i}
\{W_{m}\Omega^r_{nD, \log}\}_n \xrightarrow{R^{m-i}}  \{W_{i}\Omega^r_{nD, \log}\}_n
\to 0
\]
of pro-sheaves on $D$ in Nisnevich and {\'e}tale 
topologies.
\end{lem}
\begin{proof}
The argument below works for either of Nisnevich and {\'e}tale 
topologies.
We shall prove the lemma by modifying the proof of 
\cite[Theorem~4.6]{Morrow-ENS}. The latter result says that there is an 
exact sequence
\begin{equation}\label{eqn:log-HW-pro-0}
\{W_{m}\Omega^r_{nD, \log}\}_n \xrightarrow{{p}^i}
\{W_{m}\Omega^r_{nD, \log}\}_n \xrightarrow{R^{m-i}}  \{W_{i}\Omega^r_{nD, \log}\}_n
\to 0.
\end{equation}
It suffices therefore to show that the first arrow in this sequence has a
factorization 
\begin{equation}\label{eqn:log-HW-pro-1}
\{W_{m}\Omega^r_{nD, \log}\}_n \stackrel{R^i}{\surj} 
\{W_{m-i}\Omega^r_{nD, \log}\}_n \stackrel{\un{p}^i}{\inj} 
\{W_{m}\Omega^r_{nD, \log}\}_n.
\end{equation}

To prove this factorization, we look at the commutative diagram
\begin{equation}\label{eqn:log-HW-pro-2}
\xymatrix@C.8pc{
\{W_{m}\Omega^r_{nD, \log}\}_n \ar[r]^-{{p}^i} \ar@{^{(}->}[d] &
\{W_{m}\Omega^r_{nD, \log}\}_n \ar@{^{(}->}[d] \\
\{W_{m}\Omega^r_{nD}\}_n \ar[r]^-{{p}^i}  &
\{W_{m}\Omega^r_{nD}\}_n.}
\end{equation}
It follows from \cite[Proposition~2.14]{Morrow-ENS} that the map $p^i$ at
the bottom has a factorization (see \S~\ref{sec:dRW} for the definitions of
various filtrations of $W_m\Omega^\bullet_{nD}$)
\begin{equation}\label{eqn:log-HW-pro-3}
\{W_{m}\Omega^r_{nD}\}_n  \stackrel{R^i}{\surj} 
\{W_{m-i}\Omega^r_{nD}\}_n \stackrel{\un{p}^i}{\inj} 
\{W_{m}\Omega^r_{nD}\}_n.
\end{equation}
Since the image of $\{W_{m}\Omega^r_{nD, \log}\}_n$ under $R^i$ is
$\{W_{m-i}\Omega^r_{nD, \log}\}_n$ (note the surjectivity of the second
arrow in ~\eqref{eqn:log-HW-pro-0}), 
it follows at once  
from ~\eqref{eqn:log-HW-pro-3} that the map $p^i$ on the 
top in ~\eqref{eqn:log-HW-pro-2} has a factorization as
desired in ~\eqref{eqn:log-HW-pro-1}. This concludes the proof.
\end{proof}

\begin{prop}\label{prop:log-HW-pro-*}
There is a short exact sequence
\[
0 \to \{W_{m-i}\Omega^r_{(X,nD), \log}\}_n \xrightarrow{\un{p}^i}
\{W_{m}\Omega^r_{(X,nD), \log}\}_n \xrightarrow{R^{m-i}}  
\{W_{i}\Omega^r_{(X,nD), \log}\}_n
\to 0
\]
of pro-sheaves on $X$ in Nisnevich and {\'e}tale 
topologies.
\end{prop}
\begin{proof} 
Combine the previous two lemmas and use that the maps
$W_{m}\Omega^r_{X, \log} \to W_{m}\Omega^r_{nD, \log}$ and
$W_{m}\Omega^r_{X, \log, \nis} \to W_{m}\Omega^r_{nD, \log, \nis}$ are surjective.
\end{proof}

Applying the cohomology functor, we get

\begin{cor}\label{cor:log-HW-pro-coh}
There is a long exact 
sequence of pro-abelian groups
\[
\cdots \to \{H^j_\etl(X, W_{m-i}\Omega^r_{(X,nD), \log})\}_n 
\xrightarrow{\un{p}^i} \{H^j_\etl(X, W_{m}\Omega^r_{(X,nD), \log})\}_n
\hspace*{3cm} 
\]
\[ 
\hspace*{7.5cm} 
\xrightarrow{R^{m-i}} \{H^j_\etl(X, W_{i}\Omega^r_{(X,nD), \log})\}_n 
\xrightarrow{\partial^i}
\cdots
\] 
The same also holds in the Nisnevich topology.
\end{cor}

\vskip .3cm

\subsection{Some cohomology exact sequences}\label{sec:Exact}
Let us now assume that $k$ is a finite field and $X \in \Sch_k$ is regular of 
pure dimension $d \ge 1$. For any $p^m$-torsion abelian group $V$, we let $V^\star = 
\Hom_{{\Z}/{p^m}}(V,{\Z}/{p^m})$.
Let $D \subset X$ be an effective Cartier divisor
with complement $U$.
We let $F^j_{m,r}(n) = H^j_\etl(X, W_{m}\Omega^{r}_{(X,nD), \log})$
and $F^j_{m,r}(U) = {\varprojlim}_n F^j_{m,r}(n)$.
Each group ${(F^j_{m,r}(n))}^\star$
is a profinite abelian group (see \cite[Theorem~2.9.6]{Pro-fin}).

\begin{lem}\label{lem:Lim-exact}
The sequence
\[
\cdots \to F^j_{m-1,r}(U) \to F^j_{m,r}(U) \to F^j_{1,r}(U) \to
F^{j+1}_{m-1,r}(U) \to \cdots
\]
is exact.
\end{lem}
\begin{proof}
We first prove by induction on $m$ that ${\varprojlim}^1_n  F^j_{m,r}(n) = 0$.
The $m =1$ case follows from \lemref{lem:Finite-0}. In general, 
we break the long exact sequence of \corref{cor:log-HW-pro-coh}
into short exact sequences. This yields exact sequences
of pro-abelian groups
\begin{equation}\label{eqn:Lim-exact-0}
0 \to {\rm Image}(\partial^{1}) \to \{F^j_{m-1,r}(n)\}_n 
\to \Ker(R^{m-1}) \to 0;
\end{equation}
\begin{equation}\label{eqn:Lim-exact-1}
0 \to \Ker(R^{m-1}) \to \{F^j_{m,r}(n)\}_n \to {\rm Image}(R^{m-1}) \to 0;
\end{equation}
\begin{equation}\label{eqn:Lim-exact-2}
0 \to {\rm Image}(R^{m-1}) \to \{F^j_{1,r}(n)\}_n \to {\rm Image}(\partial^1) 
\to 0.
\end{equation}

\lemref{lem:Finite-0} says that $\{F^j_{1,r}(n)\}_n$ is
a pro-system of finite abelian groups. This implies by 
~\eqref{eqn:Lim-exact-2} that each of 
${\rm Image}(R^{m-1})$ and ${\rm Image}(\partial^1)$ is isomorphic to a
pro-system of finite abelian groups. In particular, the
pro-systems of ~\eqref{eqn:Lim-exact-2} do not admit higher derived
lim functors. On the other hand, ${\varprojlim}^1_n  F^j_{m-1,r}(n) = 0$ by
induction. It follows from ~\eqref{eqn:Lim-exact-0} that
${\varprojlim}^1_n \Ker(R^{m-1}) = 0$.
Using ~\eqref{eqn:Lim-exact-1}, we get ${\varprojlim}^1_n  F^j_{m,r}(n) = 0$.

It follows from what we have shown is that the above three short
exact sequences of pro-abelian groups remain exact after we
apply the inverse limit functor. But this implies that  
the long exact sequence of \corref{cor:log-HW-pro-coh} also remains exact
after applying the inverse limit functor.
This proves the lemma.
\end{proof}

\begin{lem}\label{lem:long-exact-lim}
There is a long exact
sequence 
\[
\cdots \to {\varinjlim}_n {(F^j_{i,r}(n))}^\star \xrightarrow{(R^{m-i})^\star}
{\varinjlim}_n {(F^j_{m,r}(n))}^\star  \xrightarrow{(\un{p}^i)^\star}
{\varinjlim}_n {(F^j_{m-i,r}(n))}^\star  \to \cdots
\]
of abelian groups for every $i \ge 0$.
\end{lem}
\begin{proof}
This is an easy consequence of \corref{cor:log-HW-pro-coh} using the fact
that the Pontryagin dual functor (recalled in \S~\ref{sec:Cont})
is exact on the category of discrete torsion
abelian groups (see \cite[Theorem~2.9.6]{Pro-fin}) and the direct limit 
functor is exact on the category of ind-abelian groups.
We also have to note that the Pontryagin dual of a discrete $p^m$-torsion 
abelian group $V$ coincides with $V^\star$.
\end{proof}

\begin{lem}\label{lem:long-exact-lim-d}
Assume further that  $D \subset X$ is a simple normal crossing 
divisor. 
Then the group
${\varprojlim}_n F^d_{m,d}(n)$ is profinite and the canonical map
\[
{\varinjlim}_n {(F^d_{m,d}(n))}^\star \to
({\varprojlim}_n F^d_{m,d}(n))^\star 
\]
is an isomorphism of topological abelian groups if 
either $k \neq \F_2$ or $d \neq 2$.
\end{lem}
\begin{proof}
By \cite[Theorem~9.1]{Kato-Saito},  \propref{prop:Milnor-iso}, 
 \corref{cor:RS-K-comp} and
\lemref{lem:Rel-dlog-iso}, $\{F^d_{m,d}(n)\}_n$ is isomorphic to
a pro-abelian group $\{E_n\}_n$ such that each $E_n$ is finite and
the map $E_{n+1} \to E_n$ is surjective. We can therefore apply
\lemref{lem:Top}.
\end{proof}

\section{Cartier map for twisted de 
Rham-Witt complex}\label{sec:Cartier}
In this section, we shall prove the existence and some properties of the 
Cartier homomorphism for the twisted de Rham-Witt complex. 
We fix a perfect field $k$ of characteristic $p >0$ and a modulus pair 
$(X,D)$ in $\Sch_k$ such that $X$ is connected and regular 
of dimension $d \ge 1$. Let $\iota \colon D \inj X$
and $j \colon U \inj X$ be the inclusions, where $U = X \setminus D$. 
Let $F$ denote the function field of $X$. 
We also fix integers $m \geq 1$, 
 $n \in \Z$ and $r \geq 0$.
Let $W_m \sO_X(nD) $ be the Nisnevich sheaf on $X$ 
defined in \S~\ref{sec:dRW}. We begin with following result describing the
behavior of the Frobenius and Verschiebung operators on 
$W_\bullet\sO_X(nD)$. We shall then use this result to describe these 
operators on the full twisted de Rham-Witt complex.
Until we talk about the topology of $X$ again in this section, we shall assume 
it to be the Nisnevich topology. 

\begin{lem}\label{lem:F-V-Witt-0}
We have
\begin{enumerate}
\item
$F(W_{m+1}\sO_X(nD)) \subseteq W_{m}\sO_X(pnD)$.
\item
$V(W_{m}\sO_X(pnD)) \subseteq W_{m+1}\sO_X(nD)$.
\item
$R(W_{m+1}\sO_X(nD)) \subseteq W_{m}\sO_X(nD)$.
\end{enumerate}
\end{lem}
\begin{proof}
We can check the lemma locally at a point $x \in X$. We let $A = \sO^h_{X,x}$
and let $f \in A$ define $D$ at $x$. The lemma is easy to check when $n \le 0$.
We therefore assume $n \ge 1$. For $w \in 
W_{m+1}(A)$, we have  $F(w[f^{-n}]) = F(w) \cdot F([f^{-n}]) = F(w) \cdot [f^{-pn}]
\in W_m\sO^h_{X,x}(pnD)$. This proves (1). For (2), we use the projection
formula to get $V(w'[f^{-pn}]) = V(w' \cdot F([f^{-n}])) = V(w') \cdot
[f^{-n}] \in  W_{m+1}\sO^h_{X,x}(nD)$ for $w' \in W_m(A)$. The part (3) is obvious.
\end{proof}

\begin{lem}\label{lem:F-V-Witt-1}
For $i \geq 1$, there is a short exact sequence
\[
0 \to W_{m-i}\sO_X(p^inD) \xrightarrow{V^i} W_m \sO_X(nD) \xrightarrow{R^{m-i}}
W_i\sO_X(nD) \to 0.
\]
\end{lem}
\begin{proof}
Since $R$ is a ring homomorphism on $j_* W_m\sO_U$ on which it is surjective,
it follows that it is surjective on $W_m \sO_X(*D)$.
Since the sequence
\begin{equation}\label{eqn:F-V-Witt-1-0}
0 \to W_{m-i}\sO_X \xrightarrow{V^i} W_m\sO_X \xrightarrow{R^{m-i}} W_i\sO_X \to
0
\end{equation}
is anyway exact, and since $V^i$ and $R$ are defined by 
\lemref{lem:F-V-Witt-0}, we only have to show that 
$V^i(j_*W_{m-i}\sO_U) \cap W_m\sO_X(nD) = V^i(W_{m-i}\sO_X(p^inD))$.
Arguing by induction on $i \ge 1$, it suffices to check this for $i = 1$. 

Now, we let $w = (a_1, \ldots , a_{m-1}) \in j_*W_{m-1}\sO_U$ be such that 
$V(w) = (0, a_1, \ldots , a_{m-1}) = (b_1, \ldots , b_{m}) \cdot [f^{-n}]
= (b_1f^{-n}, b_2f^{-pn}, \ldots , b_{m}f^{-p^{m-1}n})$, where $b_i \in \sO_X$.
This implies that $b_1 = 0$ and $a_i = b_{i+1}f^{-p^{i}n}$ for $i \ge 1$.
Hence, $w = (b_2f^{-pn}, \ldots , b_{m}f^{-p^{m-1}n}) = (b_2, \ldots , b_{m})
\cdot [f^{-pn}] \in W_{m-1}\sO_X(pnD)$.
\end{proof}

We shall now extend \lemref{lem:F-V-Witt-0} to the de Rham-Witt forms of
higher degrees.
One knows using the Gersten conjecture for the de Rham-Witt sheaves
due to Gros \cite{Gross} that the map 
$W_m \Omega^r_X \to W_m \Omega^r_K$ is
injective. It follows that the canonical map
$W_m \Omega^r_X \to j_* W_m \Omega^r_U$ is injective too.
Using the invertibility of the $W_m\sO_X$-module $W_m \sO_X(nD)$,
we get an injection
\begin{equation}\label{eqn:Injection}
\xymatrix@C.8pc{
W_m \Omega^r_X(nD) \ar@{=}[r] \ar@{^{(}->}[dr]  & 
W_m \sO_X(nD) \otimes_{W_m\sO_X} W_m \Omega^r_X 
\ar@{^{(}->}[r]  & W_m \sO_X(nD) \otimes_{W_m\sO_X} j_* W_m \Omega^r_U 
\ar@{=}[d] \\
& j_* W_m \Omega^r_U \ar@{=}[r] & W_m \sO_X(nD) \cdot j_* W_m \Omega^r_U.}
\end{equation}

\begin{lem}\label{lem:Limit}
Let $q \ge 1$ be a positive integer.  Then the inclusions
$W_m \sO_X(qnD) \inj W_m \sO_X(q(n+1)D)$ induce an
isomorphism 
\[
{\underset{n \ge 1}\varinjlim} \ W_m \sO_X(qnD) 
\xrightarrow{\cong} j_* W_m \sO_U.
\]
\end{lem}
\begin{proof}
We shall prove the lemma by induction on $m$.
For $m = 1$, the statement is clear. We now assume $m \ge 2$ and use the
commutative diagram
\begin{equation}\label{eqn:Limit-0}
\xymatrix@C.8pc{
0 \ar[r] & {\underset{n \ge 1}\varinjlim} \ W_{m-1} \sO_X(qpnD) \ar[r]^-{V} 
\ar[d] &
{\underset{n \ge 1}\varinjlim} \ W_{m} \sO_X(qnD) \ar[r]^-{R^{m-1}} \ar[d] &
{\underset{n \ge 1}\varinjlim} \ W_1 \sO_X(qnD) \ar[r] \ar[d] & 0 \\
0 \ar[r] & j_*  W_{m-1} \sO_U \ar[r]^-{V} & j_*  W_{m} \sO_U \ar[r]^-{R^{m-1}}  &
j_*  W_{1} \sO_U \ar[r] & 0.}
\end{equation}

It is easy to check that the bottom row is exact. The top row is
exact by \lemref{lem:F-V-Witt-1}.
The right vertical arrow is an isomorphism by $m =1$ case and the
left vertical arrow is an isomorphism by induction on $m$
(with $q$ replaced by $qp$).
The lemma follows.
\end{proof}

\begin{cor}\label{cor:Limit-1}
For $q \ge 1$, the canonical map
$W_m\Omega^r_X(qnD) \to W_m\Omega^r_X(q(n+1)D)$ is injective and the map
${\underset{n \ge 1}\varinjlim} \ W_m\Omega^r_X(qnD) \to j_*  W_m\Omega^r_U$
is an isomorphism of $W_m\sO_X$-modules.
\end{cor}
\begin{proof}
The injectivity claim follows from ~\eqref{eqn:Injection}.
To prove the isomorphism, we take the tensor product with
the $W_m \sO_X$-module  $W_m \Omega^r_X$
 on the two sides
of the isomorphism in \lemref{lem:Limit}.
This yields an isomorphism
${\underset{n \geq 1}\varinjlim} \ W_m \Omega^r_X(qnD) \cong (j_* W_m \sO_U)
\otimes_{W_m\sO_X}  W_m \Omega^r_X$.
On the other hand, the {\'e}tale descent property (see \cite[Proposition I.1.14]{Illusie})
of the de Rham-Witt sheaves says that the canonical map
$(j_* W_m \sO_U) \otimes_{W_m\sO_X}  W_m \Omega^r_X \to j_* W_m\Omega^r_U$
is an isomorphism. The corollary follows.
\end{proof}

In view of \corref{cor:Limit-1}, we shall hereafter consider 
the sheaves $W_m \Omega^r_X(nD)$ as $W_m \sO_X$-submodules of 
$j_* W_m\Omega^r_U$.

\begin{lem}\label{lem:FRV}
The $V,F$ and $R$ operators of $j_* W_m\Omega^r_U$ satisfy the following.
\begin{enumerate}
\item
$F(W_m \Omega^r_X(nD)) \subseteq W_{m-1} \Omega^r_X(pnD)$.
\item
$V(W_m \Omega^r_X(pnD)) \subseteq W_{m+1} \Omega^r_X(nD)$.
\item
$R(W_m \Omega^r_X(nD)) \subseteq W_m \Omega^r_X(nD)$.
\end{enumerate}
\end{lem}
\begin{proof}
We have $F([f^{-n}]\omega) = F(\omega)
F([f^{-n}]) \in W_{m-1} \Omega^r_X \cdot W_{m-1} \sO_X(pnD) =
W_{m-1} \Omega^r_X(pnD)$, where the first inclusion holds by 
\lemref{lem:F-V-Witt-0}.
We have $V([f^{-pn}]\omega) = V(F([f^{-n}]) \omega) = [f^{-n}] V(\omega)
\in  W_{m+1} \sO_X(nD) \cdot W_{m+1} \Omega^r_X =  W_{m+1} \Omega^r_X(nD)$,
where the second inclusion holds again by \lemref{lem:F-V-Witt-0}.
The last assertion is clear.
\end{proof}

\begin{lem}\label{lem:mult-p}
The multiplication by $p$ 
map $p \colon W_{m+1}\Omega^r_X(nD) \to W_{m+1}\Omega^r_X(nD)$ has a 
factorization
\[
W_{m+1}\Omega^r_X(nD) \stackrel{R}{\surj} W_{m}\Omega^r_X(nD) 
\stackrel{\un{p}}{\inj} W_{m+1}\Omega^r_X(nD)
\]
in the category of sheaves of $W_{m+1}\sO_X$-modules. This factorization
is natural in $X$.
\end{lem}
\begin{proof}
Since $R$ is $W_{m+1}\sO_X$-linear, we get an exact sequence
(see \S~\ref{sec:dRW})
\begin{equation}\label{eqn:mult-p-0}
0 \to (V^{m}W_1\Omega^r_X + dV^{m}W_1\Omega^{r-1}_X) \cdot W_{m+1}\sO_X(nD) \to
W_{m+1}\Omega^r_X(nD) \xrightarrow{R} W_{m} \Omega^r_X(nD) \to 0.
\end{equation}

On the other hand, we note that 
$p \colon  W_{m+1}\Omega^r_X \to W_{m+1}\Omega^r_X$ is also a
$W_{m+1}\sO_X$-linear homomorphism. Since $W_{m+1}\sO_X(nD)$ is an
invertible $W_{m+1} \sO_X$-module, we get an exact sequence
\begin{equation}\label{eqn:mult-p-1}
0 \to {\Ker(p)} \cdot  W_{m+1}\sO_X(nD) \to
W_{m+1}\Omega^r_X(nD) \xrightarrow{p} W_{m+1} \Omega^r_X(nD).
\end{equation}
Since $\Ker(p) = \Ker(R)$ as $X$ is regular (see \S~\ref{sec:dRW}), we get
${\Ker(p)} \cdot  W_{m+1}\sO_X(nD) = 
(V^{m}W_1\Omega^r_X + dV^{m}W_1\Omega^{r-1}_X) \cdot W_{m+1}\sO_X(nD)$.
The first part of the lemma now follows. The naturality is clear from the
proof.
\end{proof}

\begin{lem}\label{lem:Cartesian-X}
The square
\begin{equation}\label{eqn:Cartesian-X-0}
\xymatrix@C.8pc{
W_m\Omega^r_X \ar[r]^-{\un{p}} \ar[d]_-{j^*} & W_{m+1}\Omega^r_X \ar[d]^-{j^*} \\
j_* W_m\Omega^r_U \ar[r]^-{\un{p}} &  j_* W_{m+1}\Omega^r_U}
\end{equation}
is Cartesian.
\end{lem}
\begin{proof}
We can check this locally. So let $x \in j_* W_m\Omega^r_U$ be such that
$\un{p}(x) \in  W_{m+1}\Omega^r_X$. Since $j$ is affine and
$W_m\Omega^r_X$ is an $W_m\sO_X$-module, it follows that
$j_* W_{m+1}\Omega^r_U \xrightarrow{j_* R} j_* W_m\Omega^r_U$ is surjective.
We can therefore find $\tilde{x} \in j_* W_{m+1}\Omega^r_U$ such that
$R(\tilde{x}) = x$. In particular, $p \tilde{x} = \un{p}(x) \in 
W_{m+1}\Omega^r_X$. 
We thus get $VF(\tilde{x}) = p \tilde{x} \in W_{m+1}\Omega^r_X$. Since
\[
(j_* V W_m \Omega^r_U) \cap W_{m+1} \Omega^r_X =
\Ker(F^{m} \colon W_{m+1} \Omega^r_X \to j_* \Omega^r_U) =
\Ker(F^{m} \colon W_{m+1} \Omega^r_X \to \Omega^r_X)
\]
\[
 = VW_m \Omega^r_X,
\]
it follows that there exists $y' \in W_m \Omega^r_X$ such that
$VF(\tilde{x}) = Vy'$.

Since $\Ker(V \colon j_*W_m \Omega^r_U \to j_* W_{m+1} \Omega^r_U) = 
j_* FdV^{m} \Omega^{r-1}_U $
(see \cite[I.3.21.1.4]{Illusie}), it follows that there exists 
$z' \in j_* \Omega^{r-1}_U$ such that
$FdV^m(z') = F(\tilde{x}) - y'$. Equivalently, $F(\tilde{x} - dV^m(z')) = y'
\in W_m \Omega^r_X$.
Since
\[
j_* FW_{m+1} \Omega^r_U \cap W_m \Omega^r_X =
\Ker(F^{m-1}d \colon W_m \Omega^r_X \to j_* \Omega^{r+1}_U) =
\Ker(F^{m-1}d \colon W_m \Omega^r_X \to \Omega^{r+1}_X)
\]
\[
= F W_{m+1} \Omega^r_X,
\]
we can find $y'' \in W_{m+1} \Omega^r_X$ such that $F(\tilde{x} - dV^m(z'))
= F(y'')$.

Since $\Ker(F \colon j_* W_{m+1} \Omega^r_{U} \to j_*  W_m \Omega^r_U)
= j_* V^m \Omega^r_U$ (see \cite[I.3.21.1.2]{Illusie}),
we can find $z'' \in j_* \Omega^r_U$ such that 
$V^m(z'') = \tilde{x} - dV^m(z') - y''$. Equivalently,
$\tilde{x} - y'' = V^m(z'') + dV^m(z')$.
On the other hand, we have
\[
V^m j_* \Omega^r_U  + dV^m j_* \Omega^{r-1}_U =
\Ker(R \colon j_* W_{m+1} \Omega^r_U \to j_* W_m \Omega^r_U).
\]
We thus get $x = R(\tilde{x}) = R(y'')$. Since $y'' \in  W_{m+1} \Omega^r_X$,
we get $x \in W_{m} \Omega^r_X$. This proves the lemma.
\end{proof}

It is easy to see that for $n \in \Z$, the differential 
$d \colon  j_* W_{m}\Omega^r_X \to   j_* W_{m}\Omega^{r+1}_X$ restricts to a
homomorphism $d \colon W_{m}\Omega^r_X(nD) \to W_{m}\Omega^{r+1}_X((n+1)D)$
such that the composite map
$d^2 \colon W_{m}\Omega^r_X(nD) \to W_{m}\Omega^{r+2}_X((n+2)D)$ is zero
by \corref{cor:Limit-1}.
The map $d \colon  W_{m}\Omega^r_X(p^mnD) \to W_{m}\Omega^{r+1}_X((p^mn+1)D)$
actually factors through $d \colon
 W_{m}\Omega^r_X(p^mnD) \to W_{m}\Omega^{r+1}_X(p^mnD)$ as one easily checks.
In particular, $W_m\Omega^\bullet_X(p^mnD)$ is a complex for every $m\ge 1$ and
$n \in \Z$.

Let 
\[
Z_1 W_m\Omega^r_X(nD) = (j_* Z_1 W_m\Omega^r_U) \cap W_m\Omega^r_X(nD)
= j_*\Ker(F^{m-1}d) \cap  W_m\Omega^r_X(nD) 
\]
\[
= \Ker(F^{m-1}d \colon W_m\Omega^r_X(nD) \to j_* \Omega_U) =
\Ker(F^{m-1}d \colon W_m\Omega^r_X(nD) \to \Omega^{r+1}_X(p^{m-1}(n+1)D)),
\]
where the third equality is by the left exactness of $j_*$ and
the last equality by \lemref{lem:FRV}.
If $m =1$, we get
\begin{equation}\label{eqn:Loc-free-0}
Z_1W_1\Omega^r_X(nD) = Z\Omega^r_X(nD) = \Ker(d \colon 
\Omega^r_X(nD) \to \Omega^{r+1}_X((n+1)D)).
\end{equation}
We let $B\Omega^r_X(nD) = {\rm Image}(d \colon \Omega^r_X((n-1)D) \to
\Omega^r_X(nD))$.

\begin{prop}\label{prop:Cartier-map}
There exists a homomorphism $C \colon Z_1 W_m\Omega^r_X(pnD)
\to W_m\Omega^r_X(nD)$ such that the diagram
\begin{equation}\label{eqn:Cartier-map-0}
\xymatrix@C.8pc{
Z_1 W_m\Omega^r_X(pnD) \ar[rr]^-{V} \ar[dr]_-{C} & & W_{m+1}\Omega^r_X(nD) \\
& W_m\Omega^r_X(nD) \ar[ur]_-{\un{p}}}
\end{equation}
is commutative. The map $C$ induces an isomorphism of $\sO_X$-modules
\[
C \colon \sH^r(\psi_* W_1\Omega^\bullet_X(pnD)) \xrightarrow{\cong} 
W_1\Omega^r_X(nD).
\]
\end{prop}
\begin{proof}
We consider the diagram
\begin{equation}\label{eqn:Cartier-map-10}
\xymatrix@C.8pc{
Z_1 W_m\Omega^r_X(pnD) \ar@{.>}[r] \ar[d] \ar@/^1cm/[rr]^-{V} & W_m\Omega^r_X(nD)
\ar[r]^-{\un{p}} \ar[d] & W_{m+1}\Omega^r_X(nD) \ar[d] \\
j_* Z_1 W_m\Omega^r_U \ar[r]^-{C} \ar@/_1cm/[rr]^-V 
& j_* W_m\Omega^r_U \ar[r]^-{\un{p}} & j_* W_{m+1}\Omega^r_U,}
\end{equation}
where the vertical arrows are inclusions from \corref{cor:Limit-1}.
The right square exists and commutes by \lemref{lem:mult-p}.
The big outer square clearly commutes. It suffices therefore to show
that the right square is Cartesian.

We have a commutative diagram of $W_{m+1}\sO_X$-modules
\begin{equation}\label{eqn:Cartier-map-1}
\xymatrix@C.8pc{
W_{m+1}\Omega^r_X(nD) \ar[r]^{p} \ar[d] & W_{m+1}\Omega^r_X(nD) \ar[r] \ar[d]
& {W_{m+1}\Omega^r_X(nD)}/p \ar[r] \ar[d] & 0 \\
j_* W_{m+1}\Omega^r_U \ar[r]^{p} \ar[r] & j_* W_{m+1}\Omega^r_U \ar[r] 
& j_* {W_{m+1}\Omega^r_U}/p \ar[r] & 0.}
\end{equation} 
The top sequence is clearly exact and the bottom sequence is exact
by the Serre vanishing because all sheaves are $W_{m+1}\sO_X$-modules and 
$j$ is an affine morphism. By \lemref{lem:mult-p}, it suffices to show that
the right vertical arrow is injective.

To prove this injectivity, we note that the top row of 
~\eqref{eqn:Cartier-map-1} is same as the sequence
\[
W_{m+1}\Omega^r_X \otimes_{\sO} \sO(nD)
\xrightarrow{p \otimes 1} 
W_{m+1}\Omega^r_X \otimes_{\sO} \sO(nD)
\to {W_{m+1}\Omega^r_X}/p \otimes_{\sO} \sO(nD) \to 0,
\]
where $\sO = W_{m+1}\sO_X$. Similarly, we have
\[
j_* {W_{m+1}\Omega^r_U}/p \otimes_{\sO} \sO(nD) \cong 
j_*({W_{m+1}\Omega^r_U}/p \otimes_{W_{m+1}\sO_U} j^* \sO(nD))
\cong 
\]
\[
j_*({W_{m+1}\Omega^r_U}/p \otimes_{W_{m+1}\sO_U} W_{m+1}\sO_U)
\cong j_* {W_{m+1}\Omega^r_U}/p,
\]
where the first isomorphism follows from the projection formula for
$\sO$-modules using the fact that $\sO(nD)$ is an invertible $\sO$-module
(see \cite[Exercise~II.5.1]{Hartshorne-AG}).

Moreover, it is clear that the right vertical arrow in 
~\eqref{eqn:Cartier-map-1} is
the map 
\[
{W_{m+1}\Omega^r_X}/p \otimes_{\sO} \sO(nD) 
\xrightarrow{j^* \otimes 1} (j_* {W_{m+1}\Omega^r_U}/p) \otimes_{\sO} \sO(nD).
\]
Since $\sO(nD)$ is an invertible $\sO$-module, it suffices to show that
the map $j^* \colon {W_{m+1}\Omega^r_X}/p \to  j_* {W_{m+1}\Omega^r_U}/p$
is injective. But this follows from \lemref{lem:Cartesian-X} since $j^*$
in ~\eqref{eqn:Cartesian-X-0} is injective. This proves the first part of 
the lemma.

We now prove the second part for which we can assume $m =1$.
We know classically that the map $C$ on $Z\Omega^r_X$ induces an
$\sO_X$-linear isomorphism $C \colon \sH^r(\psi_* \Omega^\bullet_X) 
\xrightarrow{\cong} \Omega^r_X$.
Taking its inverse, we get an $\sO_X$-linear isomorphism 
\[
C^{-1} \colon \Omega^r_X \xrightarrow{\cong} \sH^r(\psi_* \Omega^\bullet_X).
\]
Since $\psi_* \Omega^\bullet_X \in D^{+}({\rm Coh}_X)$, we see that
$\sH^r(\psi_* \Omega^\bullet_X)$ is a coherent $\sO_X$-module.
In particular, we get an isomorphism
\[
C^{-1} \colon \Omega^r_X(nD) \xrightarrow{\cong} 
\sH^r(\psi_* \Omega^\bullet_X)(nD).
\]

On the other hand, we have
\[
\sH^r(\psi_* \Omega^\bullet_X)(nD) \cong \sH^r((\psi_* \Omega^\bullet_X)(nD))
\cong \sH^r(\psi_*(\Omega^\bullet_X \otimes_{\sO_X} \psi^*(\sO(nD)))) 
\cong \sH^r(\psi_*(\Omega^\bullet_X (pnD))).
\]
This proves the second part.
\end{proof}

\begin{lem}\label{lem:Z_1-mod}
We have the following.
\begin{enumerate}
\item
$\Ker(F^{m-1} \colon W_{m} \Omega^r_X(nD) \to \Omega^r_X(p^{m-1}nD)) =
V W_{m-1} \Omega^r_X(pnD)$.
\item
$Z_1 W_m\Omega^r_X(pnD) = F W_{m+1} \Omega^r_X(nD)$.
\end{enumerate}
\end{lem}
\begin{proof}
We first prove (1). It is clear that the right hand side is contained in the
left hand side. We prove the other inclusion.
It suffices to show that
\[
(j_* V W_{m-1} \Omega^r_U) \cap W_{m} \Omega^r_X(nD) \subset  
V W_{m-1} \Omega^r_X(pnD).
\]

We can check this locally. So let $D$ be defined by $f \in \sO_X$ and
let $y = Vx = [f^{-n}]\omega$,
where $x \in j_* W_{m-1} \Omega^r_U$ and $\omega \in  W_{m} \Omega^r_X$.
This yields $\omega = [f^n] Vx = V(F([f^n]) x) = V([f^{pn}]x)$.
This implies that $\omega \in W_{m} \Omega^r_X \cap j_* V W_{m-1} \Omega^r_U$.
On the other hand, we have
\[
W_{m} \Omega^r_X \cap j_* V W_{m-1} \Omega^r_U =
\Ker(F^{m-1} \colon W_{m} \Omega^r_X \to j_*\Omega^r_U) =
\Ker(F^{m-1} \colon W_{m} \Omega^r_X \to \Omega^r_X)
\]
\[
= V W_{m-1}\Omega^r_X.
\]
We thus get $\omega \in V W_{m-1}\Omega^r_X$. Let $y' \in W_{m-1}\Omega^r_X$
be such that $\omega = Vy'$. This yields
$y = [f^{-n}]\omega = [f^{-n}]Vy' = V(F([f^{-n}])y') = V([f^{-pn}]y')
\in VW_{m-1}\Omega^r_X(pnD)$. This proves (1).

We now prove (2). Since
$FW_{m+1} \Omega^r_X(nD) \subset j_* F W_{m+1} \Omega^r_U \cap
W_m\Omega^r_X(pnD)$ by \lemref{lem:FRV}, it follows that
$FW_{m+1} \Omega^r_X(nD) \subset Z_1 W_m \Omega^r_X(pnD)$.
We show the other inclusion.

We let $z \in Z_1 W_m \Omega^r_X(pnD)$ so that
$z = Fx = [f^{-pn}]\omega$ for some $x \in j_* W_{m+1}\Omega^r_U$ and
$\omega \in W_m\Omega^r_X$. 
We can then write $\omega = [f^{pn}]Fx = F([f^n]x)$.
This implies that $\omega \in W_m\Omega^r_X \cap j_* F W_{m+1} \Omega^r_U$.
On the other hand, we have
\[
W_m\Omega^r_X \cap j_* F W_{m+1} \Omega^r_U = 
\Ker(F^{m-1} d \colon W_m\Omega^r_X \to j_* \Omega^{r+1}_U)
= \Ker(F^{m-1} d \colon W_m\Omega^r_X \to \Omega^{r+1}_X)
\]
\[
= F W_{m+1} \Omega^{r}_X.
\]
We can thus write  $\omega = Fx'$ for some $x' \in W_{m+1} \Omega^{r}_X$. 
This gives $z = [f^{-pn}]\omega = [f^{-pn}]Fx' = F([f^{-n}]) Fx' =
F([f^{-n}]x') \in F W_{m+1}\Omega^r_X(nD)$. This proves (2).
\end{proof}

\section{The pairing of cohomology groups}\label{sec:DT}
We fix a finite field $k$ of characteristic $p$ 
and an integral smooth projective scheme $X$ 
of dimension $d \ge 1$ over $k$. Let $D \subset X$ be an effective Cartier
divisor with complement $U$. Let $\iota \colon D \inj X$ and 
$j \colon U \inj X$ be the inclusions.
In this section, we shall establish the pairing for
our duality theorem for the $p$-adic {\'e}tale 
cohomology of $U$. We fix integers $m, n \ge 1$ and $r \ge 0$.
We shall consider the {\'e}tale topology throughout our discussion of 
duality.

\subsection{The complexes}\label{sec:Complexes}
We consider the complex of {\'e}tale sheaves 
\begin{equation}\label{eqn:F-com}
W_m\sF^{r, \bullet}_{n} = [Z_1W_m\Omega^r_{X}(nD) \xrightarrow{1 - C}
W_m\Omega^r_{X}(nD)].
\end{equation}
The differential of this complex is induced by the composition
\[
Z_1W_m\Omega^r_{X}(nD) \inj Z_1W_m\Omega^r_{X}(pnD) \xrightarrow{C}
W_m\Omega^r_{X}(nD),
\]
where the last map is defined by virtue of \propref{prop:Cartier-map}.

We now consider the map $F \colon W_{m+1}\Omega^r_X(-nD) \to W_m\Omega^r_X(-pnD)$
whose existence is shown in \lemref{lem:FRV}.
We have shown in ~\eqref{eqn:dlog-2-*} that
$F(\Ker(R)) = F((V^mW_1\Omega^r_X + dV^mW_1\Omega^{r-1}_X) \cap
W_{m+1}\Omega^r_X(-nD)) \subseteq dV^{m-1}\Omega^{r-1}_X \cap 
W_m\Omega^r_X(-pnD)$. It follows that
$F$ induces a map
$\ov{F} \colon W_{m}\Omega^r_X(-nD) \to 
{W_m\Omega^r_X(-pnD)}/{(dV^{m-1}\Omega^{r-1}_X \cap 
W_m\Omega^r_X(-pnD))}$.
We denote the composition
\[
W_{m}\Omega^r_X(-nD) \to 
\frac{W_m\Omega^r_X(-pnD)}{dV^{m-1}\Omega^{r-1}_X \cap 
W_m\Omega^r_X(-pnD)} \inj \frac{W_m\Omega^r_X(-nD)}{dV^{m-1}\Omega^{r-1}_X \cap 
W_m\Omega^r_X(-nD)}
\]
also by $\ov{F}$ and consider the complex of {\'e}tale sheaves 
\begin{equation}\label{eqn:G-com}
W_m\sG^{r, \bullet}_{n} = \left[W_{m}\Omega^r_X(-nD) \xrightarrow{\pi - \ov{F}}
\frac{W_m\Omega^r_X(-nD)}{dV^{m-1}\Omega^{r-1}_X \cap 
W_m\Omega^r_X(-nD)}\right],
\end{equation}
where $\pi$ is the quotient map.
We let
\begin{equation}\label{eqn:H-com}
W_m\sH^{d,\bullet} = [W_m\Omega^d_X \xrightarrow{1 -C} W_m\Omega^d_X],
\end{equation}
where the map $C$ (see \propref{prop:Cartier-map} for $n = 0$) is defined 
because $Z_1W_m\Omega^d_X = W_m\Omega^d_X$.

\subsection{The pairing of complexes}\label{sec:pair}
We consider the pairing
\begin{equation}\label{eqn:Pair-0}
\<, \>_1 \colon Z_1W_m\Omega^r_{X}(nD) \times W_{m}\Omega^{d-r}_X(-nD) 
\xrightarrow{\wedge} W_m\Omega^d_X
\end{equation}
by letting $\<w_1, w_2\>_1 = w_1 \wedge w_2$.
This is defined by the definition of the twisted de Rham-Witt sheaves.

We define a pairing
$\<, \>_2 \colon Z_1W_m\Omega^r_{X}(nD) \times W_{m}\Omega^{d-r}_X(-nD) 
\to W_m\Omega^d_X$ by $\<w_1, w_2\>_2 = - C(w_1 \wedge w_2)$. 
We claim that $C(w_1 \wedge w_2) = 0$ if
$w_2 \in dV^{m-1}\Omega^{d-r-1}_X \cap W_m\Omega^{d-r}_X(-nD)$.
To prove it, we write $w_1 = F( w'_1)$ for some $w'_1 \in 
j_* W_{m+1} \Omega^r_U$. 
This gives
$V(w_1 \wedge j^* w_2) = V(F(w'_1) \wedge  j^* w_2) = w'_1 \wedge  j^* V w_2 = 0$, where 
the last equality holds because $V d V^{m-1} \Omega_X^{d-r-1} \subset 
p d V^{m} \Omega_X^{d-r-1}$ and $p \Omega_X^{d-r-1} = 0$.
In particular, 
$j^*(w_1 \wedge w_2) \in \Ker (V: W_m \Omega_U^{d} \to W_{m+1} \Omega_U^d) = 
Fd V^m \Omega_U^{d-1} = d V^{m-1} \Omega_U^{d-1} = \Ker(C_U)$,
where $C_U$ is the Cartier map for $U$
(see \cite[Chapitre~III]{I-R-83} for the last equality). 
It follows that $j^* \circ C(w_1 \wedge w_2) = 
C_U \circ j^*(w_1 \wedge w_2) = 0$. 
Since $W_m \Omega^d_X \inj j_* W_m \Omega^d_U$, the claim follows.
Using the claim, we get a pairing
\begin{equation}\label{eqn:Pair-1}
\<, \>_2 \colon Z_1W_m\Omega^r_{X}(nD) \times
\frac{W_m\Omega^{d-r}_X(-nD)}{dV^{m-1}\Omega^{d-r-1}_X \cap 
W_m\Omega^{d-r}_X(-nD)} \to W_m\Omega^d_X.
\end{equation}

We define our third pairing of {\'e}tale sheaves
\begin{equation}\label{eqn:Pair-2}
\<, \>_3 \colon W_m\Omega^r_{X}(nD) \times W_m\Omega^{d-r}_X(-nD) \to
W_m\Omega^d_X
\end{equation}
by $\<w_1, w_2\>_3 = w_1 \wedge w_2$. 

\begin{lem}\label{lem:Pairing-main-sheaf}
The above pairings of {\'e}tale sheaves give rise to a pairing of
two-term complexes of sheaves 
\begin{equation}\label{eqn:Pair-3}
\<, \> \colon W_m\sF^{r, \bullet}_{n} \times W_m\sG^{d-r, \bullet}_{n} \to
W_m\sH^{d, \bullet}.
\end{equation}
\end{lem}
\begin{proof}
By \cite[\S~1, p.175]{Milne-Duality}, we only have to show that 
\begin{equation}\label{eqn:Pair-3-0}
(1-C)(w_1 \wedge w_2) = (1-C)(w_1) \wedge w_2 - C(w_1 \wedge (\pi - \ov{F})w_2)
\end{equation}
for all $w_1 \in Z_1W_m\Omega^r_{X}(nD)$ 
and $w_2 \in W_m\Omega^r_{X}(-nD)$. 
But this follows from the equalities
\[
\begin{array}{lll}
(1-C)(w_1 \wedge w_2) & = & w_1 \wedge w_2 - C(w_1 \wedge w_2) \\
& = & w_1 \wedge w_2 - C(w_1) \wedge C(w_2) \\
& = & w_1 \wedge w_2 - C(w_1) \wedge w_2 + C(w_1) \wedge w_2 - 
C(w_1) \wedge C(w_2) \\
& = & (1-C)(w_1) \wedge w_2 - C(w_1) \wedge (C-1)(w_2) \\
& {=}^\dagger & (1-C)(w_1) \wedge w_2 - C(w_1 \wedge (\pi - \ov{F})(w_2)),
\end{array}
\]
where ${=}^\dagger$ holds because $C \circ \ov{F} = {\rm id}$
(see \cite[Proposition~1.1.4]{Zhau}). 
\end{proof}

\subsection{The pairing of {\'e}tale cohomology}\label{sec:Coh-pair}
Using \lemref{lem:Pairing-main-sheaf}, we get a pairing of hypercohomology
groups
\begin{equation}\label{eqn:Pair-4-0}
\H^i_\etl(X, W_m\sF^{r, \bullet}_{n}) \times 
\H^{d+1-i}_\etl(X, W_m\sG^{d-r, \bullet}_{n}) \to 
\H^{d+1}_\etl(X, W_m\sH^{d, \bullet}).
\end{equation}

By \cite[Proposition~1.1.7]{Zhau} and \cite[Corollary~1.12]{Milne-Zeta}, 
there is a quasi-isomorphism
$W_m\Omega^d_{X, \log} \xrightarrow{\cong} W_m\sH^{d, \bullet}$, and a
bijective trace homomorphism
$\Tr \colon H^{d+1}_\etl(X, W_m\Omega^d_{X, \log}) \xrightarrow{\cong} 
{\Z}/{p^m}$. We thus get a pairing of ${\Z}/{p^m}$-modules
\begin{equation}\label{eqn:Pair-4-1}
\H^i_\etl(X, W_m\sF^{r, \bullet}_{n}) \times 
\H^{d+1-i}_\etl(X, W_m\sG^{d-r, \bullet}_{n}) \to {\Z}/{p^m}.
\end{equation}
Since this pairing is compatible with the change in values of $m$ and $n$, we
get a pairing of ind-abelian (in first coordinate) and pro-abelian 
(in second coordinate) groups
\begin{equation}\label{eqn:Pair-4-2}
\{\H^i_\etl(X, W_m\sF^{r, \bullet}_{n})\}_n \times 
\{\H^{d+1-i}_\etl(X, W_m\sG^{d-r, \bullet}_{n})\}_n \to {\Z}/{p^m}.
\end{equation}

It follows from \corref{cor:Limit-1} that
${\varinjlim}_n \ W_m\sF^{r, \bullet}_{n} 
\xrightarrow{\cong} j_*([Z_1W_m\Omega^r_{U} \xrightarrow{1 - C}
W_m\Omega^r_{U}])$.
Since $j$ is affine, we get
\begin{equation}\label{eqn:Pair-4-3}
{\varinjlim}_n  \ W_m\sF^{r, \bullet}_{n} 
\xrightarrow{\cong} 
{\bf R}j_*([Z_1W_m\Omega^r_{U} \xrightarrow{1 - C}
W_m\Omega^r_{U}]) \cong {\bf R}j_*(W_m\Omega^r_{U, \log}).
\end{equation}

In order to understand the pro-complex $\{W_m\sG^{r, \bullet}_{n}\}_n$,
we let 
\[
W_m\wt{\sG}^{r, \bullet}_{n} =
\left[W_m\Omega^r_{(X,nD)} \xrightarrow{\pi - \ov{F}} 
\frac{W_m\Omega^r_{(X,nD)}}{dV^{m-1}\Omega^{r-1}_X \cap W_m\Omega^r_{(X,nD)}}\right].
\]
This complex is defined by the exact sequence
~\eqref{eqn:dlog-2}.
It follows then from \corref{cor:pro-iso-rel*} that
the canonical inclusion
$\{W_m\sG^{r, \bullet}_{n}\}_n \inj \{W_m\wt{\sG}^{r, \bullet}_{n}\}_n$
is an isomorphism of pro-complexes of sheaves.
Using the exact sequence ~\eqref{eqn:dlog-2} and \lemref{lem:Rel-dlog-iso},
we get that there is a canonical isomorphism of pro-complexes
\begin{equation}\label{eqn:Pair-4-4}
\dlog \colon \{\ov{{\wh{\sK}^M_{r, (X, nD)}}/{p^m}}\}_n \xrightarrow{\cong}
\{W_m\Omega^r_{(X,nD), \log}\}_n \xrightarrow{\cong} \{W_m\sG^{r, \bullet}_{n}\}_n.
\end{equation}
We conclude that ~\eqref{eqn:Pair-4-2} is equivalent to the pairing
of ind-abelian (in first coordinate) and pro-abelian 
(in second coordinate) groups
\begin{equation}\label{eqn:Pair-4-5}
\{\H^i_\etl(X, W_m\sF^{r, \bullet}_{n})\}_n \times
\{H^{d+1-i}_\etl(X, W_m\Omega^{d-r}_{(X,nD), \log})\}_n \to {\Z}/{p^m}.
\end{equation}

\subsection{Continuity of the pairing}\label{sec:Cont}
Recall from \S~\ref{sec:Exact}
that for a ${\Z}/{p^m}$-module $V$, one has $V^\star = 
\Hom_{{\Z}/{p^m}}(V, {\Z}/{p^m})$. 
For any profinite or discrete torsion abelian group
$A$, let $A^\vee = \Hom_{\cont}(A, {\Q}/{\Z})$ denote the Pontryagin dual of 
$A$ with the compact-open topology (see \cite[\S~7.4]{Gupta-Krishna-REC}).
If $A$ is discrete and $p^m$-torsion, then we have
$A^\vee = A^\star$. 
We shall use the following topological results.

\begin{lem}\label{lem:lim-colim}
Let $\{A_n\}_n$ be a direct system of discrete torsion 
topological abelian groups
whose limit is $A$ with the direct limit topology. Then
the canonical map $\lambda \colon A^\vee \to {\varprojlim}_n A^\vee_n$
is an isomorphism of profinite groups.
\end{lem}
\begin{proof}
The fact that $\lambda$ is an isomorphism of abelian groups is
well known and elementary.
We show that $\lambda$ is a homeomorphism. Since its source and targets
are compact Hausdorff, it suffices to show that it is continuous.
We let $B_n = {\rm Image}(A_n \to A)$.
Then, we have a surjection of finite ind-groups $\{A_n\}_n \surj \{B_n\}_n$
whose limits coincides with $A$.
We thus get maps $A^\vee \xrightarrow{\lambda'} {\varprojlim}_n B^\vee_n 
\xrightarrow{\lambda''} {\varprojlim}_n A^\vee_n$ such that
$\lambda = \lambda'' \circ \lambda'$.
Since $\lambda''$ is clearly continuous, we have to only show that
$\lambda'$ is continuous. We can therefore assume that
each $A_n$ is a subgroup of $A$. Lemma now 
follows from \cite[Lemma~2.9.3 and Theorem~2.9.6]{Pro-fin}. 
\end{proof}

\begin{lem}\label{lem:Top}
Let $\{A_n\}_n$ be an inverse system of discrete  topological
abelian groups. Let $A$ be the limit of $\{A_n\}_n$ with the inverse
limit topology. Then any $f \in A^\vee$ factors through the
projection $\lambda_n \colon A \to A_n$ for some $n \ge 1$.
In particular, the canonical map
${\varinjlim}_n A^\vee_n \xrightarrow{\eta} A^\vee$ is continuous and surjective.
This map is an isomorphism if the transition maps of $\{A_n\}_n$
are surjective.
\end{lem}
\begin{proof}
Since $f \colon A \to {\Q}/{\Z}$ is continuous and its target is discrete,
it follows that $\Ker(f)$ is open.  By the definition of the inverse limit
topology, the latter contains $\Ker(\lambda_n)$ for some $n \ge 1$.
Letting $A'_n = {A}/{\Ker(\lambda_n)}$, this implies that $f$ factors through 
$f'_n \colon A'_n \to {\Q}/{\Z}$. Since $A_n$ is discrete, the map
$(A_n)^\vee \to (A'_n)^\vee$ is surjective by 
\cite[Lemma~7.10]{Gupta-Krishna-REC}. Choosing a lift $f_n$ of $f'_n$
under this surjection, we see that $f$ factors through
$A \xrightarrow{\lambda_n} A_n \xrightarrow{f_n} {\Q}/{\Z}$.
The continuity of $\eta$ is equivalent to the assertion that
the map $A^\vee_n \to A^\vee$ is continuous for each $n$.
But this is well known since $\lambda_n$ is continuous. 
The remaining parts of the lemma are now obvious.
\end{proof}

\begin{remk}\label{remk:Top-0}
The reader should note that the map ${\varinjlim}_n A^\vee_n \to A^\vee$
may not in general be injective even if each $A_n$ is finite.
\end{remk}

\vskip .3cm

For $n \ge 1$, we endow each of $\H^i_\etl(X, W_m\sF^{r, \bullet}_{n})$ and
$H^{i}_\etl(X, W_m\Omega^{d-r}_{(X,nD), \log})$ with the discrete topology.
We endow ${\varprojlim}_n H^{i}_\etl(X, W_m\Omega^{d-r}_{(X,nD), \log})$
with the inverse limit topology and 
${\varinjlim}_n \H^i_\etl(X, W_m\sF^{r, \bullet}_{n}) \cong 
H^i_\etl(U, W_m\Omega^r_{U, \log})$ with the direct limit topology.
Note that the latter topology is discrete.

If we let $x \in {\varinjlim}_n \H^i_\etl(X, W_m\sF^{r, \bullet}_{n})$,
then we can find some $n \gg 0$ such that
$x = f_n(x')$ for some $x' \in \H^i_\etl(X, W_m\sF^{r, \bullet}_{n})$,
where $f_n \colon \H^i_\etl(X, W_m\sF^{r, \bullet}_{n}) \to
{\varinjlim}_n \H^i_\etl(X, W_m\sF^{r, \bullet}_{n})$ is the canonical map 
to the limit.
This gives a map
\begin{equation}\label{eqn:Pair-4-6}
\<x, \cdot\> \colon {\varprojlim}_n  
H^{d+1-i}_\etl(X, W_m\Omega^{d-r}_{(X,nD), \log}) \to {\Z}/{p^m}
\end{equation}
which sends $y$ to $\<x', \pi(y)\>$ under the pairing ~\eqref{eqn:Pair-4-1},
where $\pi$ is the composite map
\[
{\varprojlim}_n H^{d+1-i}_\etl(X, W_m\Omega^{d-r}_{(X,nD), \log})
\cong {\varprojlim}_n  \H^{d+1-i}_\etl(X, W_m\sG^{d-r, \bullet}_{n})
\xrightarrow{g_n} \H^{d+1-i}_\etl(X, W_m\sG^{d-r, \bullet}_{n}).
\]
One checks that ~\eqref{eqn:Pair-4-6} is defined.
Since $H^i_\etl(U, W_m\Omega^r_{U, \log})^\vee$ is profinite, 
this shows that the map (see \lemref{lem:lim-colim})
\begin{equation}\label{eqn:Pair-5-0}
\theta_m \colon {\varprojlim}_n  
H^{d+1-i}_\etl(X, W_m\Omega^{d-r}_{(X,nD), \log}) \to
H^i_\etl(U, W_m\Omega^r_{U, \log})^\vee
\end{equation}
is continuous.
Since $H^i_\etl(U, W_m\Omega^r_{U, \log})$ is discrete, the map
\begin{equation}\label{eqn:Pair-5-1}
\vartheta_m \colon H^i_\etl(U, W_m\Omega^r_{U, \log}) \to 
({\varprojlim}_n  H^{d+1-i}_\etl(X, W_m\Omega^{d-r}_{(X,nD), \log}))^\vee
\end{equation}
is clearly continuous. We have thus shown that after taking the limits,
~\eqref{eqn:Pair-4-5} gives rise to a 
continuous pairing of topological abelian groups
\[
H^i_\etl(U, W_m\Omega^r_{U, \log}) \times {\varprojlim}_n \
H^{d+1-i}_\etl(X, W_m\Omega^{d-r}_{(X,nD), \log}) \to {\Q}/{\Z}.
\]
Since each $H^{d+1-i}_\etl(X, W_m\Omega^{d-r}_{(X,nD), \log})$ is $p^m$-torsion
and discrete, we get 
the following.

\begin{prop}\label{prop:Cont-pair}
There is a continuous pairing of topological abelian groups
\begin{equation}\label{eqn:Pair-4-7}
H^i_\etl(U, W_m\Omega^r_{U, \log}) \times {\varprojlim}_n \
H^{d+1-i}_\etl(X, W_m\Omega^{d-r}_{(X,nD), \log}) \to {\Z}/{p^m}.
\end{equation}
\end{prop}

Equivalently, there is a continuous pairing of topological abelian groups
\begin{equation}\label{eqn:Pair-4-8}
H^i_\etl(U, W_m\Omega^r_{U, \log}) \times {\varprojlim}_n \
H^{d+1-i}_\etl(X, \ov{{\wh{\sK}^M_{d-r, (X, nD)}}/{p^m}}) \to {\Z}/{p^m}.
\end{equation}
These pairings are compatible with respect to the canonical surjection
${\Z}/{p^m} \surj {\Z}/{p^{m-1}}$ and the inclusion
${\Z}/{p^{m-1}} \stackrel{p}{\inj} {\Z}/{p^m}$.

It follows from ~\eqref{eqn:Pair-4-5} that the map $\vartheta_m$
in ~\eqref{eqn:Pair-5-1} has a factorization
\begin{equation}\label{eqn:Pair-5-2}
H^i_\etl(U, W_m\Omega^r_{U, \log}) \xrightarrow{\theta'_m}
{\varinjlim}_n H^{d+1-i}_\etl(X, W_m\Omega^{d-r}_{(X,nD), \log})^\vee
\stackrel{\vartheta'_m}{\surj} \hspace*{3cm}
\end{equation}
\[
\hspace*{8cm}
({\varprojlim}_n  H^{d+1-i}_\etl(X, W_m\Omega^{d-r}_{(X,nD), \log}))^\vee,
\]
where $\vartheta'_m$ is the canonical map. This map is continuous and surjective
by \lemref{lem:Top}. The map $\theta'_m$ is continuous since
its source is discrete.
Our goal in the next section will be to show that the maps
$\theta_m$ and $\theta'_m$ are isomorphisms.

\section{Perfectness of the pairing}\label{sec:Perfectness}
We continue with the setting of \S~\ref{sec:DT}.
Our goal in this section is to prove a perfectness statement for the pairing
~\eqref{eqn:Pair-4-5}.  To make our statement precise, it is
convenient to have the following definition.

\begin{defn}\label{defn:PF-lim-colim}
Let $\{A_n\}_n$ and $\{B_n\}_n$ be ind-system and pro-system
of $p^m$-torsion discrete abelian groups, respectively.
Suppose that there is a pairing 
\[
(\star) \hspace*{1cm}  \{A_n\}_n \times \{B_n\}_n \to {\Z}/{p^m}.
\]
We shall say that this pairing is continuous and semi-perfect if the induced
maps
\[
\theta \colon {\varprojlim}_n  B_n \to ({\varinjlim}_n A_n)^\vee
\ \ {\rm and} \ \
\theta' \colon {\varinjlim}_n  A_n \to {\varinjlim}_n B^\vee_n
\]
are continuous and bijective homomorphisms of topological abelian groups.
Recall here
that $ ({\varinjlim}_n A_n)^\vee \cong {\varprojlim}_n A^\vee_n$
by \lemref{lem:lim-colim}. 
We shall say that $(\star)$ is perfect if it is semi-perfect,
the surjective map (see \lemref{lem:Top})
${\varinjlim}_n B^\vee_n \to ( {\varprojlim}_n  B_n)^\vee$
is bijective and $\theta$ is a homeomorphism.
Note that perfectness implies that $\theta'$ is also a homeomorphism.
\end{defn}

The following is easy to check.

\begin{lem}\label{lem:Semi-perfect*}
The pairing $(\star)$ is perfect if it is semi-perfect and 
$\{B_n\}_n$ is isomorphic to a surjective inverse system of compact groups.
\end{lem}

Our goal is to show that ~\eqref{eqn:Pair-4-5} is semi-perfect by
induction on $m \ge 1$. 
We first consider the case $m =1$.
We shall prove this case using our earlier observation that $W_m\Omega^\bullet_X(p^mnD)$ is a 
complex for every $m \ge 1$ and $n \in \Z$.

In this case, we have $Z_1\Omega^r_X(pnD) =
\Ker(d \colon \Omega^r_X(pnD) \to \Omega^{r+1}_X(pnD))$ by 
~\eqref{eqn:Loc-free-0}. We claim that the inclusion
$d\Omega^{r-1}_X(-pnD) \inj \Omega^r_X(-pnD) \cap d\Omega^{r-1}_X$ is actually
a bijection when $n \ge 1$. Equivalently, the map
${\Omega^r_X(-pnD)}/{d\Omega^{r-1}_X(-pnD)} \to {\Omega^r_X}/{d\Omega^{r-1}_X}$
is injective. Since $\psi$ is identity on the topological space 
$X$, the latter injectivity is
equivalent to showing that the map
\begin{equation}\label{eqn:Pair-1-psi}
\frac{(\psi_* \Omega^r_X)(-nD)}{d (\psi_* \Omega^{r-1}_X)(-nD)} \cong
\frac{\psi_* \Omega^r_X(-pnD)}{d \psi_* \Omega^{r-1}_X(-pnD)} \to
\frac{\psi_* \Omega^r_X}{d \psi_* \Omega^{r-1}_X}
\end{equation}
is injective. Using a snake lemma argument, it suffices to show that
the map $\psi_* Z_1\Omega^{r-1}_X \to  \psi_* Z_1\Omega^{r-1}_X \otimes_{\sO_X}
\sO_{nD}$ is surjective. But this is obvious.

\begin{lem}\label{lem:LFPP}
For $n \in \Z$, the sheaves
$\psi_* (Z_1\Omega^r_X(pnD))$ and 
$\psi_*({\Omega^r_X(pnD)}/{d\Omega^{r-1}_X(pnD)})$
are locally free $\sO_X$-modules. For $n \ge 1$, the map
\[
\psi_* (Z_1\Omega^r_X(pnD)) \to 
\sHom_{\sO_X}(\psi_*({\Omega^{d-r}_X(-pnD)}/{d\Omega^{d-r-1}_X(-pnD)}), \Omega^d_X),
\]
induced by ~\eqref{eqn:Pair-1}, is an isomorphism.
\end{lem}
\begin{proof}
We first note that
\begin{equation}\label{eqn:LFPP-0}
\begin{array}{lll}
\psi_* (Z_1\Omega^r_X(pnD)) & \cong & 
\psi_*(\Ker(\Omega^r_X(pnD) \to \Omega^{r+1}_X(pnD))) \\
& \cong & \Ker(\psi_*\Omega^r_X(pnD) \to \psi_*\Omega^{r+1}_X(pnD)) \\
& \cong & \Ker((\psi_*\Omega^r_X)(nD) \to (\psi_*\Omega^{r+1}_X)(nD)) \\
& \cong & (\Ker(\psi_*\Omega^r_X \to \psi_*\Omega^{r+1}_X))(nD) \\
& \cong & (\psi_*Z_1\Omega^r_X)(nD).
\end{array}
\end{equation}
Since $\psi_*$ also commutes with ${\rm Coker}(d)$, we similarly get
\begin{equation}\label{eqn:LFPP-1}
\psi_*({\Omega^r_X(pnD)}/{d\Omega^{r-1}_X(pnD)}) \cong
({\psi_* \Omega^{r}_X}/{d\psi_* \Omega^{r-1}_X})(nD).
\end{equation}

In the exact sequence
\[
0 \to \psi_* Z_1\Omega^r_X \to \psi_* \Omega^r_X \xrightarrow{d}
\psi_* \Omega^{r+1}_X \to {\psi_* \Omega^{r+1}_X}/{d\psi_* \Omega^{r}_X} \to 0
\]
of $\sO_X$-linear maps between coherent $\sO_X$-modules, 
all terms are locally free by \cite[Lemma~1.7]{Milne-Duality}.
It therefore remains an exact sequence of locally free $\sO_X$-modules after
tensoring with $\sO_X(nD)$ for any $n \in \Z$.
The first part of the lemma now follows by using
~\eqref{eqn:LFPP-0} and ~\eqref{eqn:LFPP-1} (for different values of $r$).

To prove the second part, we can again use 
~\eqref{eqn:LFPP-0} and ~\eqref{eqn:LFPP-1}. Since $\sO_X(nD)$ is
invertible, it suffices to show that
$(\psi_* Z_1\Omega^r_X)(nD) \to 
\sHom_{\sO_X}({\psi_* \Omega^{d-r}_X}/{d\psi_* \Omega^{d-r-1}_X}, \Omega^d_X(nD))$
is an isomorphism. But the term on the right side of this map is 
isomorphic to the sheaf 
$\sHom_{\sO_X}({\psi_* \Omega^{d-r}_X}/{d\psi_* \Omega^{d-r-1}_X}, \Omega^d_X)(nD)$
via the canonical isomorphism 
\[
\sHom_{\sO_X}(\sA, \sB) \otimes_{\sO_X} \sE \xrightarrow{\cong} 
\sHom_{\sO_X}(\sA, \sB\otimes_{\sO_X} \sE),
\]
which locally sends $(f \otimes b) \to (a \mapsto f(a) \otimes b)$
if $\sE$ is locally free.
We therefore have to only show that the map
$\psi_* Z_1\Omega^r_X \to 
\sHom_{\sO_X}({\psi_* \Omega^{d-r}_X}/{d\psi_* \Omega^{d-r-1}_X}, \Omega^d_X)$
is an isomorphism. But this follows from \cite[Lemma~1.7]{Milne-Duality}.
\end{proof}

\begin{lem}\label{lem:Finite-0}
The groups $\H^i_\etl(X, W_1\sF^{r, \bullet}_{n})$ and
$H^i_{\etl}(X, \Omega^r_{(X,nD), \log})$
are finite for $i, r \ge 0$.
\end{lem}
\begin{proof}
Using ~\eqref{eqn:dlog-2}, the finiteness of 
$H^i_{\etl}(X, \Omega^r_{(X,nD), \log})$ is reduced to showing that
the group
$H^i_{\etl}(X, \frac{\Omega^r_{(X,nD)}}{\Omega^r_{(X,nD)} \cap d\Omega^{r-1}_{X}})$
is finite. By \eqref{eqn:dlog-2-*}, it suffices to show that 
$H^i_{\etl}(Z, \frac{\Omega^r_Z}{d\Omega^{r-1}_Z})$ is finite if
$Z \in \{X, nD\}$. Since $k$ is perfect, the absolute Frobenius $\psi$ is 
a finite morphism. In particular, $\psi_*$ is exact.
It suffices therefore to show that
$H^i_{\etl}(Z, \frac{\psi_*\Omega^r_Z}{\psi_* d\Omega^{r-1}_Z})$ is finite.
But this is clear because $k$ is finite, $Z$ is projective over $k$ and
${\psi_*\Omega^r_Z}/{\psi_* d\Omega^{r-1}_Z}$ is coherent.
The finiteness of $\H^i_\etl(X, W_1\sF^{r, \bullet}_{n})$
follows by a similar argument because $\psi_*(Z_1\Omega^r_X(nD))$ is coherent. 
\end{proof}

\begin{lem}\label{lem:PP-pro-1}
The pairing
\[
\{\H^i_\etl(X, W_1\sF^{r, \bullet}_{n})\}_n \times 
\{\H^{d+1-i}_\etl(X, W_1\sG^{d-r, \bullet}_{n})\}_n \to {\Z}/{p}
\]
is a perfect pairing of the ind-abelian and pro-abelian finite groups.
\end{lem}
\begin{proof}
The finiteness follows from \lemref{lem:Finite-0}.
For perfectness, it suffices to prove that without using the pro-systems,
the pairing in question is a perfect pairing
of finite abelian groups if we replace $n$ by $pn$. We shall show the latter.
For an $\F_p$-vector space $V$, we let $V^\star = \Hom_{\F_p}(V, \F_p)$.

Using the definitions of various pairings of sheaves and complexes of
sheaves, we have a commutative diagram of long exact sequences
\begin{equation}\label{eqn:PP-pro-1-0}
\xymatrix@C.4pc{
\cdots \ar[r] & H^{i-1}_\etl(X, Z_1\Omega^r_X(pnD)) \ar[r] \ar[d] &
H^{i-1}_\etl(X, \Omega^r_X(pnD)) \ar[r] \ar[d] &
\H^{i}_\etl(X, W_1\sF^{r, \bullet}_{pn}) \ar[r] \ar[d] &
 \cdots \\
\cdots \ar[r] & 
H^{d+1-i}_\etl(X, \frac{\Omega^{d-r}_X(-pnD)}{d\Omega^{d-r-1}_X(-pnD)})^\star
\ar[r] &
H^{d+1-i}_\etl(X, \Omega^{d-r}_X(-pnD))^\star \ar[r] &
\H^{d+1-i}_\etl(X, W_1\sG^{d-r, \bullet}_{pn})^\star \ar[r] &
\cdots .}
\end{equation}

The exactness of the bottom row follows from \lemref{lem:Finite-0} and
\cite[Lemma~7.10]{Gupta-Krishna-REC}.
Using \lemref{lem:LFPP} and Grothendieck duality for the structure map
$X \to \Spec(\F_p)$, the map
\[
H^{i}_\etl(X, \psi_*(Z_1\Omega^r_X(pnD))) \to
H^{d-i}_\etl(X, 
\psi_*(\frac{\Omega^{d-r}_X(-pnD)}{d\Omega^{d-r-1}_X(-pnD)}))^\star
\]
is an isomorphism. 
Since $\psi_*$ is exact, it follows that the map
\begin{equation}\label{eqn:PP-pro-1-1}
H^{i}_\etl(X, Z_1\Omega^r_X(pnD)) \to 
H^{d-i}_\etl(X, \frac{\Omega^{d-r}_X(-pnD)}{d\Omega^{d-r-1}_X(-pnD)})^\star
\end{equation}
is an isomorphism.
By the same reason, the map
\begin{equation}\label{eqn:PP-pro-1-2}
H^{i}_\etl(X, \Omega^r_X(pnD)) \to H^{d-i}_\etl(X, \Omega^{d-r}_X(-pnD))^\star 
\end{equation}
is an isomorphism.
Using ~\eqref{eqn:PP-pro-1-0}, ~\eqref{eqn:PP-pro-1-1} and
~\eqref{eqn:PP-pro-1-2}, we conclude that the right vertical arrow 
in ~\eqref{eqn:PP-pro-1-0} is an isomorphism. 
An identical argument shows that the
map $\H^{d+1-i}_\etl(X, W_1\sG^{d-r, \bullet}_{pn}) \to 
\H^{i}_\etl(X, W_1\sF^{r, \bullet}_{pn})^\star$
is an isomorphism. This finishes 
the proof of the perfectness of the 
pairing of the hypercohomology groups.
\end{proof}

We can now prove the main duality theorem of this paper.

\begin{thm}\label{thm:Duality-main}
Let $k$ be a finite field and $X$ a smooth and projective scheme
of pure dimension $d \ge 1$ over $k$. Let $D \subset X$ be an effective
Cartier divisor with complement $U$.
Let $m \ge 1$ and $i, r \ge 0$ be integers. Then 
\[
\{\H^i_\etl(X, W_m\sF^{r, \bullet}_{n})\}_n \times 
\{\H^{d+1-i}_\etl(X, W_m\sG^{d-r, \bullet}_{n})\}_n \to {\Z}/{p^m}
\]
is a semi-perfect pairing of ind-abelian and pro-abelian groups.
\end{thm}
\begin{proof}
We have shown already (see ~\eqref{eqn:Pair-4-7}) that the pairing is 
continuous after taking limits. We need to show that the maps
$\theta_m$ (see ~\eqref{eqn:Pair-5-0}) and $\theta'_m$ (see
~\eqref{eqn:Pair-5-2}) are isomorphisms of abelian groups.
We shall prove this by induction on $m \ge 1$.

We first assume $m = 1$. Then \lemref{lem:PP-pro-1} implies that the
map
\[
\theta_1 \colon 
\{\H^{d+1-i}_\etl(X, W_1\sG^{d-r, \bullet}_{n})\}_n \to
\{\H^i_\etl(X, W_1\sF^{r, \bullet}_{n})^\vee\}_n
\]
is an isomorphism of pro-abelian groups.
Taking the limit and using ~\eqref{eqn:Pair-4-4}, we get an isomorphism
\[
\theta_1 \colon F^{d+1-i}_{1,r}(U) \xrightarrow{\cong} 
{\varprojlim}_n \H^i_\etl(X, W_1\sF^{r, \bullet}_{n})^\vee.
\]
But the term on the right is same as $H^i_\etl(U, W_1\Omega^r_{U, \log})^\vee$
by  ~\eqref{eqn:Pair-4-3} and \lemref{lem:lim-colim}.
Lemma~\ref{lem:PP-pro-1} also implies that $\theta'_1$ is an isomorphism.
This proves $m =1$ case of the theorem.

We now assume $m \ge 2$ and recall the definitions of 
$F^j_{m,r}(n)$ and $F^j_{m,r}(U)$ from
\S~\ref{sec:Exact}. We consider the commutative diagram
\begin{equation}\label{eqn::Duality-main-1}
\xymatrix@C.8pc{
\cdots \ar[r] & F^{d+1-i}_{m-1,r}(U) \ar[r] \ar[d]_-{\theta_{m-1}} &  
F^{d+1-i}_{m,r}(U) \ar[r] \ar[d]^-{\theta_m} &  
F^{d+1-i}_{1,r}(U) \ar[r] \ar[d]^-{\theta_1} &  \cdots \\
\cdots \ar[r] & H^i_\etl(U, W_{m-1}\Omega^r_{U, \log})^\vee \ar[r]^-{R^\vee} &
H^i_\etl(U, W_{m}\Omega^r_{U, \log})^\vee \ar[r]^-{({\un{p}}^{m-1})^\vee} &
H^i_\etl(U, W_{1}\Omega^r_{U, \log})^\vee \ar[r] & \cdots.}
\end{equation}

The top row is exact by \lemref{lem:Lim-exact}.
The bottom row is exact by \lemref{lem:log-HW} and  
\cite[Lemma~7.10]{Gupta-Krishna-REC}.
The maps $\theta_i$ are isomorphisms for $i \le m-1$ by induction.
We conclude that $\theta_m$ is also an isomorphism.
An identical argument, where we apply \lemref{lem:long-exact-lim}
instead of \lemref{lem:Lim-exact}, shows that $\theta'_m$ is an
isomorphism. This finishes the proof.
\end{proof}

\begin{remk}\label{remk:Perf-semi-perf}
In \cite[Theorem~4.1.4]{JSZ} and \cite[Theorem~3.4.2]{Zhau}, 
a pairing using a logarithmic version of the pro-system 
$\{H^{d+1-r}_\etl(X, W_m\Omega^{d-r}_{(X,nD), \log})\}_n$ is 
given under the assumption that
$D_\red$ is a simple normal crossing divisor. The authors in these
papers say that their pairing is perfect. Although they do not explain 
their interpretations of perfectness,
what they actually prove is the semi-perfectness in the sense of
Definition~\ref{defn:PF-lim-colim}, according to
our understanding.
\end{remk}

\begin{cor}\label{cor:Perfect**}
The pairing of \thmref{thm:Duality-main} is perfect if
$D_\red$ is a simple normal crossing divisor, $i =1$,  $r =0$ and one of the 
conditions $\{d \neq 2, k \neq \F_2\}$ holds.
\end{cor}
\begin{proof}
Combine \thmref{thm:Duality-main}, \lemref{lem:Semi-perfect*}, 
\corref{cor:RS-K-comp}, \propref{prop:Milnor-iso},
 \cite[Theorem~9.1]{Kato-Saito} and the 
equivalence of ~\eqref{eqn:Pair-4-7} and  ~\eqref{eqn:Pair-4-8}.
\end{proof}

We let $C^{\etl}_{KS}(X,D;m) = H^d_\etl(X, \ov{{\wh{\sK}^M_{d,(X,D)}}/{p^m}})$.
We shall study the following special case in the next 
section.

\begin{cor}\label{cor:Duality-main-d}
Under the assumptions of \thmref{thm:Duality-main}, the 
map
\[
\theta_m \colon {\varprojlim}_n C^{\etl}_{KS}(X,nD;m) \to
{\pi^{\ab}_1(U)}/{p^m}
\]
is a bijective and continuous homomorphism between topological abelian groups.
This is an isomorphism of topological groups under the assumptions of
\corref{cor:Perfect**}.
\end{cor}

\subsection{The comparison theorem}\label{sec:Comp}
We shall continue with the assumptions of \thmref{thm:Duality-main}.
We fix an integer $m \ge 1$. Let $\pi^{\ab}_1(U)$ be the abelianized 
{\'e}tale fundamental group of $U$ and $\pi^\adiv_1(X,D)$ the 
co-1-skeleton {\'e}tale fundamental group of $X$ with modulus $D$,
introduced in \cite[Definition~7.5]{Gupta-Krishna-REC}. The latter
characterizes finite abelian covers of $U$ whose ramifications are
bounded by $D$ at each of its generic point, where the bound is
given by means of Matsuda's Artin conductor.
There is a natural surjection $\pi^{\ab}_1(U) \surj \pi^\adiv_1(X,D)$.

Let $C_{KS}(X,D) = H^d_\nis(X, \sK^M_{d, (X,D)})$ and
$C_{KS}(X,D; m) = H^d_\nis(X, {\sK^M_{d, (X,D)}}/{p^m}) \cong
{C_{KS}(X,D)}/{p^m}$. By \propref{prop:Milnor-iso}, we have
\begin{equation}\label{eqn:KS-NiS}
C_{KS}(X,D; m) \xrightarrow{\cong} H^d_\nis(X, \ov{{\wh{\sK}^M_{d,(X,D)}}/{p^m}}).
\end{equation}
We let $\wt{C}_{U/X} = \varprojlim_n C_{KS}(X,D)$.
The canonical map ${\wt{C}_{U/X}}/{p^m} \to \varprojlim_n C_{KS}(X,D; m)$
is an isomorphism by \cite[Lemma~5.9]{Gupta-Krishna-BF}.
The groups $C_{KS}(X,D)$ and $C_{KS}(X,D; m)$
have the discrete topology while $\wt{C}_{U/X}$ and
${\wt{C}_{U/X}}/{p^m}$ have the inverse limit topology.
The groups $\pi^{\ab}_1(U)$ and $\pi^\adiv_1(X,D)$ have the profinite topology.
By \cite[Theorem~1.2]{Gupta-Krishna-REC}, there is a commutative diagram
of continuous homomorphisms of topological abelian groups
\begin{equation}\label{eqn:REC}
\xymatrix@C.8pc{
\wt{C}_{U/X} \ar[r]^-{\rho_{U/X}} \ar@{->>}[d] & \pi^{\ab}_1(U) \ar@{->>}[d] \\
C(X,D) \ar[r]^-{\rho_{X|D}} & \pi^{\adiv}_1(X,D).}
\end{equation}
The horizontal arrows are injective with dense images.
They become isomorphisms after tensoring with
${\Z}/{p^m}$ by \cite[Corollary~5.15]{Gupta-Krishna-BF}.

We let $\wt{C}^{\etl}_{U/X}(m) = {\varprojlim}_n C^{\etl}_{KS}(X,nD;m)$.
By \corref{cor:Duality-main-d}, we have a
bijective and continuous homomorphism between topological abelian groups
$\rho^{\etl}_{U/X} \colon \wt{C}^{\etl}_{U/X}(m) \rightarrow
{\pi^{\ab}_1(U)}/{p^m} \cong H^1_{\etl}(U, {\Z}/{p^m})^\star$.
We therefore have a diagram
\begin{equation}\label{eqn:REC-0}
\xymatrix@C.8pc{
{\wt{C}_{U/X}}/{p^m} \ar[dr]^-{\rho_{U/X}} \ar[d]_-{\eta} & \\
\wt{C}^{\etl}_{U/X}(m) \ar[r]^-{\rho^{\etl}_{U/X}} & {\pi^{\ab}_1(U)}/{p^m}}
\end{equation}
of continuous homomorphisms,
where $\eta$ is the change of topology homomorphism.
We wish to prove the following.

\begin{prop}\label{prop:Nis-et-rec}
The diagram ~\eqref{eqn:REC-0} is commutative.
\end{prop}
\begin{proof}
It follows from \cite[Theorem~1.2]{Gupta-Krishna-BF} that
${\rho_{U/X}}$ is an isomorphism of profinite groups.
Since the image of the composite map
$\sZ_0(U) \xrightarrow{\cyc_{U/X}} {\wt{C}_{U/X}}/{p^m} \xrightarrow{\rho_{U/X}} 
{\pi^{\ab}_1(U)}/{p^m}$ is dense by the generalized Chebotarev density theorem,
it follows that the image of $\cyc_{U/X}$ is also dense in
${\wt{C}_{U/X}}/{p^m}$. Since ${\pi^{\ab}_1(U)}/{p^m}$ is Hausdorff, it suffices
to show that $\rho_{U/X} \circ \cyc_{U/X} = \rho^{\etl}_{U/X} \circ \eta
\circ \cyc_{U/X} $.
Equivalently, we have to show that for every closed point $x \in U$, 
one has that $(\rho^{\etl}_{U/X} \circ \eta \circ \cyc_{U/X})([x])$ 
is the image of the 
Frobenius element under the map $\Gal({\ov{k}}/{k(x)}) \to 
{\pi^{\ab}_1(U)}/{p^m}$.  
But this is well known (e.g., see \cite[Theorem~3.4.1]{Kerz-Zhau}). 
\end{proof}

\subsection{A new filtration of $H^1_\etl(U, {\Q_p}/{\Z_p})$}
\label{sec:Filt-et}
We keep the assumptions of \thmref{thm:Duality-main}.
By ~\eqref{eqn:Pair-4-4} and \thmref{thm:Duality-main}, we 
have the isomorphism of abelian groups
\[
\theta'_m \colon H^1(U, {\Z}/{p^m}) \xrightarrow{\cong}{\varinjlim}_n C^{\etl}_{KS}(X,nD;m)^{\vee},
\]
where $H^1(U, {\Z}/{p^m})$ denote the \'etale cohomology 
$ H_{\etl}^1(U, {\Z}/{p^m})$. 
 We let 
\[
\Fil^{\etl}_D H^1(U, {\Z}/{p^m}) = 
(\theta'_m)^{-1}({\rm Image}(C^{\etl}_{KS}(X,nD;m)^\vee \to
{\varinjlim}_n C^{\etl}_{KS}(X,nD;m)^{\vee})).
\]
We set
\begin{equation}\label{eqn:Filt-et-0}
\Fil^{\etl}_D H^1(U, {\Q_p}/{\Z_p}) = 
{\varinjlim}_m \Fil^{\etl}_D H^1(U, {\Z}/{p^m}).
\end{equation}
It follows that $\{\Fil^{\etl}_{nD} H^1(U, {\Q_p}/{\Z_p})\}_n$ defines
an increasing filtration of $H^1(U, {\Q_p}/{\Z_p})$. This is an
{\'e}tale version of the filtration $\Fil_{D} H^1(U, {\Q_p}/{\Z_p})$
defined in \cite[definition~7.12]{Gupta-Krishna-REC}.
This new filtration is clearly exhaustive.
We do not know if $\Fil_{D} H^1(U, {\Q_p}/{\Z_p})$ and
$\Fil^{\etl}_D H^1(U, {\Q_p}/{\Z_p})$ are comparable in general.
We can however prove the following.

\begin{thm}\label{thm:Comp-Main}
Let $k$ be a finite field and $X$ a smooth and projective scheme
of pure dimension $d \ge 1$ over $k$. Let $D \subset X$ be an effective
Cartier divisor with complement $U$ such that $D_\red$ is a simple
normal crossing divisor.
Assume that either $d \neq 2$ or $k \neq \F_2$.
Then one has 
\[
\Fil^{\etl}_D H^1(U, {\Z}/{p^m}) \subseteq
\Fil_D H^1(U, {\Z}/{p^m}) \ \ {\rm and} \ \
\Fil^{\etl}_D H^1(U, {\Q_p}/{\Z_p}) \subseteq 
\Fil_D H^1(U, {\Q_p}/{\Z_p})
\]
as subgroups of $H^1(U, {\Z}/{p^m})$ and 
$ H^1(U, {\Q_p}/{\Z_p})$, respectively.
These inclusions are equalities if $D_\red$ is regular.
\end{thm}
\begin{proof}
The claim about the two inclusions follows directly from the
definitions of the filtrations in view of \propref{prop:Nis-et-rec}
and \corref{cor:Perfect**}. 

Moreover, \corref{cor:Duality-main-d}
yields that $\rho^{\etl}_{U/X}$ is an isomorphism of
profinite topological groups. Hence, there exists a
unique quotient $\pi^{\etl}_1(X,nD;m)$ of ${\pi^{\ab}_1(U)}/{p^m}$ such that
the diagram
\begin{equation}\label{eqn:REG-*0}
\xymatrix@C.8pc{
\wt{C}^{\etl}_{U/X}(m) \ar[r]^-{\rho^{\etl}_{U/X}} \ar@{->>}[d] 
& {\pi^{\ab}_1(U)}/{p^m} \ar@{->>}[d] \\
\pi^{\etl}_1(X,nD;m) \ar[r]^-{\rho^{\etl}_{X|nD}} & 
{\pi^{\adiv}_1(X,nD)}/{p^m}}
\end{equation}
commutes and the horizontal arrows are isomorphisms of topological groups.
One knows that $\Fil_{nD} H^1(U, {\Z}/{p^m}) =
({\pi^{\adiv}_1(X,nD)}/{p^m})^\vee$.  By \thmref{thm:Comparison-3}, 
it is easy to see that $\Fil^{\etl}_{nD} H^1(U, {\Z}/{p^m})
\cong \pi^{\etl}_1(X,nD;m)^\vee$ when $D_\red$ is regular. 

More precisely,  if $D_\red$ is a simple
normal crossing divisor and 
either $d \neq 2$ or $k \neq \F_2$, we have the following
 diagram. 
\begin{equation}\label{eqn:REG-*0.5}
\xymatrix@C.8pc{
 H^1U, {\Z}/{p^m})  \ar[r]^-{\theta'}_-{\cong} &
 {\varinjlim}_n C^{\etl}_{KS}(X,nD;m)^{\vee}
 \ar[r]^-{\cong} &   ({\varprojlim}_n  C^{\etl}_{KS}(X,nD;m))^{\vee}\\
\Fil^{\etl}_D H^1(U, {\Z}/{p^m}) \ar@{_{(}->}[d] \ar@{^{(}->}[u]
\ar@{.>}[r] 
& C^{\etl}_{KS}(X,nD;m)^{\vee} \ar[u] \ar[d]& \\
\Fil_D H^1(U, {\Z}/{p^m})  \ar[r]^-{\cong}&C_{KS}(X,nD;m)^{\vee}. }
\end{equation}
Assume now that $D_{\red}$ is regular. By  \thmref{thm:Comparison-3}, 
it then follows that the transition maps in the pro-system 
$\{C^{\etl}_{KS}(X,nD;m))\}_n $ are surjective. In particular, the 
right top vertical arrow is injective and hence the middle 
horizontal dotted arrow is a honest arrow such that the top  square
 commutes. It follows from \propref{prop:Nis-et-rec}, that the 
 bottom square commutes as well. 
 By the definition of $\Fil^{\etl}_D H^1(U, {\Z}/{p^m})$, 
it is clear that  the middle 
horizontal arrow is now surjective. 
To prove its injectivity, it suffices to show
that the right bottom vertical arrow is injective. But this follows from  \thmref{thm:Comparison-3}. 
\end{proof}

\section{Reciprocity theorem for $C_{KS}(X|D)$}\label{sec:REC*}
In this section, we shall prove the reciprocity theorem for the idele
class group $C_{KS}(X|D)$. Before going into this, we recall the 
definition of some filtrations of $H^1_\etl(K, {\Q}/{\Z})$ for
a Henselian discrete valuation field $K$.

\subsection{The Brylinski-Kato and Matsuda filtrations}
\label{sec:BKM}
Let $K$ be a Henselian discrete valuation field of 
characteristic $p > 0$ with ring of integers
$\sO_K$, maximal ideal $\fm_K \neq 0$ and residue field $\ff$.
We let $H^q(A) = H^q_\etl(A, {\Q}/{\Z}(q-1))$ for any commutative ring $A$.
The generalized Artin-Schreier sequence gives rise to an exact sequence
\begin{equation}\label{eqn:DRC-3}
0 \to {\Z}/{p^r} \to W_r(K) \xrightarrow{1 - F} W_r(K)
\xrightarrow{\partial} H^1_{\et}(K, {\Z}/{p^r}) \to 0,
\end{equation}
where $F((a_{r-1}, \ldots , a_{0})) = (a^p_{r-1}, \ldots , a^p_{0})$.
We write the Witt vector $(a_{r-1}, \ldots , a_{0})$ as $\un{a}$ in short.
Let $v_K \colon K^{\times} \to \Z$ be the normalized valuation.
Let $\delta_r$ denote the composite map
$W_r(K) \xrightarrow{\partial} H^1_{\et}(K, {\Z}/{p^r}) \inj H^1(K)$.
Then one knows that $\delta_r = \delta_{r+1} \circ V$.
For an integer $m \ge 1$, let
${\rm ord}_p(m)$ denote the $p$-adic order of $m$ and
let $r' = \min(r, {\rm ord}_p(m))$. We let ${\rm ord}_p(0) = - \infty$.

For $m \ge 0$, we let
\begin{equation}\label{eqn:DRC-4}
\Fil^{\bk}_m W_r(K) = \{\un{a}| p^{i}v_K(a_i) \ge - m\}; \ \ {\rm and}
\end{equation}
\begin{equation}\label{eqn:DRC-5}
\Fil^{\ms}_m W_r(K) = \Fil^{\bk}_{m-1} W_r(K) + 
V^{r-r'}(\Fil^{\bk}_{m} W_{r'}(K)).
\end{equation}
We have $\Fil^{\bk}_0 W_r(K) = \Fil^{\ms}_0 W_r(K) = W_r(\sO_K)$.
We let $\Fil^{\bk}_{-1} W_r(K) = \Fil^{\ms}_{-1} W_r(K) = 0$.

For $m \geq 1$, we let
\begin{equation}\label{eqn:DRC-6}
\Fil^{\bk}_{m-1} H^1(K) = H^1(K)\{p'\} \bigoplus {\underset{r \ge 1}\bigcup}
\delta_r(\Fil^{\bk}_{m-1} W_r(K)) \ \ \mbox{and}
\end{equation}
\[
\Fil^{\ms}_m H^1(K) = H^1(K)\{p'\} \bigoplus {\underset{r \ge 1}\bigcup}
\delta_r(\Fil^{\ms}_m W_r(K)).
\]
Moreover, we let 
$\Fil^{\ms}_0 H^1(K) = \Fil^{\bk}_{-1} H^1(K) = H^1(\sO_K)$,
 the subgroup
of unramified characters.
The filtrations $\Fil^{\bk}_\bullet H^1(K)$ and $\Fil^{\ms}_\bullet H^1(K)$
are due to Brylinski-Kato \cite{Kato-89} and Matsuda \cite{Matsuda},
respectively. We refer to \cite[Theorem~6.1]{Gupta-Krishna-REC} for the
following.

\begin{thm}\label{thm:Fil-main}
The two filtrations defined above satisfy the following relations.
\begin{enumerate}
\item
$H^1(K) = {\underset{m \ge 0}\bigcup} \Fil^{\bk}_m H^1(K) 
= {\underset{m \ge 0}\bigcup} \Fil^{\ms}_m H^1(K) $.
\item
$\Fil^{\ms}_m H^1(K) \subset  \Fil^{\bk}_m H^1(K) \subset \Fil^{\ms}_{m+1} H^1(K)$
for all $m \ge -1$. 
\item If $m \geq 1$ such that ${\rm ord}_p(m)=0$, then 
$\Fil^{\bk}_{m-1} H^1(K) = \Fil^{\ms}_m H^1(K)$. In particular, 
$\Fil^{\bk}_0 H^1(K) = \Fil^{\ms}_1 H^1(K)$, which  is the subgroup of
tamely ramified characters.
\end{enumerate}
\end{thm}

For integers $m \ge 0$ and $r \ge 1$, let $U_m K^M_r(K)$ be
the subgroup $\{1 + \fm^m_K, K^{\times}, \ldots , K^{\times}\}$ of $K^M_r(K)$.
We let $U'_mK^M_r(K)$ be the subgroup $\{1 + \fm^m_K, \sO^{\times}_K, \ldots ,
\sO^{\times}_K\}$ of $K^M_r(K)$. 
It follows from \cite[Lemma~6.2]{Gupta-Krishna-REC} that 
$U_{m+1}K^M_r(K) \subseteq U'_mK^M_r(K) \subseteq U_{m}K^M_r(K)$ for every
integer $m \ge 0$.
If $K$ is a $d$-dimensional Henselian local field 
(see \cite[\S~5.1]{Gupta-Krishna-REC}), there is a pairing
\begin{equation}\label{eqn:Kato-pair}
\{,\} \colon K^M_d(K) \times H^1(K) \to H^{d+1}(K) \cong {\Q}/{\Z}.
\end{equation}

The following result is due to Kato and Matsuda (when $p \neq 2$).
We refer to \cite[Theorem~6.3]{Gupta-Krishna-REC} for a proof.

\begin{thm}\label{thm:Filt-Milnor-*}
Let $\chi \in H^1(K)$ be a character. Then the following hold.
\begin{enumerate}
\item
For every integer $m \ge 0$, we have that 
$\chi \in \Fil^{\bk}_m H^1(K)$ if and only if $\{\alpha, \chi\} = 0$
for all $\alpha \in U_{m+1} K^M_d(K)$.
\item
For every integer $m \ge 1$, we have that $\chi \in \Fil^{\ms}_m H^1(K)$ if and 
only if $\{\alpha, \chi\} = 0$ for all $\alpha \in U'_{m} K^M_d(K)$.
\end{enumerate}
\end{thm}

We have the inclusions $K \inj K^{sh} \inj \ov{K}$, where $\ov{K}$ is 
a fixed separable closure of $K$ and $K^{sh}$ is the strict Henselization of 
$K$. We shall use the following key result in the proof of our reciprocity
theorem.

\begin{prop}\label{prop:Fil-SH-BK}
For $m \ge 0$, the canonical square
\[
\xymatrix@C.8pc{
\Fil^{\bk}_m H^1(K) \ar[r] \ar[d] & H^1(K) \ar[d] \\
\Fil^{\bk}_m H^1(K^{sh}) \ar[r]  & H^1(K^{sh})}
\]
is Cartesian.
\end{prop}
\begin{proof}
Since $\Fil^{\bk}_{\bullet} H^1(K)$ is an exhaustive
filtration of $H^1(K)$ by \thmref{thm:Fil-main}(1), it 
suffices to show that for every $m \ge 1$, the square
\[
\xymatrix@C.8pc{
\Fil^{\bk}_{m-1} H^1(K) \ar[r] \ar[d] & \Fil^{\bk}_{m} H^1(K) \ar[d] \\
\Fil^{\bk}_{m-1} H^1(K^{sh}) \ar[r]  & \Fil^{\bk}_{m} H^1(K^{sh})}
\]
is Cartesian. Equivalently, it suffices to show that for every $m \ge 1$, the
map 
\[
\phi^*_m \colon \gr^{\bk}_m H^1(K) \to \gr^{\bk}_m H^1(K^{sh}),
\]
induced by the inclusion $\phi \colon K \inj K^{sh}$, is injective.

We fix $m \ge 1$. By \cite[Corollary~5.2]{Kato-89}, 
there exists an injective (non-canonical)
homomorphism 
${\rm rsw}_{\pi_K} \colon \gr^{\bk}_m H^1(K) \inj \Omega^1_{\ff}  \oplus \ff$. 
By Theorem~5.1 of loc. cit., this refined swan conductor 
${\rm rsw}_{\pi_K} $
depends only on the choice of a uniformizer $\pi_K$ of $K$. 
Since $\sO^{sh}_K$ is unramified over $\sO_K$, we can choose
$\pi_K$ to be a uniformizer of $K^{sh}$ as well. It therefore 
follows that for all $m\geq 1$, the diagram  
\[
\xymatrix@C2pc{
\gr^{\bk}_m H^1(K) \ar@{^{(}->}[r]^-{{\rm rsw}_{\pi_K}} \ar[d] &
\Omega^1_{\ff}  \oplus \ff \ar[d] \\
\gr^{\bk}_m H^1(K^{sh}) \ar@{^{(}->}[r]^-{{\rm rsw}_{\pi_K}}  &
\Omega^1_{\ov{\ff}}  \oplus \ov{\ff}}
\]
is commutative, where $\ov{\ff}$ is a separable closure of 
$\ff$.
Since the horizontal arrows in the above diagram are injective, it suffices to 
show 
that the natural map $\Omega^1_{\ff}  \to \Omega^1_{\ov{\ff}}$ is injective. 
But this is clear.
\end{proof}

\subsection{Logarithmic fundamental group with modulus}
\label{sec:LFGM}
Let $k$ be a finite field of characteristic $p$ and $X$ an integral
projective scheme over $k$ of dimension $d \ge 1$. Let $D \subset X$
be an effective Cartier divisor with complement $U$. Let $K = k(\eta)$ denote 
the function field of $X$. We let $C = D_\red$. 
We fix a separable closure $\ov{K}$ of $K$
and let $G_K$ denote the absolute Galois group of $K$.
Recall the following notations from \cite[\S~3.3]{Kato-Saito} or
\cite[\S~2.3]{Gupta-Krishna-REC}. 

Assume that $X$ is normal.
Let $\lambda$ be a generic point of $D$.
Let $K_\lambda$ denote the Henselization of $K$ at $\lambda$.
Let $P = (p_0, \ldots , p_{d-2}, \lambda, \eta)$ be a Parshin
chain on $(U \subset X)$. 
Let $V \subset K$ be a $d$-DV which dominates $P$. 
Let $V = V_0 \subset \cdots \subset V_{d-2}
\subset V_{d-1} \subset V_d = K$ be the chain of valuation rings
in $K$ induced by $V$. Since $X$ is normal, it is easy to check
that for any such chain, one must have $V_{d-1} = \sO_{X,\lambda}$.
Let $V'$ be the image of $V$ in $k(\lambda)$. Let $\wt{V}_{d-1}$
be the unique Henselian discrete valuation ring 
having an ind-{\'e}tale local homomorphism 
 $V_{d-1} \to \wt{V}_{d-1}$ such that its residue field $E_{d-1}$
is the quotient field of $(V')^h$. Then $V^h$ is the inverse image of
$(V')^h$ under the quotient map $\wt{V}_{d-1} \surj E_{d-1}$.
It follows that its function field $Q(V^h)$ is a $d$-dimensional
Henselian discrete valuation field whose ring of integers is 
$\wt{V}_{d-1}$ (see \cite[\S~3.7.2]{Kato-Saito}).
It then follows that
there are canonical inclusions of discrete valuation rings 
\begin{equation}\label{eqn:Fil-SH-3}
\sO_{X,\lambda}  \inj \wt{V}_{d-1} \inj \sO^{sh}_{X,\lambda}.
\end{equation}
Moreover, we have (see the proof of \cite[Proposition~3.3]{Kato-Saito})
\begin{equation}\label{eqn:Fil-SH-5}
\sO^h_{X,P'} \cong {\underset{V \in \sV(P)}\prod} \wt{V}_{d-1},
\end{equation}
where $\sV(P)$ is the set of $d$-DV's in $K$ which dominate $P$.
As an immediate consequence of \propref{prop:Fil-SH-BK},
we therefore get the following.

\begin{cor}\label{cor:Fil-SH-4}
For every $m \ge 0$, the square
\[
\xymatrix@C.8pc{
\Fil^{\bk}_m H^1(K_\lambda) \ar[r] \ar[d] & H^1(K_\lambda) \ar[d] \\
\Fil^{\bk}_m H^1(Q(V^h)) \ar[r]  & H^1(Q(V^h))}
\]
is Cartesian.
\end{cor}

Let $\Irr_C$ denote the set of all generic points of $C$ and
let $C_\lambda$ denote the closure of an element $\lambda \in \Irr_C$.
We write $D = {\underset{\lambda \in \Irr_C}\sum} n_\lambda C_\lambda$.
We can also write $D = {\underset{x \in X^{(1)}}\sum}  n_x \ov{\{x\}}$, where $n_x =0$ for all $x\in U$.
We allow $D$ to be empty in which case we 
write $D = 0$.
For any $x \in X^{(1)} \cap C$, we let $\wh{K}_x$ denote the quotient field of 
the $\fm_x$-adic completion $\wh{\sO_{X,x}}$ of $\sO_{X,x}$. 
Let $\sO^{sh}_{X,x}$ denote the strict Henselization of $\sO_{X,x}$
and let $K^{sh}_x$ denote its quotient field.
Then it is clear from the definitions that there are
inclusions 
\begin{equation}\label{eqn:Field-incln}
K \inj K_x \inj K^{sh}_x \inj \ov{K} \ \mbox{and} \ \ K \inj K_x \inj \wh{K}_x. 
\end{equation}

\begin{defn}\label{defn:Fid_D}
Let $\Fil^{\bk}_D H^1(K)$ denote the subgroup of characters
$\chi \in H^1(K)$ such that for every $x \in X^{(1)}$, 
the image $\chi_{x}$ of
$\chi$ under the canonical surjection $H^1(K) \surj H^1(K_{x})$ 
lies in $\Fil^{\bk}_{n_{x}-1} H^1(K_{x})$.
It is easy to check that $\Fil^{\bk}_D H^1(K) \subset H^1(U)$.
We let $\Fil^{\bk}_D H^1_{\etl}(U, {\Z}/{m}) = 
H^1_{\etl}(U, {\Z}/{m}) \cap \Fil^{\bk}_D H^1(K)$.
\end{defn}

Recall that  $H^1(K)$ is a torsion abelian group. We consider it a
topological abelian group with discrete topology. In particular, 
all the subgroups $H^1(U)$ (e.g., $\Fil^{\bk}_D H^1(K)$) are 
also considered as discrete topological abelian groups. 
Recall from \cite[Definition~7.12]{Gupta-Krishna-REC} that
$\Fil_D H^1(K)$ is a subgroup of $H^1(U)$ which is
defined similar to $\Fil^{\bk}_D H^1(K)$, where we only
replace $\Fil^{\bk}_{n_{\lambda}-1} H^1(K_{\lambda})$ by
$\Fil^{\ms}_{n_{\lambda}} H^1(K_{\lambda})$.

\begin{defn}\label{defn:Fun-D-BK}
We define the quotient $\pi_1^{\abk}(X, D)$ of $\pi_1^{\ab}(U)$
 to be the Pontryagin dual of $\Fil^{\bk}_D H^1(K) \subset H^1(U)$, 
 i.e., 
 \[
 \pi_1^{\abk}(X, D) := \Hom_{\cont}(\Fil^{\bk}_D H^1(K), \Q/\Z) = 
\Hom(\Fil^{\bk}_D H^1(K), \Q/\Z).
 \]
\end{defn}
Since $H^1_{\etl}(U, {\Z}/{p^m}) = \ _{p^m}H^1(U, {\Q}/{\Z})$,
it follows that 
\[
{\pi_1^{\abk}(X, D)}/{p^m} \cong 
\Hom_{\cont}(\Fil^{\bk}_D H^1_{\etl}(U, {\Z}/{p^m}), {\Q}/{\Z}).
\]
Since $\Fil^{\bk}_D H^1(K)$ is a discrete topological group, it follows that
$\pi_1^{\abk}(X, D)$ is a profinite group. Moreover,  
\cite[Theorem~2.9.6]{Pro-fin} implies that 
\begin{equation}\label{eqn:dual-D-BK}
\pi_1^{\abk}(X, D)^\vee \cong \Fil^{\bk}_D H^1(K) \ \ {\rm and} \ \
({\pi_1^{\abk}(X, D)}/{p^m})^\vee \cong \Fil^{\bk}_D H^1_{\etl}(U, {\Z}/{p^m}).
\end{equation}  

\begin{lem} \label{lem:adiv-abk}
The quotient map $\pi^{\ab}_1(U) \to \pi_1^{\abk}(X, D)$ 
factors through $\pi_1^{\adiv}(X, D)$. 
\end{lem}
\begin{proof}
This is a straightforward consequence of \thmref{thm:Fil-main} and
\cite[Theorem~7.16]{Gupta-Krishna-REC} once we recall that
\begin{equation}\label{eqn:adiv-abk-0}
\pi_1^{\adiv}(X, D) = \Hom(\Fil_D H^1(K), \Q/\Z).
\end{equation}
\end{proof}

\begin{remk}\label{remk:Tannakian-defn}
Using \cite[\S~3]{Abbes-Saito}, one can mimic the
construction of \cite[\S~7]{Gupta-Krishna-REC} to show that 
$\pi_1^{\abk}(X, D)$ is isomorphic to the abelianization of the
automorphism group of the fiber functor of a certain Galois subcategory of
the category of finite {\'e}tale covers of $U$. 
Under this Tannakian interpretation, 
$\pi_1^{\abk}(X, D)$ characterizes the finite
{\'e}tale covers of $U$ whose ramifications are bounded at each generic
point of $D$ by means of Kato's Swan conductor.
\end{remk}

\vskip .3cm

\subsection{Reciprocity for $C_{KS}(X|D)$}\label{sec:REC-KS}
We shall continue with the setting of \S~\ref{sec:LFGM}.
Before we prove the reciprocity theorem for $C_{KS}(X|D)$, 
we recall the construction of the reciprocity map for
$C_{U/X}$ from \cite[\S~5.4]{Gupta-Krishna-REC}.
We let $P = (p_0, \ldots , p_s)$ be any Parshin chain on $X$ with
the condition that $p_s \in U$.
Let $X(P) = \ov{\{p_s\}}$ be the integral closed subscheme of $X$
and let $U(P) = U \cap X(P)$.
We let $V \subset k(p_s)$ be an $s$-DV dominating $P$.
Then $Q(V^h)$ is an
$s$-dimensional Henselian local field. By 
\cite[Proposition~5.1]{Gupta-Krishna-REC}, we therefore 
have 
the reciprocity map  
\begin{equation}\label{eqn:Kato-rec}
\rho_{Q(V^h)} \colon K^M_s(Q(V^h)) \to \Gal({{Q(V^h)}^{\ab}}/{Q(V^h)})
\cong \pi^{\ab}_1(\Spec(Q(V^h))).
\end{equation}
Taking the sum of these maps over $\sV(P)$ and using 
\cite[Lemma~5.10]{Gupta-Krishna-REC}, we get a reciprocity map
\begin{equation}\label{eqn:Kato-rec-0}
\rho_{k(P)} \colon K^M_s(k(P)) \to \pi^{\ab}_1(\Spec(k(P))) \to \pi^{\ab}_1(U), 
\end{equation}
where last map exists because $p_s \in U$. Taking sum over all Parshin 
chains on the pair $(U\subset X)$, we get a reciprocity map 
$\rho_{U/X} \colon I_{U/X} \to \pi^{\ab}_1(U)$.
By \cite[Theorem~5.13, Proposition~5.15]{Gupta-Krishna-REC},
this descends to a continuous homomorphism of topological groups
\begin{equation}\label{eqn:Kato-rec-4}
\rho_{U/X} \colon C_{U/X} \to 
\pi^{\ab}_1(U).
\end{equation}

Recall that $\pi_1^{\adiv}(X, D)_0$ is the kernel of the
composite map $\pi_1^{\adiv}(X, D) \surj \pi_1^{\ab}(X) \to \wh{\Z}$.
One defines $\pi_1^{\abk}(X, D)_0$ similarly.
One of the main results of this paper is the following.

\begin{thm}\label{thm:REC-RS-Main}
There is a continuous homomorphism 
\[
\rho'_{X|D} \colon C(X|D) \to \pi_1^{\abk}(X, D)
\]
with dense image such that the diagram
\begin{equation}\label{eqn:Rec-D-map-0}
\xymatrix@C.8pc{
{C}_{U/X} \ar[r]^-{\rho_{U/X}} \ar@{->>}[d]_-{p'_{X|D}} & \pi^{\ab}_1(U) 
\ar@{->>}[d]^-{q'_{X|D}} \\
C(X|D) \ar[r]^-{\rho'_{X|D}} & \pi^{\abk}_1(X,D)}
\end{equation}
is commutative. If $X$ is 
normal and $U$ is regular, then $\rho'_{X|D}$ induces an
isomorphism of finite groups
\[
\rho'_{X|D} \colon C(X|D)_0 \xrightarrow{\cong} \pi_1^{\abk}(X, D)_0.
\]
\end{thm}
\begin{proof}
We first show the existence of $\rho'_{X|D}$.
Its continuity and density of its image will then follow by 
the corresponding assertions for $\rho_{U/X}$, shown in 
\cite{Gupta-Krishna-REC}. 
In view of ~\eqref{eqn:Kato-rec-4}, we only have to show that
if $\chi$ is a character of $\pi_1^{\abk}(X, D)$, then the composite
$\chi \circ q'_{X|D} \circ \rho_{U/X}$ annihilates $\Ker(p'_{X|D})$.
By ~\eqref{eqn:Idele-D}, we only need to show that 
$\chi \circ q'_{X|D} \circ \rho_{U/X}$ annihilates 
the image of $\wh{K}^M_{d}(\sO^h_{X,P'}|I_D) \to C_{U/X}$, where $P$ is any
maximal Parshin chain on $(U \subset X)$.
But the proof of this is completely identical to that of
\cite[Theorem~8.1]{Gupta-Krishna-REC}, only difference being 
that we have to use part (1) of \thmref{thm:Filt-Milnor-*}
instead of part (2).

We now assume that $X$ is normal and $U$ is regular.
In this case, we can replace $C(X|D)$ by $C_{KS}(X|D)$ by
\thmref{thm:KZ-main}. We show that $\rho'_{X|D}$ is injective on all of 
$C_{KS}(X|D)$.
By Lemmas~\ref{lem:KS-GK-iso} and ~\ref{lem:adiv-abk}, there is a diagram
\begin{equation}\label{eqn:Rec-D-map-1}
\xymatrix@C.8pc{
\Fil^{\bk}_D H^1(K) \ar[r]^-{\rho'^\vee_{X|D}} \ar[d]_-{\alpha^\vee} & 
C_{KS}(X|D)^\vee \ar[d]^-{\beta^\vee} \\
\Fil_D H^1(K) \ar[r]^-{\rho^\vee_{X|D}} & C_{KS}(X,D)^\vee,}
\end{equation} 
whose vertical arrows are injective. It is clear from the construction of 
the reciprocity maps that this diagram is commutative.

To show that $\rho'_{X|D}$ is injective, it suffices to show that 
$\rho'^\vee_{X|D}$ is surjective (see \cite[Lemma~7.10]{Gupta-Krishna-REC}). 
We fix a character 
$\chi \in C_{KS}(X|D)^\vee$ and let $\wt{\chi} = \beta^\vee(\chi)$.
Since $\rho^\vee_{X|D}$ is surjective by 
\cite[Theorem~1.1]{Gupta-Krishna-BF}, we can find a character
$\chi' \in \Fil_D H^1(K)$ such that $\wt{\chi} = \rho^\vee_{X|D}(\chi')$. 
We need to show that $\chi' \in \Fil^{\bk}_D H^1(K)$.

We fix a point $x \in \Irr_C$ and let $\chi'_x$ be the 
image of $\chi'$ in
$H^1(K_x)$. We need to show that $\chi'_x \in \Fil^{\bk}_{n_x -1} H^1(K_x)$,
where $n_x$ is the multiplicity of $D$ at $x$.
By \corref{cor:Fil-SH-4}, it suffices to show that
for some maximal Parshin chain $P = (p_0, \ldots , p_{d-2}, x, \eta)$
on $(U \subset X)$
and $d$-DV $V \subset K$ dominating $P$, the image of $\chi'_x$ in 
$H^1(Q(V^h))$ lies in the subgroup $\Fil^{\bk}_{n_x -1} H^1(Q(V^h))$.
But this is proven by repeating the proof of 
\cite[Theorem~1.1]{Gupta-Krishna-BF} mutatis mutandis by using only
one modification. Namely, we need to use part (1) of 
\thmref{thm:Filt-Milnor-*} instead of part (2).

The surjectivity of $\rho'_{X|D}$ on the degree zero part follows because
the top horizontal arrow in the
commutative diagram 
\begin{equation}\label{eqn:Rec-mod-bk-2}
\xymatrix@C.8pc{  
C_{KS}(X,D)_0 \ar[r]^-{\rho_{X|D}} \ar@{->>}[d]_-{\beta} & \pi^{\adiv}_1(X,D)_0 
\ar@{->>}[d]^-{\alpha} \\
C_{KS}(X|D)_0 \ar[r]^-{\rho'_{X|D}} & \pi^{\abk}_1(X,D)_0}
\end{equation}
is surjective by \cite[Theorem~1.1]{Gupta-Krishna-BF}. 
The finiteness of $C_{KS}(X|D)_0$ follows because 
$\beta$ is surjective by \lemref{lem:KS-GK-iso} and
$C_{KS}(X,D)_0$ is finite by \cite[Theorem~4.8]{Gupta-Krishna-BF}.
This concludes the proof.
\end{proof}

Using \thmref{thm:REC-RS-Main} and \cite[Lemma~8.4]{Gupta-Krishna-REC},
we get the following.
\begin{cor}\label{cor:REC-RS-Main-fin}
Under the additional assumptions of \thmref{thm:REC-RS-Main}, the map
\[
\rho'_{X|D} \colon {C_{KS}(X|D)}/m \to {\pi^{\abk}_1(X,D)}/m
\]
is an isomorphism of finite groups for every integer $m \ge 1$.
\end{cor}

\subsection{Filtration of Kerz-Zhao}
\label{sec:Etale-bk}
Let $k$ be a finite field of characteristic $p$ and $X$ an integral regular
projective scheme over $k$ of dimension $d \ge 1$. Let $D \subset X$
be an effective Cartier divisor with complement $U$ such that
$D_\red$ is a simple normal crossing divisor. Let $K$ denote 
the function field of $X$.
By \cite[Theorems~1.1.5, 4.1.4]{JSZ}, there is an isomorphism
\begin{equation}\label{eqn:JSJ-0}
\lambda_m \colon H^1_{\etl}(U, {\Z}/{p^m}) \xrightarrow{\cong}
{\varinjlim}_n H^d_{\etl}(X, \ov{{\wh{\sK}^M_{d,X|nD}}/{p^m}})^\vee.
\end{equation}
Each group $H^d_{\etl}(X, \ov{{\wh{\sK}^M_{d,X|nD}}/{p^m}})$ is finite
by \cite[Theorem~3.3.1]{Kerz-Zhau} and \corref{cor:REC-RS-Main-fin}.
It follows that ~\eqref{eqn:JSJ-0} 
is an isomorphism of discrete torsion topological abelian groups.

Since $H^d_{\etl}(X, \ov{{\wh{\sK}^M_{d,X|D}}/{p^m}})^\vee
\inj {\varinjlim}_n H^d_{\etl}(X, \ov{{\wh{\sK}^M_{d,X|nD}}/{p^m}})^\vee$,
we can define
\[
\Fil^{\bk}_{D} H^1_\etl(U, {\Z}/{p^m}) :=
(\lambda_m)^{-1}(H^d_{\etl}(X, \ov{{\wh{\sK}^M_{d,X|D}}/{p^m}})^\vee).
\]
This filtration was defined by Kerz-Zhao \cite[Definition~3.3.7]{Kerz-Zhau}.
We let $\Fil^{\bk}_{D} H^1_\etl(U, {\Q_p}/{\Z_p}) = {\varinjlim}_m 
\Fil^{\bk}_{D} H^1_\etl(U, {\Z}/{p^m})$ and
$\Fil^{\bk}_{D, \etl} H^1(K) = H^1(K)\{p'\} \oplus 
\Fil^{\bk}_{D} H^1_\etl(U, {\Q_p}/{\Z_p})$.

Using \cite[Theorem~3.3.1]{Kerz-Zhau} and \corref{cor:REC-RS-Main-fin}, we
get the following logarithmic version of \thmref{thm:Comp-Main} (without assuming that
$D_\red$ is regular).
This identifies the filtration due to  Kerz-Zhao to
the one induced by the Brylinski-Kato filtration.

\begin{cor}\label{cor::REC-RS-Main-fin-0}
As subgroups of $H^1(K)$, one has 
\[
\Fil^{\bk}_{D, \etl} H^1(K) = \Fil^{\bk}_D H^1(K).
\]
\end{cor}

\vskip .3cm

\section{Counterexamples to Nisnevich descent}\label{sec:CE}
We shall now prove \thmref{thm:Main-5}.
Let $k$ be a finite field of characteristic $p$ and $X$ an integral regular
projective scheme of dimension two over $k$. 
Let $C \subset X$ be an integral regular curve with complement $U$.
Let $K = k(X)$ and $\ff = k(\lambda)$, where $\lambda$ is the
generic point of $C$. Let $K_\lambda$ be the Henselian discrete valuation field
with ring of integers $\sO^h_{X,\lambda}$ and residue field $\ff$.
Fix a positive integer $m_0 = p^rm'$, where $r \ge 1$ and $p \nmid m'$. 
Assume that $p \neq 2$.
Then Matsuda has shown (see the proof of \cite[Proposition~3.2.7]{Matsuda})
that there is an isomorphism
\begin{equation}\label{eqn:Artin-cond}
\eta \colon \frac{\Fil^{\ms}_{m_0} H^1(K_\lambda)}{\Fil^{\bk}_{m_0 -1} H^1(K_\lambda)}
\xrightarrow{\cong} B_r \Omega^1_{\ff},
\end{equation}
where $B_\bullet \Omega^1_{\ff}$ is an increasing filtration of
$\Omega^1_{\ff}$, recalled in the proof of \lemref{lem:BK-RS}.
Suppose that $B_1 \Omega^1_{\ff} = d(\ff) \subset \Omega^1_{\ff}$ is not zero.
Then $B_r\Omega^1_{\ff} \neq 0$ for every $r \ge 1$.
Hence, the left hand side of ~\eqref{eqn:Artin-cond} is not zero.
Since the restriction map $\delta_\lambda \colon H^1(K) \to H^1(K_\lambda)$
is surjective, it follows that we can find a continuous character
$\chi \in H^1(K)$ such that $\delta_\lambda(\chi) \in
\Fil^{\ms}_{m_0} H^1(K_\lambda) \setminus \Fil^{\bk}_{m_0 -1} H^1(K_\lambda)$.

We let $U' \subseteq U$ be the largest open subscheme where 
$\chi$ is unramified and let $C' = X\setminus U'$ with the reduced
closed subscheme structure.
Since $X$ is a surface, we can find a morphism $f \colon X' \to X$
which is a composition of a sequence monoidal transformations
such that the reduced closed subscheme $f^{-1}(C')$ is a simple
normal crossing divisor.
In particular, $E_0 \to C$ is an isomorphism if we let $E_0$ be
the strict transform of $C$. We let $E = f^{-1}(C')$ with reduced
structure. Note that there is a finite closed subset $T \subset C'$
such that $f^{-1}(X \setminus T) \to X \setminus T$ is an isomorphism.

We let $E = E_0 + E_1 + \cdots + E_s$, where each $E_i$ is integral. 
We let $\lambda_i$ be the generic point of $E_i$ so that $\lambda_0
= \lambda$.
Let ${\rm ar}_{\lambda_i} \colon H^1(K) \to \Z$ be the Artin conductor
(see \cite[Definition~3.2.5]{Matsuda}). 
\thmref{thm:Fil-main} implies that 
${\rm ar}_{\lambda_0}(\chi) = m_0$.
We let $m_i = {\rm ar}_{\lambda_i}(\chi)$ for $i \ge 1$
and define $D' = \stackrel{s}{\underset{i = 0}\sum} m_i E_i$.
It is then clear that $\chi \in \Fil_D H^1(K) \setminus \Fil^{\bk}_D H^1(K)$.
We have thus found a smooth projective integral surface $X'$ and an
effective Cartier divisor $D' \subset X'$ with the property that
$D'_\red$ is a simple normal crossing divisor and
$\frac{\Fil_{D'} H^1(K)}{\Fil^{\bk}_{D'} H^1(K)} \neq 0$ if $K$ is the function 
field of $X'$.   
It follows that $\Ker(\pi^{\adiv}_1(X', D') \surj \pi^{\abk}_1(X', D'))
\neq 0$. Equivalently,
$\Ker(\pi^{\adiv}_1(X', D')_0 \surj \pi^{\abk}_1(X', D')_0)
\neq 0$.

We now look at the the commutative diagram
\begin{equation}\label{eqn:Rec-cycle}
\xymatrix@C.8pc{
\CH_0(X'|D')_0 \ar[r]^-{\cyc_{X'|D'}} \ar[d] & 
\pi^{\adiv}_1(X', D')_0 \ar[d] \\
H^{4}_\sM(X'|D', \Z(2))_0 \ar[r]^-{\cyc'_{X'|D'}} &
\pi^{\abk}_1(X', D')_0.}
\end{equation}
The top horizontal arrow is an isomorphism by 
\cite[Theorems~1.2, 1.3]{Gupta-Krishna-BF} and the bottom
horizontal arrow is an isomorphism by \corref{cor:Main-4}.
We conclude that the left vertical arrow is surjective but not injective.

To complete the construction of a counterexample, what remains is to
find a pair $(X,C)$ such that $d(\ff) \subset \Omega^1_{\ff}$ is not zero.
But this is an easy exercise. For instance, take $X = \P^2_k$ and
$C  \subset \P^2_k$ a coordinate hyperplane. Then $\ff \cong k(t)$
is a purely transcendental extension of degree one and
$d(t)$ is a free generator of $\Omega^1_{\ff}$.
We remark that we had assumed above that $p \neq 2$, but this condition
can be removed using the proof of \cite[Theorem~6.3]{Gupta-Krishna-REC}.

\vskip .3cm

Let $(X',D')$ be as above and let $F = k(t)$ be a purely transcendental extension of degree one.
We let $\wt{X} = X'_{F}$ and $\wt{D} = D'_F$. Then we have a commutative diagram
\[
  \xymatrix@C.8pc{
    \CH_0(X'|D') \ar[r] \ar[d] &   \CH_0(\wt{X}|\wt{D}) \ar[d] \\
    H^{4}_\sM(X'|D', \Z(2)) \ar[r] & H^{4}_\sM(\wt{X}|\wt{D}, \Z(2)),}
\]
where the horizontal arrows are the flat pull-back maps. It easily follows from the proof of
\cite[Proposition~4.3]{KP-Doc} that the top horizontal arrow is injective. It follows that
the right vertical arrow is not injective. This shows the failure of Nisnevich descent for
the Chow groups with modulus over infinite fields too.

\vskip .4cm

\noindent\emph{Acknowledgements.}
Gupta was supported by the
SFB 1085 \emph{Higher Invariants} (Universit\"at Regensburg).

\end{document}